\documentclass[letterpaper, 11pt,  reqno]{amsart}

\usepackage{amsmath,amssymb,amscd,amsthm,amsxtra, esint, mathtools, xcolor}
\usepackage{latexsym}
\usepackage{skak}
\usepackage{accents, cancel}
\usepackage{bbm}
\usepackage{enumerate}

\usepackage[implicit=true]{hyperref}

\usepackage{stmaryrd}

\allowdisplaybreaks[2]

\newtheorem{theorem}{Theorem} [section]

\newtheorem{lemma}[theorem]{Lemma}
\newtheorem{proposition}[theorem]{Proposition}
\newtheorem{remark}[theorem]{Remark}

\newtheorem{definition}[theorem]{Definition}
\newtheorem{corollary}[theorem]{Corollary}

\DeclareMathOperator*{\supp}{supp}

\newcommand{\I}{\hspace{0.5mm}\text{I}\hspace{0.5mm}}

\newcommand{\noi}{\noindent}
\newcommand{\Z}{\mathbb{Z}}
\newcommand{\R}{\mathbb{R}}

\newcommand{\T}{\mathbb{T}}

\let\Im=\undefined\DeclareMathOperator*{\Im}{Im}

\let\P= \undefined
\newcommand{\P}{\mathbf{P}}   
\newcommand{\prob}{\mathbb{P}}  

\newcommand{\E}{\mathbb{E}}

\newcommand{\B}{\mathcal{B}}

\newcommand{\F}{\mathcal{F}}

\newcommand{\al}{\alpha}

\newcommand{\dl}{\delta}

\newcommand{\FL}{\mathcal{F}L} 

\newcommand{\vp}{\varphi}

\newcommand{\ep}{\varepsilon}

\newcommand{\g}{\gamma}

\newcommand{\ld}{\lambda}

\newcommand{\s}{\sigma}

\newcommand{\ft}{\widehat}

\newcommand{\cj}{\overline}
\newcommand{\dx}{\partial_x}
\newcommand{\dt}{\partial_t}

\newcommand{\nbar}{\overline{n}}
\newcommand{\ra}{\rightarrow}

\DeclarePairedDelimiter{\ceil}{\lceil}{\rceil}

\newcommand{\ta}{\theta}

\renewcommand{\l}{\ell}
\renewcommand{\o}{\omega}
\renewcommand{\O}{\Omega}

\newcommand{\les}{\lesssim}
\newcommand{\ges}{\gtrsim}

\newcommand{\jb}[1]
{\langle #1 \rangle}

\newcommand{\ind}{\mathbf 1}

\renewcommand{\S}{\mathcal{S}}

\newcommand{\M}{\mathcal{M}}
\def\e{\varepsilon}

\newcommand{\N}{\mathbb{N}}

\newcommand{\NN}{\mathcal{N}}
\newcommand{\MM}{\mathcal{M}}
\newcommand{\CC}{\mathcal{C}}

\renewcommand{\I}{\mathcal{I}}

\newcommand{\TT}{\mathcal{T}}

\newtheorem*{ackno}{Acknowledgements}

\usepackage{tikz}

\usetikzlibrary{shapes.misc}
\usetikzlibrary{shapes.symbols}
\usetikzlibrary{shapes.geometric}
\tikzset{
	dot/.style={circle,fill=black,draw=black,inner sep=0pt,minimum size=0.5mm},
	>=stealth,
	}
\tikzset{
	ddot/.style={circle,fill=white,draw=black,inner sep=0pt,minimum size=0.8mm},
	>=stealth,
	}

\tikzset{decision/.style={ 
        draw,
        diamond,
        aspect=1.5
    }}

\tikzset{dia2/.style
={diamond,fill=white,draw=black,inner sep=0pt,minimum size=1mm},
	>=stealth,
	}

\tikzset{dia/.style
={star,fill=black,draw=black,inner sep=0pt,minimum size=1mm},
	>=stealth,
	}

\makeatletter
\def\DeclareSymbol#1#2#3{\expandafter\gdef\csname MH@symb@#1\endcsname{\tikz[baseline=#2,scale=0.15]{#3}}}
\def\<#1>{\csname MH@symb@#1\endcsname}
\makeatother

\numberwithin{equation}{section}
\numberwithin{theorem}{section}

\begin{document}
\baselineskip = 14pt

\title[A.s. GWP for BBM with infinite $L^{2}$ data]
{Almost sure global well posedness for the BBM equation
with infinite $L^{2}$ initial data}

\author[J.~Forlano]
{Justin Forlano}

 \address{
 Justin Forlano, Maxwell Institute for Mathematical Sciences\\
 Department of Mathematics\\
 Heriot-Watt University\\
 Edinburgh\\ 
 EH14 4AS\\
  United Kingdom}

\email{j.forlano@hw.ac.uk}



\begin{abstract}
We consider 
the probabilistic Cauchy problem for the Benjamin-Bona-Mahony equation (BBM) on the one-dimensional torus $\T$ with initial data below $L^{2}(\T)$. With respect to random initial data of strictly negative Sobolev regularity, we prove that BBM is almost surely globally well-posed. 
The argument employs the $I$-method to obtain an a priori bound on the growth of the `residual' part of the solution. We then discuss the stability properties of the solution map in the deterministically ill-posed regime.
\end{abstract}

\maketitle
\tableofcontents

\section{Introduction}


We consider 
the Benjamin-Bona-Mahony equation (BBM): 
\begin{align}
\begin{cases}
\partial_{t}u-\partial_{xxt}u+\partial_{x}u+\frac{1}{2}\partial_{x}(u^{2})  = 0\\
u|_{t = 0} = u_0, 
\end{cases}
\ (x, t) \in \T \times \R_{+},
\label{BBM0}
\end{align}

\noi
where $u:\T \times \R_{+} \mapsto \R$ is the unknown function and $\T:=\R / (2\pi \Z)$ is the one-dimensional torus\footnote{We could also consider the BBM equation \eqref{BBM0} on $\T\times \R$ because of the time-reversal symmetry $u(x,t)\mapsto u(-x,-t)$ (viewing $\T$ as $[-\pi,\pi)$). However, for simplicity, we consider only positive times in the following.}.

The BBM equation is a model for the propagation of long wavelength, short amplitude water waves~\cite{Peregrine, BBM}.
 In particular, in \cite{BBM}, it was proposed as an alternative to the Korteweg-de Vries (KdV) equation.
 This is in part due to the boundedness of the dispersion relation for BBM while the dispersion relation for KdV is unbounded. Along with its more preferable  
 analytical qualities, BBM is also known as the regularised long wave equation. For further discussion on the physical validity of the BBM model, see for example \cite{Model1, Model2, Model3}.

Our goal in this paper is to study the well-posedness of BBM~\eqref{BBM0} in the low regularity setting. We begin by putting \eqref{BBM0} into an alternative form which is more amenable for this study. 
By factorising the time derivative, we rewrite \eqref{BBM0} 
into the equivalent form:
\begin{equation}
\partial_{t}u =-(1-\partial_{x}^{2})^{-1}\partial_{x}\left( u+\frac{1}{2}u^2 \right). \label{1.2}
\end{equation}
With $D_{x}:=-i\partial_{x}$, let $\vp(D_x)$ be the Fourier multiplier operator 
with symbol  $\varphi(n):=\frac{n}{1+n^{2}}$; that is,
\begin{align*}
\ft{ \vp(D_x)f}= \vp(n)\ft f (n)
\end{align*}
for every $n\in \Z$, where $\ft f$ denotes the Fourier transform on $\T$ of $f$.
 Then, \eqref{1.2} reads as
\begin{equation}
i\partial_{t}u=\varphi(D_{x})u + \frac{1}{2}\NN(u). \label{BBMD}
\end{equation}
Here, we specifically interpret the nonlinearity as
\begin{align}
 \NN(u):=\sum_{n\neq 0} \vp(n) e^{inx} \sum_{\substack{n_1, n_2 \in \Z \\ n=n_1+n_2}} \ft u(n_1) \ft u(n_2).  \label{nonlin}
\end{align} 
 The key point here is the \emph{explicit} absence of the zero frequency\footnote{Notice that $\vp(0)=0$.}. Note that if $u\in L^{2}(\T)$, $\NN(u)=\vp(D_x)(u^{2})$ and hence \eqref{BBMD} is equivalent to \eqref{BBM0}; see Remark~\ref{rmk:nonlin}.
 We thus consider \eqref{BBMD} as the natural version of \eqref{BBM0} to study below $L^{2}(\T)$, and we will refer to \eqref{BBMD} as the BBM equation unless otherwise stated.
We have the following integral (Duhamel) formulation of \eqref{BBMD}: 
 \begin{equation}
u(t)=S(t)u_{0}-\frac{i}{2}\int_{0}^{t}S(t-t')\NN(u(t')) dt', \label{Duhamel1}
\end{equation} 
where $S(t):=e^{-it\varphi(D_{x})}$ is the linear BBM propagator. We stress that solutions to \eqref{BBM0} and \eqref{BBMD} are real-valued; the presence of $i$ in the above formulas is a side-effect of writing the multiplier $\vp(D_x)$. 
We say that $u$ is a solution to \eqref{BBMD} if it satisfies the Duhamel formulation \eqref{Duhamel1}.

The well-posedness of the Cauchy problem for BBM~\eqref{BBM0} on $\mathcal{M}=\R$ or $\T$ has been well studied within the class of $L^{2}$-based Sobolev spaces $H^{s}(\mathcal{M})$.
 Benjamin, Bona and Mahony~\cite{BBM} obtained global well-posedness of \eqref{BBM0} in $H^{k}(\R)$ for all integers $k\geq 1$. When~$k=1$, globalisation of local-in-time solutions follows immediately from the conserved energy
\begin{align}
E(u(t))&:=\frac{1}{2}\int_{\mathcal{M}}u^{2}(t)+(\partial_{x}u(t))^{2}\, dx=\frac 12 \|u(t)\|_{H^{1}(\mathcal{M})}^{2}. \label{energycons} 
\end{align}
Namely, if $u(t)\in H^{1}(\mathcal{M})$ satisfies \eqref{Duhamel1} for all $t\in [0,T]$, then for any $t\in [0,T]$, we have
\begin{align}
E(u(t))=E(u(0))=E(u_0). \label{econs2}
\end{align}
Local well-posedness in $L^{2}(\R)$ was obtained by Bona, Chen and Saut~\cite{BCS1,BCS2}. In \cite{BonaTzvet}, Bona and Tzvetkov proved BBM~\eqref{BBM0} is globally well-posed in $H^s (\R)$ for all $s\geq 0$. Adapting the arguments in \cite{BonaTzvet}, Roum\'egoux~\cite{Roumegoux} extended this result to the periodic setting. On the other hand, the results of Panthee~\cite{Panthee} and Bona and Dai~\cite{BonaNormInf} showed that the BBM equation~\eqref{BBM0} is ill-posed in negative Sobolev spaces. For further discussion on the ill-posedness of~\eqref{BBM0} in negative Sobolev spaces, we refer to Subsection \ref{intro:norminf} and Theorem~\ref{thm:norminf} below. 
In this paper, we study the well-posedness of \eqref{BBMD} below $L^{2}(\T)$ with random initial data.

\subsection{Almost sure local well-posedness}  
Recall that \emph{well-posedness} in the sense of Hadamard corresponds to (i) existence of a solution, (ii) uniqueness of the solution (in some suitable sense) and (iii) continuous dependence with respect to initial data. 
The ill-posedness results for \eqref{BBM0} below $L^{2}(\T)$ that we mentioned above are all based on contradicting~(iii). More precisely, they show the solution map $\Phi:u_0\in H^{s}(\mathcal{M})\mapsto u\in C([0,T];H^{s}(\mathcal{M}))$ for BBM~\eqref{BBM0} is discontinuous when $s<0$; see Theorem~\ref{thm:norminf} and Corollary~\ref{cor:dscty}. Namely, they construct a smooth sequence $u_{0,n}\to 0$ in $H^{s}(\mathcal{M})$, such that the smooth solutions $\Phi(u_n)$ to \eqref{BBM0} fail to converge to zero in $C([0,T_{n}]; H^s(\mathcal{M}))$.
In particular, these same ill-posedness results also hold for \eqref{BBMD}. 
Note however that this does not preclude the possible existence (and even uniqueness) of solutions within the ill-posed regime. This leads us to the following question: can we still construct solutions in the ill-posed regime and if so, in what sense do we retain (iii), the continuity of the solution map?

Our goal in this paper is to address this question within the context of BBM~\eqref{BBMD} with random initial data below $L^{2}(\T)$. Namely, we consider randomised initial data of the form\footnote{We drop the factor of $2\pi$ as it plays no role in our analysis.} 
\begin{equation}
u_{0}^{\omega}(x)=\sum_{n\in \Z}\frac{g_{n}(\omega)}{\jb{n}^{\alpha}}e^{inx}, \label{initialdata}
\end{equation}
where $\al\in \R$, $\jb{\, \cdot\,}:=(1+|\cdot|^{2})^{\frac{1}{2}}$ and  $\{g_{n}\}_{n\in \Z}$ is a sequence of independent standard complex-valued Gaussian random variables on a probability space $(\Omega, \mathcal{F}, \prob)$ satisfying the reality condition $g_{n}=\overline{g_{-n}}$ and $g_{0}$ is real.
A computation shows\footnote{Here, we use the notation $a-$ (respectively, $a+$) to denote $a-\ep$ (respectively, $a+\ep$), where $0< \ep \ll  1$ is extremely small.}
\begin{align}
u_{0}^{\o}\in H^{\alpha-\frac{1}{2}-}(\T)\setminus H^{\alpha-\frac{1}{2}}(\T) \label{regu0}
\end{align}
 almost surely; see \eqref{integ} below. Thus, in view of the global well-posedness of BBM~\eqref{BBMD} in $L^{2}(\T)$ and above, we concentrate on when $\al \leq \tfrac 12$. Our first result is almost sure local well-posedness for the BBM equation \eqref{BBMD} with respect to the random initial data \eqref{initialdata} for $\al\in (\tfrac 14, \tfrac 12]$.
 \vspace{-0.5mm}
\begin{theorem}\label{Thm:ASLWP} 
 Let $\alpha\in (\frac{1}{4}, \frac{1}{2}]$ and $s\in (\tfrac 12 -\al, 2\al )$.
Then, the BBM equation \eqref{BBMD} is locally well-posed almost surely with respect to the random initial data \eqref{initialdata}. More precisely, there exists a set $\Sigma \subset \O$ with $\prob(\Sigma)=1$, such that for every $\o\in \Sigma$, there exist $T^{\o}>0$ and a unique solution $u$ to \eqref{BBMD} in
\begin{align*}
e^{-it\varphi(D_{x})}u_{0}^{\o}+C([0,T^{\o}]; H^{s}(\T))\subset C([0,T^{\o}]; H^{\alpha-\frac{1}{2}-}(\T))
\end{align*}
with initial condition $u_{0}^{\omega}$ of the form \eqref{initialdata}. 
 \end{theorem}
Uniqueness in the above statement refers to uniqueness within the space of functions $u\in C([0,T]; H^{\al-\frac 12 - }(\T))$ which can be written as
\begin{align*}
u=e^{-it\varphi(D_{x})}u_{0}^{\o}+v &\in C([0,T]; H^{\al-\frac 12 -}(\T))+C([0,T]; H^{s}(\T)) 
\\ & \subset C([0,T]; H^{\al-\frac 12 - }(\T)).
\end{align*}
See also Remark~\ref{remark:cty}.

In the proof of Theorem~\ref{Thm:ASLWP}, we obtain solutions $u$ of the form
\begin{align*}
u(t) =  z^{\o}(t)+v(t), 
\end{align*}
where $z^{\o}(t)=S(t)u_0^{\o}$ is the random linear solution and the remainder $v:=u-z$ is almost surely smoother than $z$. In particular, $v$ belongs to $H^{s}(\T)$ for $s<2\al$
  and we construct $v$ by a contraction mapping argument for the following \emph{perturbed} BBM equation:
\begin{align}
\begin{cases}
i\dt v=  \vp(D_x)(v)+\frac 12 (v^{2}+2zv) +\frac{1}{2}\NN(z),  \\
v|_{t = 0} = 0.
\end{cases}
\label{BBMv}
\end{align}

Due to this expectation of additional smoothness for $v$, the Sobolev multiplication law will allow us to make sense of the product $zv$. However, as $z\notin L^{2}(\T)$ almost surely, it is essential that we interpret the forcing term $\tfrac 12 \NN(z)$ in the sense of \eqref{nonlin}; indeed, see Remark~\ref{rmk:nonlin}.
In studying the fixed point problem for $v$ corresponding to \eqref{BBMv} we crucially make use of the fact that while the random initial data \eqref{initialdata} is no more regular in space as compared to the (deterministic) function 
\begin{align*}
\sum_{n\in \Z} \frac{1}{\jb{n}^{\al}}e^{inx} \in H^{\al-\frac 12 -}(\T),
\end{align*}
it does benefit from a gain of integrability. More precisely, we have $u_0^{\o}\in W^{\al-\frac 12 -,\infty}(\T)$ almost surely. 
 Indeed, for any $2\leq p<\infty$, we have\footnote{Here we use that if $\{a_n\}_{n\in \Z}\in \l^2_{n}(\Z)$, then $\sum_{n\in \Z}a_n g_{n}(\o)$ is a mean-zero  complex-valued Gaussian random variable with variance $\|a_n\|_{\l^{2}_{n}}^{2}$ and hence satisfies 
 \begin{align*}
\bigg\| \sum_{n\in \Z}a_n g_{n}(\o)\bigg\|_{L^{p}(\O)}\sim \sqrt{p}\|a_n\|_{\l^2_{n}},
\end{align*}
for any $2\leq p<\infty$. See also Lemma~\ref{Lemma:wienerchaos}.}
\begin{align}
\begin{split}
\E[\|u_{0}^{\o}\|_{W^{\alpha-\frac{1}{2}-,p}}^{p}]&=\bigg \| \bigg\| \sum_{n}\frac{\jb{n}^{\alpha-\frac{1}{2}- }e^{inx}}{\jb{n}^{\alpha}}g_{n}(\o) \bigg\|_{L^{p}(\O)}^{p}\bigg\|_{L^{p}_{x}(\T)} \\
& \lesssim \sqrt{p}\| \|\jb{n}^{-\frac{1}{2}-}\|_{\l^{2}_{n}}\|_{L^{p}_{x}(\T)} <\infty.
\end{split} 
\label{integ}
\end{align} 
For the endpoint $p=\infty$, we first apply the Sobolev inequality to reduce to some large but finite spatial integrability exponent and then apply Minkowski's integral inequality.  
In the dispersive PDE with random data literature, a perturbative expansion of the form~\eqref{dpdtrick} goes back to the works of McKean~\cite{mckean} and Bourgain~\cite{Bourgain1} and is known as the Da Prato-Debussche trick in the context of stochastic PDEs, after \cite{dpd}.

Initial data of the form \eqref{initialdata} correspond to typical elements belonging to the support of the infinite-dimensional Gaussian measure $\mu_{\al}$ which formally has density
\begin{align}
d\mu_{\al}=Z_{\al}^{-1}e^{ -\frac{1}{2}\|u\|_{H^{\al}(\T)}^{2}}du. \label{mual}
\end{align}
Here, $du$ is the (non-existent) infinite-dimensional Lebesgue measure. More rigorously, given $\al\in \R$, the Gaussian measure $\mu_{\al}$ is the induced probability measure under the map
\begin{align*}
\omega\in\Omega\longmapsto u^\omega(x)=\sum\limits_{n\in\Z}\frac{g_n(\omega)}{\langle n\rangle^\al}e^{inx}.
\end{align*}
From \eqref{regu0}, we see that $\mu_{\al}$ is supported on $H^{\al-\frac 12-}(\T)\setminus H^{\al-\frac 12}(\T)$. Using this perspective, we may rephrase Theorem~\ref{Thm:ASLWP} as almost sure local well-posedness of BBM~\eqref{BBMD} with respect to the Gaussian measure $\mu_{\al}$ supported on $H^{\al-\frac 12-}(\T)$ for any $\al>\frac{1}{4}$.

Beginning with the work of Bourgain~\cite{Bourgain2, Bourgain1} on the periodic nonlinear Schr\"{o}dinger equation (NLS) and Burq and Tzvetkov~\cite{BTlocal, BTglobal} on nonlinear wave equations (NLW), there has been an intense interest on constructing solutions to nonlinear dispersive PDE in the ill-posed regime using randomised initial data. As the literature on nonlinear dispersive PDE with random initial data is by now quite vast, we will focus on those of immediate relevance for our study of the BBM equation \eqref{BBM0}. 
In studying invariance properties of certain weighted Gaussian measures (more specifically, Gibbs measures) for the cubic NLS on $\T^{2}$, Bourgain~\cite{Bourgain1} first needed to construct a well-defined flow emanating from (a two dimensional version of) the initial data \eqref{initialdata} with $\al=1$. In~\cite{CollOh}, Colliander and Oh proved the cubic NLS on $\T$ is locally and globally well-posed almost surely with respect to random initial data of the form \eqref{initialdata} for $\al>\tfrac 16$ and $\al>\tfrac{5}{12}$, respectively. We refer the reader to the survey paper~\cite{BOP4} for further details on nonlinear dispersive PDE with random data.  

For the context of BBM~\eqref{BBM0}, the transport properties of Gaussian measures under the nonlinear flow of \eqref{BBM0} have been well-studied~\cite{desuzzoni1, desuzzoni2, desuzzoni3, Tzv3}. As BBM is a Hamiltonian PDE with Hamiltonian given by the energy \eqref{energycons}, it has a naturally associated Gibbs measure 
 \begin{align*}
d\mu_1=Z_{1}^{-1}e^{-E(u)}du.
\end{align*}
 We thus expect $\mu_1$ to be invariant under the (nonlinear) flow of \eqref{BBM0} and this was proved by de Suzzoni~\cite{desuzzoni2}. 
For the Gaussian measures $\mu_{\al}$ with $\al\neq 1$, we no longer expect invariance. However, Tzvetkov~\cite{Tzv3} proved that the push-forward of the Gaussian measures $\mu_{\al}$ under the BBM flow for integer $\al\geq 2$ are quasi-invariant (i.e.~mutually absolutely continuous) with respect to $\mu_{\al}$. In view of the global well-posedness of the BBM equation \eqref{BBM0} in $L^{2}(\T)$, it would be of interest to study if the quasi-invariance of Gaussian measures persists for all $\al>\tfrac 12$.
 For $\al \leq \tfrac{1}{2}$, a first step is the probabilistic well-posedness of \eqref{BBMD} below $L^{2}(\T)$ with initial data of the form~\eqref{initialdata}. In this paper, we establish local well-posedness (Theorem~\ref{Thm:ASLWP}) and global well-posedness (Theorem~\ref{Thm:ASGWP}) of BBM~\eqref{BBMD} below $L^{2}(\T)$ with initial data of the form~\eqref{initialdata}. We discuss our global well-posedness result in the next subsection.

We conclude this subsection with a few remarks. 

\begin{remark}\rm \label{remark:cty}
The proof of Theorem~\ref{Thm:ASLWP} decomposes the ill-posed solution map $\Phi: u_0^{\o}\in  H^{\al-\frac 12 -}(\T) \mapsto u \in C([-T, T]; H^{\al - \frac 12 -}(\T))$ for \eqref{BBMD} with data $u_0^{\o}$ given by \eqref{initialdata} into the following sequence of maps: 
\begin{align*}
u_{0}^{\o} \stackrel{\text{(I)}}{\longmapsto} (z^{\o}, Z^{\o}) & \stackrel{\text{(II)}}{\longmapsto} v \in C([0,T]; H^{s}(\T))
 \stackrel{\text{(III)}}{\longmapsto} u=z+v \in C([0,T]; H^{\al-\frac 12 - }(\T)),
\end{align*}
where $z^{\o}=S(t)u_0^{\o}$ and $Z^{\o}:=\NN(z^{\o})$. Step (I) uses tools from stochastic analysis in order to construct the enhanced data set $(z^{\o},Z^{\o})$. This is the result of Proposition~\ref{prop:stochobjects}. Step (II) is a deterministic fixed point argument for \eqref{BBMv} which views $(z^\o, Z^\o)$ as a given data set. In particular, this step implies the continuity of the map:
\begin{align}
\begin{split}
\Psi:(z^{\o},Z^{\o})\in C([0,T];W^{\al-\frac 12-,\infty}(\T))& \times C([0,T];H^{s}(\T)) \\
&\longmapsto v^{\o}\in C([0,T];H^{s}(\T)),
\end{split}  \label{Psi}
\end{align}
where $v$ solves the perturbed BBM equation~\eqref{BBMv}.
 Finally, step (III) recovers $u$ through the expansion $u=z+v$. 
 Similar decompositions of this type for ill-posed solution maps appear prominently in the theories of stochastic PDEs~\cite{HairerLec, GPnotes} and rough paths~\cite{FH}.
\end{remark}

\begin{remark}\rm \label{rmk:nonlin}
In this remark, we discuss the necessity for the interpretation of the nonlinearity $\vp(D_x)(u^2)$ as $\NN(u)$ in \eqref{nonlin}. 
Let $\rho \in C(\R; [0,1])$ with $\supp \rho \subset \big(-\frac 12, \frac 12 \big]$  be such that $\int_{\R} \rho \, dx=1$ and we set $\rho_{k}(x)=k\rho(kx)$ for $k\in \mathbb{N}$. As $k\geq 1$, we see that $\{ \rho_{k}\}_{k\in \mathbb{N}}$ is an approximate identity on $\T$. To motivate \eqref{nonlin}, we consider the following smoothed version of \eqref{BBMD}:
\begin{align}
\begin{cases}
i\dt u_k = \vp(D_x) ( u_k)+\frac 12 \vp(D_x)(u^{2}_{k}),    \\
u_{k}|_{t = 0} = u_0^{\o}\ast \rho_{k}.
\end{cases}
\label{BBMe}
\end{align}
Given $k\in \mathbb{N}$, the global theory in $L^{2}(\T)$ for \eqref{BBMD} and a persistence-of-regularity argument shows the solution $u_{k}$ to \eqref{BBMe} exists globally in time and is smooth. Now, let $z_{k}(t)=S(t)(u_0^{\o}\ast \rho_{k})$ be the random linear solution to \eqref{BBMe} and consider an expansion of $u_k$ about $z_k$ by setting 
\begin{align}
v_k : =u_k -z_k,  \qquad \text{so} \qquad u_k = z_k +v_k. \label{dpdtrick}
\end{align}
Then, $v_k$ solves the perturbed BBM equation
\begin{align}
i\dt v_k = \vp(D_x) (v_k) + \frac 12\vp(D_x)( v_{k}^2 + 2v_k z_k)+ \frac 12 \vp(D_x)( z_{k}^{2}  ), \label{vepseq}
\end{align} 
with $v_k |_{t=0}=0$.
The problematic term here is $z_{k}^{2}$. More specifically, the zero frequency mode (equivalently, the mean) of $z_{k}^{2}$ behaves like
 \begin{align}
 \begin{split}
C_{k}:&=\E \big[ \P_{0}( z^{2}_{k}(x,t)) \big] \\
& = \sum_{0=n_1+n_2} \frac{\ft{\rho}(k^{-1} n_1) \ft{\rho}(k^{-1} n_2)e^{-it\vp(n_1)}e^{-it\vp(n_2)}}{\jb{n_1}^{\al} \jb{n_2}^{\al}} \E[ g_{n_1}g_{n_2}] \\
& \sim \sum_{m} \frac{ |\ft{\rho}(k^{-1} m)|^{2}}{\jb{m}^{2\al}} 
 \sim  \sum_{|m| \leq k} \frac{ 1}{\jb{m}^{2\al}}, 
 \end{split}  \label{Ceps}
\end{align}
where $\P_{0}$ denotes the projection onto the zero Fourier mode: $\ft{\P_{0}f}(n) = \ft f(n)\ind_{\{n=0\}}$. Hence, $C_k$ diverges like $\log k$ if $\al=\frac 12$ and like $k^{1-2\al}$ if $\al < \frac 12$ as $k \ra \infty$. Notice that $C_{k}$ depends on the choice of mollifying kernel $\rho$ but is independent of $(x,t)\in \T \times \R_{+}$. Thus, $z_{k}^{2}$ will not converge in the limit as $k \ra \infty$ in any reasonable sense. 

However, as $z_{k}$ is smooth and the fact that\footnote{Equivalently, the operator $\vp(D_x)$ vanishes on constants.} 
 $\vp(0)=0$, we have
 \begin{align}
\vp(D_x)(z_{k}^2) = \vp(D_x)( \P_{\neq 0}( z^{2}_{k}))=  \sum_{n\neq 0} \vp(n)e^{inx} \sum_{\substack{n_1, n_2 \in \Z \\ n_1+n_2=n}} \ft z_{k}(n_1) \ft z_{k}(n_2),  \label{obvious}
\end{align}
and we show that the right hand side of \eqref{obvious} converges almost surely to the distribution 
\begin{align}
 \sum_{n\neq 0} \vp(n) e^{inx} \sum_{\substack{n_1, n_2 \in \Z \\ n_1+n_2=n}} \ft z(n_1) \ft z(n_2)=\NN(z),\label{nonlinz}
\end{align} 
where $z(t):=S(t)u_0^{\o}$ is the random linear solution with data \eqref{initialdata}; see Proposition~\ref{prop:stochobjects}. Thus, we are led to consider $\NN(u)$ in \eqref{nonlin}. 
Note that when $u\in L^2(\T)$, $\NN(u)$ is equal to $\vp(D_x)(u^2)$ since
\begin{align*}
\NN(u)=\vp(D_x)\bigg( u^{2}-\int_{\T}u^{2} dx \bigg)=\vp(D_x)(u^2)-\vp(D_x)\bigg(\int_{\T}u^{2} dx \bigg)=\vp(D_x)(u^2). 
\end{align*}
The above computation shows, at least formally, we do not `see' any difference at the level of the equation \eqref{BBM0} between the two notions of nonlinearity $\vp(D_x)(u^2)$ and $\NN(u)$. 
\end{remark}

\begin{remark}\label{rmk:lowerbd}\rm
The lower bound $\al >\tfrac 14$ in Theorem~\ref{Thm:ASLWP} is sharp in the following sense.
In~\cite{BOP3, OPT}, it is was shown that one may improve upon regularity thresholds for almost sure local well-posedness of dispersive PDE with random initial data by considering a higher order perturbative expansion. In particular, a higher order expansion is actually necessary for the KdV equation with random initial data~\cite{HiroSZEGOKDV}.
 In the context of the BBM equation \eqref{BBMD}, this corresponds to writing
\begin{align}
u = z^\o+\tilde{Z}^\o + w, \label{2ndord}
\end{align} 
where 
\begin{align*}
 \tilde{Z}^{\o}:=-\frac{i}{2}\int_{0}^{t}S(t-t')\NN(z(t'))dt'
\end{align*}
is the second Picard iterate
and studying the fixed point problem for $w$. The idea is that the expansion \eqref{2ndord} has removed the term $\tilde{Z}^{\o}$ which is responsible for the regularity threshold obtained from studying just a first-order expansion. 
However, we show in Subsection~\ref{subs:sharpness} that $\tilde{Z}^{\o}$ fails to define a distribution almost surely when $\al\leq \tfrac 14$, and hence we find no improvement from considering higher order expansions.
\end{remark}

\subsection{Almost sure global well-posedness}

Our next goal is to globalise in time the local solutions constructed in Theorem~\ref{Thm:ASLWP}. In this direction, we establish the following:

\begin{theorem}\label{Thm:ASGWP} 
The BBM equation \eqref{BBMD} is almost surely globally well posed in $H^{-\e}(\T)$, for any $0<\e \ll 1$, with respect to random initial data of the form 
\begin{align}
u_{0}^{\o}=\sum_{n\in \Z}\frac{g_{n}(\o)}{\jb{n}^{\frac 12}}e^{inx}. \label{globaldata}
\end{align}
 More precisely, for almost every $\omega \in \Omega$, there exists a unique solution $u$ of \eqref{BBMD} in \begin{align*}
e^{-it\varphi(D_{x})}u_{0}+C(\R; H^{s}(\T))\subset C(\R; H^{-\e}(\T))
\end{align*}
with initial condition $u_{0}^{\omega}$ of the form \eqref{globaldata}. 
\end{theorem}

For fixed $\al\in (\tfrac 14, \tfrac 12]$, our probabilistic local theory (Theorem~\ref{Thm:ASLWP}) shows that we may extend the local-in-time solutions $u=z+v$ provided the $H^s(\T)$-norm, with $s=2\al-$, of the solutions $v$ to the perturbed BBM equation \eqref{BBMv} remains finite. That is, for any $T>0$, we seek to establish the following bound:
\begin{align}
\sup_{t\in [0,T]}\|v(t)\|_{H^s(\T)}\leq C(T)<\infty. \label{vhsbnd}
\end{align}
Notice in this setting, we only know $v\in H^s(\T)$ for $s<1$ and hence we cannot make use of the energy $E(v(t))$ of \eqref{energycons}, regardless of its non-conservation under the equation~\eqref{BBMv}.
This seems to indicate smoothing $v$ which motivated us to apply the $I$-method of Colliander, Keel, Staffilani, Takaoka and Tao~\cite{Iteam1, Iteam2} in this probabilistic context.

The approach is as follows: 
we smooth the initial data by applying the Fourier multiplier operator $I_{N}$ given by 
$\widehat{I_{N}f}(n)=m_{N}(n)\widehat{f}(n)$, $n\in\Z$, where $m_{N}(n)$ is the restriction to the integers of the smooth function $m:\R \mapsto \R$ defined by 
\begin{align}
m_{N}(\xi):= m\bigg(\frac{\xi}{N}  \bigg) =
\begin{cases}
1    \,\,\, & \text{if}\,\,\, |\xi|\leq N,  \\
\Big( \frac{N}{|\xi|} \Big)^{1-s} \,\,\,&\text{if}\,\,\, |\xi|> N.
\end{cases} \label{Imult}
\end{align}
Thus, the operator $I_{N}$ is the identity on low frequencies and a fractional integral operator on high frequencies, hence the name $I$-method.
For simplicity of presentation, we will now drop the subscript $N$. 
It is easy to see that $Iv(t)\in H^{1}(\T)$ almost surely and satisfies 
\begin{align}
\begin{cases}
\partial_{t}Iv = -(1-\partial_{x}^{2})^{-1}\partial_{x}(v)-\frac{1}{2}(1-\partial_{x}^{2})^{-1}\partial_{x}[ I(v^{2})+2I(vz)]+\frac{1}{2} I(\NN(z)) , \\
Iv|_{t = 0} = 0. 
\end{cases} 
\label{IBBM}
\end{align}
By defining the `modified energy' 
\begin{align*}
E(Iv)(t)=\frac{1}{2}\| Iv(t)\|_{H^{1}(\T)}^{2},
\end{align*}
and observing 
\begin{align*}
\|v(t)\|_{H^{s}(\T)} \les E(Iv)(t),
\end{align*}
we reduce proving \eqref{vhsbnd} to obtaining 
\begin{align}
 \sup_{t\in [0,T]}E(Iv)(t) \leq C(T)<\infty. \label{modvbnd}
\end{align}

 Now, $E(Iv)$ will not be conserved under the flow of \eqref{IBBM} because (i) $v$ solves a perturbed BBM equation and (ii) $I$ does not commute with products. However, $Iv$ is expected to `almost' solve the same equation as $v$, namely \eqref{BBMv}, in the sense that the error terms generated from the failure of commutation are themselves commutators. Indeed, taking a time derivative of $E(Iv)$, inserting \eqref{IBBM} and making appear commutators, we schematically arrive at 
\begin{align}
\begin{split}
\frac{d}{dt}E(Iv)(t) \sim  \int_{\T} (\dx Iv)[I(v^2)-(Iv)^{2}] dx &+ \int_{\T}Iz \, Iv \,\dx Iv \, dx  \\
\\ & + \text{lower order terms}. 
\end{split}
\label{schematicODE}
\end{align}
We estimate the first expression by  
\begin{align*}
\int_{\T} (\dx Iv)[I(v^2)-(Iv)^{2}]\, dx \les N^{-\beta}E^{\frac{3}{2}}(Iv)  
\end{align*}
for some $\beta>0$ and notice that this term in \eqref{schematicODE} implies that $E(Iv)$ will blow-up in a finite time $T_{N}$. However, up to time $T_N$, the negative power of $N$ allows us view this term as part of the lower order corrections. For the second term, we have 
\begin{align*}
\int_{\T}Iz \, Iv \,\dx Iv \, dx \les \|Iz\|_{L^{2}}E(Iv).
\end{align*}
For this term, placing $Iz$ into $L^{2}(\T)$ comes at the expense of a loss in $N$; see Lemma~\ref{lemma:izintmoment} (an analogue of this can be found in \cite{GKOT, Leonardo}).
When $\al=\tfrac{1}{2}$, we lose only a logarithm of $N$, and it is only this loss which is acceptable for ensuring we may take $T_{N}\geq 2T$ by choosing $N=N(T)$ large enough.
By Gronwall's inequality, we obtain \eqref{modvbnd} provided $\al=\tfrac 12$ which yields Theorem~\ref{Thm:ASGWP}.

 This $I$-method approach has recently been applied in the context of nonlinear stochastic dispersive PDE by Gubinelli, Koch, Oh and Tolomeo~\cite{GKOT} and Tolomeo~\cite{Leonardo}. 
We closely follow their arguments, although certain technical difficulties are absent for \eqref{BBMD} as compared to their setting. A natural modification of the argument above would be to include the low-high splitting idea in \cite{BonaTzvet}; however, this does not seem to lead to any regularity improvement over Theorem~\ref{Thm:ASGWP}; see Remark~\ref{remark:lowhigh}.

\begin{remark}\rm \label{remark: gaussianity}
We may relax the Gaussianity assumption on the random variables $\{g_{n}\}_{n\in \Z}$. More precisely, in Appendix~\ref{app:Gauss}, we detail how our arguments allow us to establish analogous almost sure local and global existence results for BBM~\eqref{BBMD} with respect to initial data of the form:
\begin{align*}
u_{0}^{\o}=\sum_{n\in \Z}\frac{g_{n}(\o)}{\jb{n}^{\al}}e^{inx},
\end{align*}
where the complex-valued (not necessarily Gaussian) random variables $\{g_{n}\}_{n\in \Z}$ satisfy assumptions \eqref{item1}-\eqref{item5} in Appendix~\ref{app:Gauss}.  
Note, however, that with such randomisations we lose the link with the Gaussian measures $\mu_{\al}$ in~\eqref{mual}.
\end{remark}

\subsection{Norm inflation at general data in negative Sobolev spaces}\label{intro:norminf}
In this subsection, we study the (purely) deterministic ill-posedness of the solution map $\Phi: u_0\in H^{s}(\mathcal{M})\mapsto u\in C([0,T];H^{s}(\mathcal{M}))$ to BBM~\eqref{BBM0}. In~\cite{Panthee}, Panthee showed the failure of continuity of the solution map at the origin in $H^{s}(\M)$ for any $s<0$. This result implies that BBM~\eqref{BBM0} is ill-posed in negative Sobolev spaces. 
 Bona and Dai~\cite{BonaNormInf} showed that for $s<0$, the solution map exhibits the stronger phenomenon known as \textit{norm inflation at zero}: given $s<0$, for any $\ep >0$, there exists a smooth solution $u_{\ep}$ to BBM~\eqref{BBM0} and times $t_{\ep}\in (0,\ep)$ such that 
\begin{align*}
\|u_{\ep}(0)\|_{H^{s}(\M)}<\ep \qquad \text{and} \qquad \|u_{\ep}(t_{\ep})\|_{H^{s}(\M)}>\ep^{-1}.
\end{align*} 
Norm inflation based at \textit{general} initial data has been studied for NLW~\cite{Xia, Tzv1} and NLS~\cite{HiroInflation} in negative Sobolev spaces.
We establish norm inflation at any $u_{0}\in H^s(\M)$, with $s<0$, for the BBM equation \eqref{BBM0}. It is clear that the same result also holds for \eqref{BBMD}.

\begin{theorem}\label{thm:norminf} 
Let $\M=\R$ or $\T$, $s<0$ and fix $u_{0}\in H^{s}(\M)$. Then, given any $\e >0$, there exists a smooth solution $u_{\e}$ to \eqref{BBM0} on $\M$ and $t_{\e}\in (0,\ep)$ such that 
\begin{align*}
\|u_{\e}(0)-u_{0}\|_{H^{s}(\M)}<\e \qquad \text{ and } 
\qquad \| u_\ep(t_{\e})\|_{H^s(\M)} > \ep^{-1}. 
\end{align*} 
\end{theorem}
As an immediate corollary, we obtain the everywhere discontinuity of the solution map of \eqref{BBM0} in negative Sobolev spaces. 

\begin{corollary}\label{cor:dscty}
Let $\M=\R$ or $\T$ and $s<0$.
Then, 
  for any $T>0$, 
the solution map $\Phi: u_0\in  H^s(\M) \mapsto u \in C([0, T]; H^s(\M))$ 
to the BBM equation \eqref{BBM0} is discontinuous
everywhere in $H^s(\M)$.
\end{corollary}

To prove Theorem~\ref{thm:norminf}, we employ the argument in Oh~\cite{HiroInflation}, where norm inflation based at general initial data was studied for the cubic NLS in negative Sobolev spaces. Briefly, the key idea is to write a solution $u$ to \eqref{BBM0} with $u|_{t=0}=u_0$ in terms of its power series expansion:
\begin{align*}
u=\sum_{j=1}^{\infty} \Xi_{j}(u_0),
\end{align*} 
which is an infinite sum of recursively defined homogeneous multilinear operators $\Xi_{j}$, in terms of $u_0$, of increasing order; see also \cite{Christ1, IO, CP, Kishimoto}. Such a power series expansion is motivated by the Picard iteration scheme. In \cite{HiroInflation}, these power series expansions are indexed using trees, which simplifies their handling, both combinatorially and analytically (in terms of obtaining multilinear estimates). One then exploits a high-to-low energy transfer in the second term of the expansion in order to exhibit the instability stated in Theorem~\ref{thm:norminf}.

We stress that Theorem~\ref{thm:norminf} and its proof are entirely deterministic. We relate this result to the solutions constructed from rough random initial data of the form \eqref{initialdata} in the next subsection.

\begin{remark}\label{rmk:flp} \rm 
We extend our study of the (deterministic) BBM equation~\eqref{BBM0} to the context of the Fourier-Lebesgue spaces $\FL^{s,p}(\M)$ which are defined through the norm 
\begin{align*}
\|f\|_{\FL^{s,p}(\M)}=\| \jb{\xi}^{s}\ft f(\xi)\|_{L^{p}(\ft \M)},
\end{align*}
where 
\begin{align*}
\ft \M = 
\begin{cases}
\R & \text{if}\,\,\, \M=\R, \\
\Z &  \text{if} \,\,\, \M=\T,
\end{cases}
\end{align*}
where we use the Lebesgue measure if $\ft \M =\R$ and the counting measure if $\ft \M =\Z$. When $p=2$, we have $\FL^{s, 2}(\M)=H^{s}(\M)$ and when $p=1$ and $s=0$, the space $\FL^{0,1}(\M)$ is the Wiener algebra. For convenience, we write $\FL^{p}(\M)$ instead of $\FL^{0,p}(\M)$. 
In Section~\ref{section:norminf}, we show that BBM~\eqref{BBM0} is locally well-posed in $\FL^{s,p}(\M)$ for any $s\geq 0$ when $1\leq p\leq 2$ and for any $s>\tfrac 12- \tfrac 1p$ when $p>2$.
In analogy to \eqref{BBM0} in $H^{s}(\M)$, we can extend the norm inflation result of Theorem~\ref{thm:norminf} to norm inflation at general data in $\FL^{s,p}(\M)$ for any $s<0$ and $1\leq p <\infty $; see Section~\ref{section:norminf} for details. In the periodic case, our interest in this result lies in the following observation:
\begin{align*}
\FL^{s, p}(\T)\subseteq H^{s}(\T)
\end{align*}
for any $1\leq p \leq 2$ and any $s\in \R$. Namely, the instability in Theorem~\ref{thm:norminf} persists in the stronger $\FL^{s,p}(\T)$-norm.
  On $\M=\R$, there is no inclusion between these spaces for a fixed $s$.
\end{remark}

\subsection{Stability in the ill-posed regime} \label{intro:stability}

In this subsection, we discuss notions of stability for the solution map $\Phi$ of Theorem~\ref{Thm:ASLWP}.
 Combining the almost-sure local well-posedness of Theorem~\ref{Thm:ASLWP} with the norm-inflation result of Theorem~\ref{thm:norminf}, we obtain the following `almost sure norm inflation' phenomenon which highlights the strong instability in the map $\Phi$.
  This phenomenon is known to also occur for certain NLW equations in the super-critical regime~\cite{Xia, OOT}. 

\begin{theorem}[Almost sure norm inflation]\label{thm:asnorminf}
Let $\al\in (\tfrac 14, \tfrac 12]$ and fix $\o \in \Sigma$, where $\Sigma$ is the set of full $\prob$-measure from Theorem~\ref{Thm:ASLWP}.
Let $u^{\o}$ be the \textup(local\textup) solution to the BBM equation \eqref{BBM0} with $u^{\o}|_{t=0}=u_0^{\o}$. Then, given $k\in \N$, there exist $u^{\o}_{k}$ smooth \textup(random\textup) solutions to \eqref{BBM0} such that 
\begin{align*}
\lim_{k \ra \infty} \|u^{\o}_k(0)-u_0^{\o}\|_{H^{\al-\frac 12-}(\T)}=0 \quad \text{and} \quad \lim_{k\ra \infty}\|u^{\o}_k -u^{\o}\|_{C([0,k^{-1}]; H^{\al-\frac 12-}(\T))}=\infty.
\end{align*}
\end{theorem}

The almost sure norm inflation above implies that the solution map $\Phi$ is almost surely discontinuous everywhere over $H^{\al-\frac 12 - }(\T)$.  In other words, we can always find a (random) smooth sequence $\{u_{0,k}^{\o}\}_{k\in \N}$ which approximates the realisation $u_0^{\o}$ but whose (smooth) solutions exhibit the instability as stated in Theorem~\ref{thm:asnorminf}. 
However, this does not rule out the possibility that there is \emph{some} class of reasonable smooth solutions which do approximate the random solutions lying below $L^{2}(\T)$. 
Indeed, the class of smooth solutions obtained from mollified data provides a good approximation property.

\begin{theorem}\label{thm:smoothapp} 
Let $\al \in (\tfrac 14, \tfrac 12]$. Let $u=u^{\o}$ be as in Theorem~\ref{thm:asnorminf}. Denote by $u_{0,k}^{\o}=\rho_{k} \ast u_0^{\o}$ the regularisation of $u_0^{\o}$ by a smooth mollifier $\{ \rho_{k}\}_{k\in \N}$ and let $u_k$ be the solution to the BBM equation \eqref{BBMD} with $u_{k}|_{t=0}=u_{0,k}^{\o}$. Then, we have 
\begin{align*}
\lim_{k \ra \infty} \|u_{0, k}^{\o}-u_0^{\o}\|_{H^{\al-\frac 12-}(\T)}=0 \qquad \text{and} \qquad \lim_{k\ra \infty}\|u^{\o}_k -u^{\o}\|_{C([0,T^{\o}]; H^{\al-\frac 12-}(\T))}=0.
\end{align*} 
Moreover, the limit $u$ is independent of the choice of mollification kernel $\rho$.
\end{theorem}

The proof of Theorem~\ref{thm:smoothapp} is a direct corollary of Theorem~\ref{Thm:ASLWP}, Remark~\ref{remark:cty} and Proposition~\ref{prop:stochobjects}. Namely, Proposition~\ref{prop:stochobjects} implies the almost sure convergence of the enhanced data set $(z_{k}, \NN(z_{k}))$ to $(z,\NN(z))$ (in the appropriate topology, see~\eqref{Psi}) with the limit independent of the mollification kernel.  Then, continuity of the map $\Psi$ given in \eqref{Psi} (from Theorem~\ref{Thm:ASLWP}; see also Remark~\ref{remark:cty}) implies $v_k \ra v$ almost surely in $C([0,T]; H^{s}(\T))$ for any $s<2\al$. 

In Theorem~\ref{thm:smoothapp}, the independence of the limit on the choice of mollifying kernel provides a well-defined notion of stability for the solution map $\Phi$. Thus, the random solutions constructed by Theorem~\ref{Thm:ASLWP} may be approximated by certain `reasonable' regularisations of the initial data. This is in direct contrast to the setting of deterministic well-posedness where approximations may be completely arbitrary.

\begin{remark}\rm
As will be clear by the proof of Proposition~\ref{prop:stochobjects} below, we may also consider in Theorem~\ref{thm:smoothapp} the regularisation by the (non-smooth) Dirichlet projection $\P_{\leq N}$ onto frequencies $\{n : \, |n|\leq N\}$. We then extend the uniqueness of the limiting solution among the class of mollifiers and  the projection $\P_{\leq N}$.
\end{remark}

\begin{remark}\rm 
As Theorem~\ref{thm:asnorminf} is a direct corollary of Theorem~\ref{thm:norminf}, the sequence of solutions $\{ u_{k}^{\o}\}_{k\in \N}$ can be taken with respect to a continuous index: for fixed `good' $\o$, there is a sequence of smooth solutions $\{ u_{\ep}^{\o}\}_{0<\ep \ll 1}$ which exhibit the instability as stated in Theorem~\ref{thm:asnorminf} as $\ep \to 0$. With a sequence of mollifiers $\rho_{\ep}(x):=\ep^{-1}\rho(\ep^{-1}x)$ on $\T$ and for $\al \in(\tfrac 14, \tfrac 12]$, we can show that the smooth solutions $u_{\ep}$ to BBM~\eqref{BBMD} with initial data $u_{0}^{\o}\ast \rho_{\ep}$, where $u_0^{\o}$ is as in \eqref{initialdata}, converge in $C([0,T^{\o}];H^{\al-\frac{1}{2}-}(\T))$ in probability to a unique limit $u$ as $\ep \to 0$, where $T^{\o}>0$ almost surely. Moreover, the limit $u$ is independent of the choice of the mollifier $\rho$. 
\end{remark}

We now provide an outline of the following paper. In Section 2, we collect some necessary results and tools of deterministic and probabilistic natures. We then carefully construct and study properties of the random linear solutions $z$ and the nonlinear object $\NN(z)$. In Section 3, we prove Theorem~\ref{Thm:ASLWP} on constructing local-in-time solutions below $L^{2}(\T)$ and show that the regularity result there is sharp. We then show, in Section 4, we can globalise those random solutions (in a restricted regularity). In the purely deterministic setting, we establish norm inflation for BBM at general data in negative Sobolev spaces in Section 5. In Appendix~\ref{app:Gauss}, we detail how we obtain almost sure existence of solutions with non-Gaussian random initial data as in Remark~\ref{remark: gaussianity}.
Finally, we also include a brief appendix on obtaining exponential tail estimates on stochastic processes.

\section{Deterministic and probabilistic tools}\label{sect:lemmas}
In this section, we collect here some useful deterministic and probabilistic results. 
\subsection{Deterministic tools}

First, we recall the following key bilinear estimate, due to Bona and Tzvetkov \cite{BonaTzvet} (see also Roum\'egoux~\cite{Roumegoux}), which immediately implies the local well-posedness of BBM~\eqref{BBM0} in $H^{s}(\M)$ for any $s\geq 0$. A proof is contained within that of a slightly more general bilinear estimate which we give in Lemma \ref{lemma:flspineq}. 

\begin{lemma}[\cite{BonaTzvet, Roumegoux}]\label{lemma:bonatzvet}
For any $s\geq 0$ and any $f,g\in H^{s}(\M)$, we have \begin{equation}
\| \varphi(D_{x})(fg)\|_{H^{s}(\M)} \lesssim \|f\|_{H^{s}(\M)}\|g\|_{H^{s}(\M)}. \label{productestimateBT}
\end{equation} Furthermore, the estimate \eqref{productestimateBT} is false if $s<0$.
\end{lemma}

We next state a useful summing estimate, a proof of which can be found in, for example,~\cite[Lemma 4.2]{GTV}.

\begin{lemma} \label{lemma:sumestimate} If $\beta\geq \gamma \geq 0$ and $\beta+\gamma >1$, then $$ \sum_{n}\frac{1}{\jb{n-k_{1}}^{\beta}\jb{n-k_{2}}^{\gamma}}\lesssim \frac{\phi_{\beta}(k_{1}-k_{2})}{\jb{k_{1}-k_{2}}^{\gamma}},$$  where 
\begin{align}
 \phi_{\beta}(k):=\sum_{|n|\leq |k|}\frac{1}{\jb{n}^{\beta}}\sim \begin{cases} 
1, \quad  &\textup{if } \,\,\beta>1, \\ 
\log(1+\jb{k}), \quad& \textup{if } \,\, \beta=1, \\
 \jb{k}^{1-\beta}, \quad & \textup{if } \,\, \beta<1.
\end{cases} \label{phi}
\end{align}
\end{lemma}

Finally, we need the following paraproduct estimate:

\begin{lemma} [{\cite[Lemma 3.4]{GKO}}] \label{lemma:negderivprod} Let $0\leq s\leq 1$ and suppose that $1<p,q,r<\infty$ satisfy $\frac{1}{p}+\frac{1}{q}=\frac{1}{r}+s$. Then, we have 
\begin{align*}
\| \jb{\nabla}^{-s}(fg)\|_{L^{r}(\T)}\lesssim \| \jb{\nabla}^{-s}f\|_{L^{p}(\T)}\| \jb{\nabla}^{s}g\|_{L^{q}(\T)}. 
\end{align*}
\end{lemma}

\subsection{Probabilistic tools}
Let $\{g_{n}\}_{n\in \N}$ be a sequence of independent standard Gaussian random variables on a probability space $(\O,\mathcal{F},\prob)$, where $\mathcal{F}$ is the $\s$-algebra generated by the sequence $\{g_{n}\}_{n\in \N}$. Given $\l\in \N\cup\{0\}=\N_{0}$, we define a polynomial chaos of degree $k$ to be a polynomial of the form 
$\prod_{j=1}^{\infty} H_{\l_j}(g_{n})$,
where $\{\l_{j}\}_{j\in \N_{0} } \subset \N_{0}$ satisfy\footnote{Note at most finitely many $\l_{j}$ are non-zero.} $\l=\sum_{j=1}^{\infty}\l_{j}$ and $H_{\l_j}$ is the Hermite polynomial of degree $\l_j$. 
 We then define the homogeneous Wiener chaos $\mathcal{H}_{\l}$ of order $\l$ as the closure under $L^{2}(\O,\mathcal{F},\prob)$ of the linear span of polynomial chaoses of degree $\l$. We write 
 \begin{align*}
\mathcal{H}_{\leq \l}:=\bigoplus_{j=0}^{\l} \mathcal{H}_{j}
\end{align*}
and we have the following so-called Wiener chaos estimate which we use to exploit the randomisation in the multilinear term $\NN(z)$.

\begin{lemma}[Wiener chaos estimate]\label{Lemma:wienerchaos}  
Given $\l\in \N_{0}$, let $X\in \mathcal{H}_{\leq \l}$. Then, we have 
\begin{align*}
\|X\|_{L^{p}(\O)} \leq (p-1)^{\frac{\l}{2}} \|X\|_{L^{2}(\O)}
\end{align*}
for any $2\leq p<\infty$. 
\end{lemma}

The proof of Lemma~\ref{Lemma:wienerchaos} follows from the hypercontractivity of the Ornstein-Uhlenbeck semigroup due to Nelson~\cite{Nelson}. See for instance \cite[Proposition 2.4]{ThomTzvet} for more details. 
Notice as a special case of Lemma~\ref{Lemma:wienerchaos}, for any $(a_{n})\in \l^{2}_{n}$, we have
 \begin{equation} \bigg\| \sum_{n\in \Z}a_{n}g_{n}(\o)  \bigg\|_{L^{p}(\Omega)} \lesssim p^{\frac{1}{2}}\|a_{n}\|_{\l^{2}_{n}}. \label{wienerchaos}
\end{equation}

In order to study the regularity properties of random distributions we make use of the following result, a proof of which can be found in \cite{OPT}.
We specialise the argument in \cite{OPT} to one spatial dimension and to the Sobolev spaces $W^{s,p}(\T)$. 
Given $0<\g<1$, we say $f\in C^{\g}([0,T]; W^{s,p}_{x}(\T))$ if the following norm is finite:
\begin{align*}
\|f\|_{C^{\g}([0,T]; W^{s,p}_{x}(\T))}=\|f\|_{C([0,T]; W^{s,p}_{x}(\T))}+\|f\|_{\dot{C}^{\g}([0,T];W^{s,p}_{x}(\T))},
\end{align*}
where we define the semi-norm $ \|\cdot \|_{\dot{C}^{\g}([0,T];W^{s,p}_{x}(\T))}$ by 
\begin{align*}
\|f \|_{\dot{C}^{\g}([0,T];W^{s,p}_{x}(\T))}:=\sup_{0\leq t'<t\leq T}\frac{ \| f(t)-f(t')\|_{W^{s,p}_{x}(\T)}}{|t-t'|^{\g}}.
\end{align*}
We occasionally write $C([0,T];W^{s,p}(\T))$ as $C_{T}W^{s,p}$.
Given $h\in \R$, we denote by $\dl_{h}$ the difference operator:
\begin{align*}
\dl_h X(t) = X(t+h)-X(t).
\end{align*}

\begin{proposition}[Regularity and convergence of stochastic processes] \label{prop:reg} 
Let $\{ X_k \}_{k\in \N}$ be a sequence of stochastic processes on $\R_{+}$ with values in $\S'(\T)$ such that $X_k (t)\in \mathcal{H}_{\leq \l}$ for each $t\in \R_{+}$ and $k\in \N$. Fix $T>0$. Suppose there exist $1\leq p\leq \infty$ and $s\in \R$ such that the following statements hold:
\begin{enumerate}[\normalfont (i)]

\item We have 
\begin{align}
\E[ \| X_k (t)\|_{W^{s,p}(\T)}^{q}]< \infty, \label{251}
\end{align}
for any $q\geq 1$ and uniformly in $t\in [0,T]$ and $k\in \N$.

\item There exists $\ta>0$ such that 
\begin{align}
\E [ \|X_{k'}(t)-X_k (t) \|_{W^{s,p}(\T)}^{q}] \les_{T} k^{-q\ta }q^{\frac{\l q}{2}} \label{252}
\end{align}
for any $q\geq 1$, $t\in [0,T]$ and $k'\geq k\geq 1$.

\end{enumerate}

Then, \textup{(i)} implies for each fixed $t\in [0,T]$, $X_{k}(t)\in W^{s,p}(\T)$ almost surely and moreover, \textup{(ii)} implies
there exists $X(t)\in W^{s,p}(\T)$ such that $\{ X_k (t,\cdot)\}_{k\in \N}$ converges to $X(t)$ in $L^{q}(\O; W^{s,p}(\T))$, for any $1\leq q<\infty$ and almost surely in $W^{s,p}(\T)$.

Suppose, in addition, the following statements hold:

\begin{itemize}

\item[(iii)] There exists $\g>0$ such that 
\begin{align}
\E [ \| \dl_h X_k (t)\|_{W^{s,p}(\T)}^{q}] \les_{q,T} |h|^{\frac{q}{2}\g} \label{253}
\end{align}
for any $q\geq 1$ and $h\in [-1,1]$, uniformly in $t\in [0,T]$ and $k\in \N$. 

\item[(iv)] There exist $\ta,\g >0$, such that 
\begin{align}
\E [ \| \dl_h X_{k'}(t)-\dl_h X_k (t)\|_{W^{s,p}(\T)}^{q}] \les_{q,T} k^{-q\ta }|h|^{\frac{q}{2}\g} \label{254}
\end{align}
for any $q\geq 1$ and $h\in [-1,1]$, uniformly in $t\in [0,T]$ and $k'\geq k\geq 1$. 

\end{itemize}

Then, \textup{(iii)} implies $X_{k}, X \in C^{\beta}([0,T]; W^{s,p}(\T))$ almost surely for $\beta<\frac{\g}{2}$ and moreover, \textup{(iv)} implies 
$\{ X_k \}_{k\in \N}$ converges to $X$ in $L^{q}(\O; C^{\beta}([0,T]; W^{s,p}(\T)))$, for any $1\leq q<\infty$ and almost surely in $C^{\beta}([0,T];W^{s,p}(\T))$.

\end{proposition}

\subsection{Properties of the stochastic objects}

In this section we study the regularity and integrability properties of the random linear solution to BBM~\eqref{BBM0} with initial data \eqref{initialdata}, which we write as $z=S(t)u_{0}^{\o}$ and the bilinear term $\NN(z) $ given in \eqref{nonlinz}.
We verify that both $z$ and $\NN(z)$ are the limit of the mollified sequences $\{ z_{k}=S(t)(u_0^{\o}\ast \rho_{k})\}_{k\in \N}$ and $\{ \NN(z_{k})\}_{k\in \N}$, independent of the choice of mollifier. Moreover, we verify 
\begin{align*}
\NN(z)\in C([0,T];H^{2\al -}(\T))
\end{align*}
almost surely.

\begin{proposition}\label{prop:stochobjects} 
Let $\frac 14 < \al \leq \frac 12$,
\begin{align*}
s_1 <\al-\frac 12 \quad \text{and} \quad s_2<2\al.
\end{align*}
 Let $\{ \rho_k \}_{k\in \N}$ be a family of mollifiers on $\T$. Given $T>0$, let $z_k =S(t)( u_0^{\o}\ast \rho_k)$
where $t\in [0,T]$.
Then, $(z, \NN(z))\in C([0,T]; W^{s_1,\infty}(\T)) \times C([0,T]; W^{s_2,\infty}(\T))$ almost surely and  
\begin{align*}
(z_k, \NN(z_{k})) \longrightarrow (z, \NN(z)),
\end{align*}
as $k \ra \infty$ in $L^{q}(\O;  C([0,T]; W^{s_1,\infty}(\T))\times C([0,T]; W^{s_2,\infty}(\T)))$ for any $q\geq 1$ and almost surely in $ C([0,T]; W^{s_1,\infty}(\T))\times C([0,T]; W^{s_2,\infty}(\T))$. Moreover, the limit $(z,\NN(z))$ is independent of the choice of mollification kernel $\rho$, including the regularisation by the (non-smooth) Dirichlet projection $\P_{\leq N}$.
Furthermore, there exist $C,C',c>0$ such that 
\begin{align*}
\prob \big(\|z\|_{C([0,T]; W^{s_1,\infty}_{x}(\T))} +\|\NN(z)\|_{C([0,T]; W^{s_2,\infty}_{x}(\T))}  >\lambda \big) \leq C' \big(e^{-\frac{C\lambda}{T^{c}}} +e^{-C\ld} \big)
\end{align*}
 for any $T>0$ and $\lambda >0$. 
\end{proposition}

\begin{proof} 
We first verify the claims made for the random linear solution $z$. Clearly, $\{ z_{k}\}_{k\in \N} \subset \mathcal{H}_{\leq 1}$ for every fixed $t\in [0,T]$.
We have
\begin{align*}
 \ft{z_{k\hphantom{'}}}(n,t)= \frac{g_n(\o)\ft{\rho_{k\hphantom{'}}}(n) e^{-it\vp(n)}}{\jb{n}^{\al}}.
\end{align*} 
Then, by Sobolev embedding, Minkowski's integral inequality (for $q$ sufficiently large) and \eqref{Lemma:wienerchaos}, we have
\begin{align}
\begin{split} \label{zk1}
\E[ \| z_{k}(t)\|_{W^{s_1,\infty}(\T)}^{q}] & \leq \bigg\|  \big\| \jb{\dx}^{s_1-\ep}z_{k}(t) \big\|_{L^{q}(\O)}   \bigg\|_{L^{p}(\T)}^{q} \\
& \les q^{\frac{q}{2}} \bigg\| \| \jb{\dx}^{s_1-\ep} z_{k}(t) \|_{L^{2}(\O)}   \bigg\|_{L^{p}(\T)}^{q} \\
& \les q^{\frac q2} \bigg\| \bigg( \sum_{n\in \Z} \frac{ \jb{n}^{2s_1-2\ep} |\ft{\rho_{k\hphantom{'}}}(n) |^{2} }{\jb{n}^{2\al}} \bigg)^{\frac{1}{2}} \bigg\|_{L^{p}(\T)}^{q},
\end{split} 
\end{align}
where $0<\ep:=\ep(p,q)\ll 1$.
As $|\ft{\rho_{k\hphantom{'}}}(n)|\les 1$ uniformly in both $n\in \Z$ and $k\in \N$ and $s_1<\al-\frac 12$, we verfiy \eqref{251} and hence $z_k(t)\in W^{s_1,\infty}(\T)$ almost surely.
From the same computations as in \eqref{zk1}, we have 
\begin{align}
\E[ \|z_{k'}(t)-z_{k}(t)\|_{W^{s_1,\infty}(\T)}^{q}] \les q^{\frac{q}{2}} \bigg(\sum_{n\in \Z} \frac{ \jb{n}^{2s_1-2\ep}|\ft{\rho_{k'\hphantom{'}}}(n)-\ft{\rho_{k\hphantom{'}}}(n)|^{2} }{\jb{n}^{2\al}} \bigg)^{\frac{q}{2}}.
\label{zk2} 
\end{align}
Now, given $k' \geq k > 0$, the mean value theorem implies 
\begin{align*}
|\ft{\rho_{k'\hphantom{'}}}(n)-\ft{\rho_{k\hphantom{'}}}(n)|\les |n|| (k')^{-1}-k^{-1}|\leq 2|n| k^{-1}.
\end{align*}
Interpolating this with the trivial bound  $|\ft{\rho_{k\hphantom{'}}}(n)-\ft{\rho_{k'}}(n)|\les 2$, we obtain
\begin{align}
|\ft{\rho_{k'\hphantom{'}}}(n)-\ft{\rho_{k\hphantom{'}}}(n)|\les |n|^{\ta}k^{-\ta} \label{rhomvt}
\end{align}
for $0\leq \ta\leq 1$. Inserting \eqref{rhomvt} into \eqref{zk2} we get 
\begin{align*}
\E[ \|z_{k'}(t)-z_{k}(t)\|_{W^{s_1,\infty}(\T)}^{q}] \les k^{-q\ta }q^{\frac{q}{2}},
\end{align*} 
provided $s_1<\al-\frac{1}{2}-\ta$.
As $\ta >0$ was arbitrary, we have verified (i) and (ii) of Proposition~\ref{prop:reg} with $s_1<\al-\frac 12$. We now move onto establishing the temporal regularity of $z_{k}$. 
We verify the appropriate analogue of \eqref{254} since the same ideas will apply to obtain \eqref{253}.
Analogously to \eqref{zk1}, we have 
\begin{align*}
\E [ \| \dl_h z_{k'}(t)-& \dl_h z_k (t)\|_{W^{s_1,\infty}(\T)}^{q}]  \\
& \les q^{\frac{q}{2}} \bigg(\sum_{n\in \Z} \frac{ \jb{n}^{2s_1-2\ep}|\ft{\rho_{k'\hphantom{'}}}(n)-\ft{\rho_{k\hphantom{'}}}(n)|^{2} |e^{-ih\vp(n)}-1|^{2} }{\jb{n}^{2\al}} \bigg)^{\frac{q}{2}}.
\end{align*}
Using \eqref{rhomvt} and $|e^{-ih\vp(n)}-1|\les |h| |\vp(n)|$, we then obtain 
\begin{align*}
\E [ \| \dl_h z_{k'}(t)-\dl_h z_k (t)\|_{W^{s_1,\infty}(\T)}^{q}] \les q^{\frac{q}{2}}k^{-q\ta}|h|^{\frac{q}{2}}.
\end{align*}
Thus by Proposition~\ref{prop:reg}, $z_k$ converges almost surely to $z$ in $C_{T}W^{s_1,\infty}(\T)$ and in $L^{q}(\O; C_T W^{s_1,\infty}(\T))$ for any $q\geq 1$. 
We verify the independence of $z$ on the mollifier $\rho$ later in this proof; see \eqref{zindep}.
Taking a limit in the analogue of \eqref{253} as $k \ra \infty$ gives 
\begin{align}
\E [ \|z(t)-z(t')\|_{W^{s_1,\infty}(\T)}^{q} ] \les q^{\frac q2}|t-t'|^q. \label{zdiff}
\end{align}
With $\g<1-\frac 1q$, we have 
\begin{align}
\|z\|_{C_T  W^{s_1,\infty}_{x}}\leq T^{\g} \|z \|_{\dot{C}^{\g}([0,T];W^{s_1,\infty}_{x}(\T))} +\|z(0)\|_{W^{s_1,\infty}_{x}}. \label{BQ}
\end{align}
Therefore by \eqref{exptail} in Appendix~\ref{app:Garsia}, we have 
\begin{align}
\begin{split}
\prob \big(\|z\|_{C([0,T]; W^{s_1,\infty}_{x}(\T))} >\lambda \big) \leq& \, \prob \big(\|z\|_{\dot{C}^{\g}([0,T];W^{s_1,\infty}_{x}(\T))} >\tfrac{\lambda}{2} \big) \\
& +\prob \big(\|z(0)\|_{ W^{s_1,\infty}_{x}(\T)} >\tfrac{\lambda}{2} \big) \\
 \leq &\, e^{-C\frac{\lambda^{2}}{T^{c}}}+e^{-c\ld^{2}}. \label{BQ2}
\end{split}
\end{align}

We now consider the object $\NN(z)$. For any $k\in \N$ and fixed $t>0$, $\NN(z_{k})\in \mathcal{H}_{\leq 2}$.
We verify appropriate versions of \eqref{252} and \eqref{253}, which themselves contain the necessary calculations required to also obtain versions of \eqref{251} and \eqref{254}. 
We write 
\begin{align*}
\jb{\dx}^{s_2-\ep}\NN(z_{k})=\sum_{n_1,n_2\in \Z} \mathcal{R}_{k}(n_1,n_2;t)g_{n_1}g_{n_2},
\end{align*}
where 
\begin{align}
\begin{split} \label{Rkdefn}
\mathcal{R}_{k}(n_1,n_2;t)& := \ind_{\{ n_1+n_2\neq 0\}}\jb{n_1+n_2}^{s_2-\ep}\vp(n_1+n_2)e^{i(n_1+n_2)x}\prod_{j=1}^{2}a_{k}(n_j;t), \\
a_{k}(n;t)& :=\frac{e^{-it\vp(n)}}{\jb{n}^{\al}}\ft{\rho_{k\hphantom{'}}}(n) .
\end{split}
\end{align}
By Sobolev embedding, Minkowski's integral inequality (for $q$ sufficiently large) and Lemma~\ref{Lemma:wienerchaos}, we have
\begin{align*}
\E[ \| \NN(z_{k'})(t)&-\NN(z_{k})(t)\|_{W^{s_2,\infty}(\T)}^{q}]  \\
& \leq q^{q} \,\Bigg\|  \bigg\|\sum_{n_1,n_2\in \Z} \big[\mathcal{R}_{k'}-\mathcal{R}_{k}\big](n_1,n_2;t) g_{n_1}g_{n_2} \bigg\|_{L^{2}(\O)}   \Bigg\|_{L^{p}(\T)}^{q}.
\end{align*}
It suffices to show 
\begin{align}
\bigg\|\sum_{n_1,n_2\in \Z} \big[\mathcal{R}_{k'}-\mathcal{R}_{k}\big](n_1,n_2;t) g_{n_1}g_{n_2} \bigg\|_{L^{2}(\O)} \les k^{-\ta} 
\label{z2suff}
\end{align}
for some $\ta>0$, any $t\in [0,T]$ and $k'\geq k\geq 1$.
Now
\begin{align}
\begin{split}\label{z21}
 \bigg\|&\sum_{n_1,n_2\in \Z} \big[\mathcal{R}_{k'}-\mathcal{R}_{k}\big](n_1,n_2;t) g_{n_1}g_{n_2} \bigg\|_{L^{2}(\O)}^{2} \\
 & = \sum_{ \substack{n_1.n_2\in \Z \\ m_1,m_2\in \Z}} \big[\mathcal{R}_{k'}-\mathcal{R}_{k}\big](n_1,n_2;t) \cj{\big[\mathcal{R}_{k'}-\mathcal{R}_{k}\big](m_1,m_2;t)}\E[g_{n_1}g_{n_2}\cj{g_{m_1}g_{m_2}}].
 \end{split}
\end{align}
First, we assume all of $n_1,n_2,m_1$ and $m_2$ are non-zero and not all equal. Then, Wick's theorem implies
 \begin{align*}
\E[g_{n_1}g_{n_2}\cj{g_{m_1}g_{m_2}}]= \E[g_{n_1}\cj{g_{m_1}}]\E[g_{n_2}\cj{g_{m_2}}] &+ \E[g_{n_1}\cj{g_{m_2}}]\E[g_{n_2}\cj{g_{m_1}}]\\
&+ \E[g_{n_1}g_{n_2}]\E[\cj{g_{m_1}}\cj{g_{m_2}}]. 
\end{align*}
The first two terms are non-zero if and only if $n_j=m_{\sigma(j)}$, where $\sigma$ is a permutation of $\{1,2\}$. The third term vanishes identically since $n_1 +n_2 \neq 0$ and $m_1 +m_2 \neq 0$. This is precisely where we use the definition of the product \eqref{nonlin} which does not contain the zero frequency. 
Therefore, we have
\begin{align}
\text{LHS of} \,\, \eqref{z21}  & =\sum_{n_1,n_2\in \Z} |\mathcal{R}_{k'}(n_1,n_2;t)-\mathcal{R}_{k}(n_1,n_2;t)|^{2} \notag   \\
&    \sim \sum_{n\in\Z} \jb{n}^{2s_2-2\ep-2} \sum_{\substack{n_1,n_2\in\Z \\ n=n_1+n_2 } }\frac{|\ft{\rho_{k'}}(n_1)\ft{\rho_{k'}}(n_2)-  \ft{\rho_{k\hphantom{'}}}(n_1)\ft{\rho_{k\hphantom{'}}}(n_2)|^2 }{\jb{n_1}^{2\al}\jb{n_2}^{2\al}}.     \label{Rdiff}
\end{align}
By the triangle inequality, we bound the inner summation in \eqref{Rdiff} by 
\begin{align*}
\sum_{\substack{n_1,n_2\in\Z \\ n=n_1+n_2 } } \frac{|\ft{\rho_{k\hphantom{'}}}(n_1)|^2 |\ft{\rho_{k'}}(n_2)-\ft{\rho_{k\hphantom{'}}}(n_2)|^{2}}{\jb{n_1}^{2\al}\jb{n_2}^{2\al}}
+\sum_{\substack{n_1,n_2\in\Z \\ n=n_1+n_2 } } \frac{|\ft{\rho_{k\hphantom{'}}}(n_1)|^2 |\ft{\rho_{k\hphantom{'}}}(n_2)-\ft{\rho_{k'}}(n_2)|^{2}}{\jb{n_1}^{2\al}\jb{n_2}^{2\al}}.
\end{align*}
It suffices to estimate just the first sum, with the same ideas applying for the second. By \eqref{rhomvt} and Lemma~\ref{lemma:sumestimate}, we get
\begin{align*}
\sum_{\substack{n_1,n_2\in\Z \\ n=n_1+n_2 } } \frac{|\ft{\rho_{k\hphantom{'}}}(n_1)|^2 |\ft{\rho_{k'}}(n_2)-\ft{\rho_{k\hphantom{'}}}(n_2)|^{2}}{\jb{n_1}^{2\al}\jb{n_2}^{2\al}} &\les \frac{1}{k^{2\ta}} \sum_{n_2 \in \Z} \frac{1}{ \jb{n-n_2}^{2\al}\jb{n_2}^{2\al-2\ta}}\\
& \les k^{-2\ta} \jb{n}^{-2\al+2\ta}\phi_{2\al}(n),
\end{align*} 
provided $4\al -2\ta>1$, where $\phi_{2\al}$ is given in \eqref{phi}. From this contribution, we arrive at 
\begin{align*}
\text{LHS of} \,\, \eqref{Rdiff}  & \les k^{-2\ta} \sum_{n\in \Z} \jb{n}^{2s_2-2\ep -2-2\al+2\ta}\phi_{2\al}(n) \les k^{-2\ta} 
\end{align*}
provided $s_2<2\al-\ta$.
Now we consider the case when at least one of $n_1,n_2,m_1$ or $m_2$ are zero. Noting that $|\mathcal{R}_{k}(n_1,n_2;t)|=|\mathcal{R}_{k}(n_2,n_1;t)|$ and $n_1+n_2\neq 0$, we may assume $n_1=0$. Then, the only non-zero contribution comes from when $m_1=0$ and $m_2=n_2$ (using the symmetry in $|\mathcal{R}_{k}(m_1,m_2;t)|$). Using $\ft{\rho_{k\hphantom{'}}}(0)=1$ and \eqref{rhomvt}, we have
\begin{align*}
\text{LHS of} \,\, \eqref{z21} &=\sum_{n_2} |\mathcal{R}_{k'}(0,n_2;t)-\mathcal{R}_{k}(0,n_2;t)|^{2} \\
& \sim \sum_{n_2} \jb{n_2}^{2s_2-2\ep-2-2\al}|\ft{\rho_{k'\hphantom{'}}}(n_2)-\ft{\rho_{k\hphantom{'}}}(n_2)|^2  \les k^{-2\ta},
\end{align*}
provided $s_2<\frac{1}{2}+\al-\ta$.
Finally, we have the contribution when $n_1=n_2=m_1=m_2$, which occurs only if $n_1+n_2, m_1+m_2\in 2\Z$. We then have
\begin{align*}
\text{LHS of} \,\, &\eqref{z21} \sim \sum_{n\in \Z} \big|\mathcal{R}_{k'}\big(\tfrac n2 ,\tfrac n2;t\big)-\mathcal{R}_{k}\big(\tfrac n2 ,\tfrac n2;t\big)\big|^{2}\les k^{-2\ta},
\end{align*}
provided $s_2<\frac{1}{2}+2\al-\ta$. This completes the proof of \eqref{z2suff}.
As $\ta>0$ is arbitrary, 
we conclude the limit $\NN(z)(t) \in W^{s_2,\infty}(\T)$ for any $s_2<2\al $ and for every fixed $t\in [0,T]$, provided $\al>\tfrac 14$. 

We now show 
\begin{align}
\E [ \|  \dl_{h}\NN(z_{k})(t)\|_{W^{s_2,\infty}}^{q}] \les q^{q}|h|^{q}. \label{z2tdiff}
\end{align}
As before, this reduces to proving 
\begin{align}
\bigg\|\sum_{n_1,n_2\in \Z} \dl_{h}\mathcal{R}_{k}(n_1,n_2;t) g_{n_1}g_{n_2} \bigg\|_{L^{2}(\O)} \les |h|. 
\label{z2tsuff}
\end{align}
for $h\in [-1,1]$, uniformly in $t\in [0,T]$ and $k\in \N$.
Expanding, we have
\begin{align*}
(\text{LHS of} \,\, \eqref{z2tsuff})^{2}
=   \sum_{ \substack{n_1.n_2\in \Z \\ m_1,m_2\in \Z}} \dl_{h}\mathcal{R}_{k}(n_1,n_2;t) \cj{\dl_{h}\mathcal{R}_{k}(m_1,m_2;t)} \E[g_{n_1}g_{n_2}\cj{g_{m_1}g_{m_2}}].
\end{align*}
Assuming all of $n_1,n_2, m_1$ and $m_2$ are non-zero and not equal, the expectation is non-zero only if 
$n_j=m_{\sigma(j)}$, where $\sigma$ is a permutation of $\{1,2\}$. In this case, we get 
\begin{align*}
(\text{LHS of} \,\, \eqref{z2tsuff})^{2} &= \sum_{ n_1,n_2\in \Z} |\dl_{h}\mathcal{R}_{k}(n_1,n_2;t)|^{2} \\
& \sim \sum_{n} \jb{n}^{2s_2-2\ep-2}\sum_{\substack{n_1,n_2\in\Z \\ n=n_1+n_2 }  }\frac{|\ft{\rho_{k\hphantom{'}}}(n_1)\ft{\rho_{k\hphantom{'}}}(n_2)|^2}{\jb{n_1}^{2\al}\jb{n_2}^{2\al}}|e^{-ih[\vp(n_1)+\vp(n_2)]}-1|^2 \\
& \les |h|^{2} \sum_{n} \jb{n}^{2s_2-2\ep-2} \sum_{\substack{n_1,n_2\in\Z \\ n=n_1+n_2 }  }\frac{|\vp(n_1)|^{2}+|\vp(n_2)|^{2}}{\jb{n_1}^{2\al}\jb{n_2}^{2\al}} \\
& \les |h|^2,
\end{align*}
provided $s<\frac{1}{2}+\al$. The remaining contributions to the expectation can be estimated in a similar fashion.
We conclude by Proposition~\ref{prop:reg} that $\NN(z_k)$ converges almost surely in $C_{T}W^{s_2,\infty}(\T)$ to $\NN(z)$ as long as $\al >\frac 14$.
Taking a limit as $k\ra \infty$ in \eqref{z2tdiff} implies 
\begin{align*}
\E [ \| \dl_{h}\NN(z)(t)\|_{W^{s_2,\infty}}^{q}] \les q^{q}|h|^{q}. 
\end{align*}
Applying the arguments in Appendix~\ref{app:Garsia} and a similar analysis as in~\eqref{BQ} and \eqref{BQ2},  we obtain  
\begin{align*}
\prob \big(\|\NN(z)\|_{C([0,T]; W^{s_2,\infty}_{x}(\T))} >\lambda \big) \leq e^{-C\frac{\lambda}{T^{c}}}+e^{-c\ld}.
\end{align*}

We now show that $\NN(z)$ is independent of the chosen mollifying kernel. Given two mollifying kernels $\{\rho_k \}_{k\in \N}$ and $\{ \eta_{_\l} \}_{\l\in \N}$, let 
\begin{align*}
\NN_{\rho}(z_{k}) = \NN(\rho_{k}\ast z),\\
\NN_{\eta}(z_{_\l}) = \NN(\eta_{_\l}\ast z).
\end{align*}
 Additionally, suppose there exist events of full probability $\O_{\rho}$ and $\O_{\eta}$ and random variables $\NN_{\rho}(z), \NN_{\eta}(z) \in C([0,T]; W^{s_2,\infty}(\T))$ almost surely, such that 
\begin{align}
\NN_{\rho}(z_{k}) \ra \NN_{\rho}(z)\,\,\, \text{in} \,\, C_{T}W^{s_2,\infty}(\T)\,\, \text{for every} \,\, \o\in \O_{\rho} \, \text{and}\label{convrho}\\
\NN_{\eta}(z_{\l}) \ra \NN_{\eta}(z) \,\,\, \text{in} \,\, C_{T}W^{s_2,\infty}(\T)\,\, \text{for every} \,\, \o\in \O_{\dl}, \notag
\end{align}
as $k,\l \ra \infty$. The goal is to show $\NN_{\rho}(z) \equiv \NN_{\eta}(z)$ almost surely (at least on $\O_{\rho}\cap \O_{\dl}$). We claim it suffices to establish the following difference estimate: there exists $\ta>0$ sufficiently small such that
\begin{align}
\E[ \| \NN_{\rho}(z_k)(t)- \NN_{\eta}(z_{_\l})(t)\|_{W^{s_2,\infty}(\T)}^{q}] \les q^{q} (k^{-\ta }+\l^{-\ta })^{q}, \label{z2diffmol}
\end{align}
uniformly in $t\in [0,T]$.
By \eqref{convrho}, taking $k \ra \infty$ gives
\begin{align*}
\E[ \| \NN_{\rho}(z)(t)- \NN_{\eta}(z_{_\l})(t)\|_{W^{s_2,\infty}(\T)}^{q}] \les q^{q} \l^{-\ta q}.
\end{align*}
Now taking $\l \ra \infty$ implies $\NN_{\eta}(z_{\dl})(t) \ra \NN_{\rho}(z)(t)$ in $L^{q}(\O; W^{s_2,\infty}(\T))$ and hence $\NN_{\rho}(z)(t)=\NN_{\eta}(z)(t)$ for each $t\in [0,T]$.  
As we have seen multiple times before, in order to prove \eqref{z2diffmol}, it suffices to prove 
\begin{align}
\bigg\|\sum_{n_1,n_2\in \Z} \big[\mathcal{R}^{\rho}_{k}-\mathcal{R}^{\eta}_{\l}\big](n_1,n_2;t) g_{n_1}g_{n_2} \bigg\|_{L^{2}(\O)} \les k^{-\ta} +\l^{-\ta}.
\label{z2molsuff}
\end{align}
uniformly in $t\in [0,T]$, where $\mathcal{R}_{k}^{\rho}$ and $\mathcal{R}_{\l}^{\eta}$ are defined as in \eqref{Rkdefn} with the mollifiers $\rho$ and $\eta$ inserted appropriately.  
We expand out to get 
\begin{align*}
(\text{LHS of} \,\, \eqref{z2molsuff})^{2}
=   \sum_{ \substack{n_1.n_2\in \Z \\ m_1,m_2\in \Z}} \big[\mathcal{R}^{\rho}_{k}-\mathcal{R}^{\eta}_{\l}\big](n_1,n_2;t) &\cj{\big[\mathcal{R}^{\rho}_{k}-\mathcal{R}^{\eta}_{\l}\big](m_1,m_2;t)} \\
&\times \E[g_{n_1}g_{n_2}\cj{g_{m_1}g_{m_2}}].
\end{align*}
We will just consider the case when $n_1=m_1$ and $n_2=m_2$ with remaining cases following by either symmetry or similar calculations. We have 
\begin{align}
(\text{LHS of} \,\, \eqref{z2molsuff})^{2} = \sum_{n} \jb{n}^{2s_2-2\ep-2} \sum_{\substack{n_1,n_2\in\Z \\ n=n_1+n_2 }  } \frac{|\ft{\rho_{k\hphantom{'}}}(n_1)\ft{\rho_{k\hphantom{'}}}(n_2)-\ft{\eta_{_{\l}\hphantom{'}}}(n_1)\ft{\eta_{_\l \hphantom{'}}}(n_2)|^{2}}{\jb{n_1}^{2\al}\jb{n_2}^{2\al}}. \label{z2indep1}
\end{align}
For fixed $n\in \Z$ and $k\in \N$, $\ft{\rho}(0)=1$ and the mean value theorem imply
\noi
 $|1-\ft{\rho_{k\hphantom{'}}}(n_2)|\les |n_2|k^{-1} $.
Interpolating this with the trivial bound $|1-\ft{\rho_{k\hphantom{'}}}(n_2)|\les 2$, gives
\begin{align}
|1-\ft{\rho_{k\hphantom{'}}}(n_2)| \les |n_2|^{\ta}k^{-\ta},  \label{interprho}
\end{align} 
for any $0\leq \ta \leq 1$. A similar bound to \eqref{interprho} is also true for the mollifier $\eta$ replacing the mollifier $\rho$.
The triangle inequality and \eqref{interprho} imply
\begin{align*}
|\ft{\rho_{k\hphantom{'}}}(n_1)\ft{\rho_{k\hphantom{'}}}(n_2)-\ft{\eta_{_{\l}\hphantom{'}}}(n_1)\ft{\eta_{_\l \hphantom{'}}}(n_2)|^{2}& \les \sum_{j=1}^{2} |\ft{\rho_{k\hphantom{'}}}(n_j)-1|^{2}+ |\ft{\eta_{_{\l}\hphantom{'}}}(n_j)-1|^{2} \\
& \les (|n_1|^{2\ta}+|n_2|^{2\ta})(k^{-2\ta}+\l^{-2\ta}).
\end{align*}
Inserting this into \eqref{z2indep1} and applying Lemma~\ref{lemma:sumestimate} yields \eqref{z2molsuff}.
Finally, as described above, to show $z$ is independent of the choice of mollifier $\rho$, it suffices to obtain the estimate 
\begin{align}
\E[ \|z_{k,\rho}(t)- z_{\l,\eta}(t)\|_{W^{s_1,\infty}(\T)}^{q}] \les q^{q} (k^{-\ta }+\l^{-\ta })^{q}, \label{zindep}
\end{align}
for any $0< \ta \leq 1$. This follows easily using similar analysis as above and is thus omitted.

 \end{proof}

\section{Probabilistic local theory on $\T$}\label{section:localt}

\subsection{Proof of Theorem~\ref{Thm:ASLWP}} \label{Sec:local}
Our goal in this section will be to obtain the local theory as in Theorem~\ref{Thm:ASLWP}. 
For the purposes of iteration of the probabilistic local theory to a global result, we consider first the deterministic perturbed initial value problem:
\begin{align}
\begin{cases}
i\dt v=  \vp(D_x)(v + \frac{1}{2}v^{2}+z_{1}v) + \frac{1}{2}z_{2}, \\
v|_{t = t_0} =v_{0},
\end{cases}
\label{BBMperturb}
\end{align}
 with initial data $v_0 \in H^{s}(\T)$, $s\in (0,1)$ and under the following assumptions on the forcings $(z_1,z_2)$:
\begin{align}
z_{1}\in  C_{t,\text{loc}}W_{x}^{s_1 ,\infty}(\R\times \T) \qquad \text{and} \qquad z_{2}\in C_{t,\text{loc}}H^{s}_{x}(\R\times \T), \label{assumptions}
\end{align}
where $s_1<0$.

\begin{proposition}\label{prop:detlocal} Fix $-\tfrac 12 <s_1<0$ and let 
\begin{align}
\textup{(i)} \, s\geq -s_1 \,\, \textup{when} \,\,s\in \big(0,\tfrac 12 \big]   \quad \textup{or} \quad \textup{(ii)} \,\, s\leq 1+s_1 \,\, \textup{when} \,\,s\in \big(\tfrac 12, 1 \big). \label{ss1}
\end{align}
 Then, there exists a constant $C>0$ such that for every time interval $I=[t_0,t_1]$ of size $1$, every $L\geq 1$,  every $v_0 \in H^{s}(\T)$, and every pair of forcings $(z_1,z_2)$ satisfying \eqref{assumptions} and such that 
 \begin{align}
\|v_{0}\|_{H^{s}_{x}}+\|z_{1}\|_{C_{t}W^{s_1,\infty}_{x}(I\times \T)}+\|z_{2}\|_{C_{t}H^{s}_{x}(I\times \T)}\leq L, \label{Kcond}
\end{align}
there exists a unique solution $v\in C_{t}H^{s}_{x}([t_0, t_0+C^{-1}L^{-1}]\times \T)$ to \eqref{BBMperturb}.
\end{proposition}

\begin{proof}
For simplicity, we assume $I=[0,1]$. We fix $0<T\leq 1$ to be chosen later. 
We will construct $v$ as a fixed point of the operator
\begin{align}
\Gamma v(t):=&S(t)v_0 \notag \\
&-\frac{i}{2}\int_{0}^{t}S(t-t')\varphi(D_{x})v^{2}(t')\, dt' \label{vv}\\ 
& -i\int_{0}^{t}S(t-t')\varphi(D_{x})(z_{1}v)(t')\, dt' \label{vz}
\\ & -\frac{i}{2}\int_{0}^{t}S(t-t')z_{2}(t') \, dt'. \label{zz}
\end{align}
in the ball  
\begin{align*}
B_{R}:=\{ v\in L^{\infty}([0,T];H^{s}(\T))\,:\, \|v\|_{L^{\infty}([0,T];H^{s})}\leq R\},
\end{align*}
with $R>0$ also to be chosen later. 
By the unitarity of the linear operator $S(t)$ on $H^{s}$ we have 
\begin{align*}
\|S(t)v_0\|_{L^{\infty}_{T}H^{s}}=\|v_0\|_{H^{s}}.
\end{align*}
We estimate each of \eqref{vv}, \eqref{vz} and \eqref{zz} separately. 

\eqref{vv}: By Minkowski's inequality, unitarity of $S(t)$ on $H^{s}$ and Lemma~\ref{productestimateBT}, we have 
 $$\|\eqref{vv}\|_{L^{\infty}_{T}H^{s}_{x}}\les T\|v\|_{L^{\infty}_{T}H^{s}}^{2}.$$

\eqref{zz}:
From \eqref{Kcond}, we have 
\begin{align*} 
\bigg\|\int_{0}^{t} S(t-t')z_{2}(t')dt' \bigg\|_{L^{\infty}_{T}H^{s}_{x}} \les T\|z_{2}\|_{L^{\infty}_T H^{s}} \les TL.
\end{align*}

\eqref{vz}: 
Consider first when $s\in (0,\frac{1}{2}]$. By Lemma~\ref{lemma:negderivprod}, we have \begin{align*}
\| \varphi(D_{x})(z_1 v)\|_{H^{s}(\T)} & = \| \jb{\dx}^{s}\varphi(D_{x})(z_1 v)\|_{L^{2}(\T)} \\
 & \lesssim \| \jb{\dx}^{-(1-s)}(z_1 v)\|_{L^{2}(\T)}  \\
 & \lesssim \| \jb{\dx}^{-s}(z_{1}v)\|_{L^{2}(\T)} \\
 & \lesssim \| \jb{\dx}^{-s}z_1\|_{L^{\frac{1}{s}}(\T)}\| \jb{\dx}^{s}v\|_{L^{2}(\T)}.
\end{align*}
Now, provided $s\geq -s_1$, \eqref{Kcond} implies
$\| \varphi(D_{x})(z_1 v)\|_{L^{\infty}_{T}H^{s}(\T)}\lesssim L\| v\|_{L_{T}^{\infty}H^{s}(\T)}.$

 For $s\in (\frac{1}{2},1)$, we apply Lemma~\ref{lemma:negderivprod} as follows:
\begin{align*}
\| \varphi(D_{x})(z_1 v)\|_{H^{s}(\T)} 
 & \lesssim \| \jb{\dx}^{-(1-s)}(z_1 v)\|_{L^{2}(\T)}  \\
 & \lesssim \| \jb{\dx}^{-(1-s)}z_1\|_{L^{\frac{1}{1-s}}(\T)}\| \jb{\dx}^{1-s}v\|_{L^{2}(\T)} \\
 & \les L\|v\|_{H^{s}(\T)}
\end{align*}
provided $s\leq 1+s_1$.

With $s$ satisfying \eqref{ss1}, we have shown
\begin{align}
 \|\Gamma v\|_{L^{\infty}_{T}H^{s}}\leq C\|v_0\|_{H^{s}}+CTL+CTLR+CTR^{2}, \label{fixed1}
\end{align} for any $v\in B_{R}$. 
Choosing $R=4CL$ and $T\leq \min(\tilde{C}^{-1}K^{-1},1)$, \eqref{fixed1} implies $\Gamma$ maps $B_{R}$ into itself for sufficiently small $T>0$. Similarly, given $v_{1},v_{2}\in B_{R}$ with $v_1 |_{t=0}=v_2 |_{t=0}=v_0$, in the same way as we estimated the terms \eqref{vv} and \eqref{vz} above, we obtain
\begin{align*}
\| \varphi(D_{x})\left[v_1^{2}+2z_1 v_1 -2z_1 v_2 -v_2^2\right] \|_{H^{s}} 
& \lesssim \| \varphi(D_{x})\left[(v_{1}-v_{2})(v_{1}+v_{2})\right] \|_{H^{s}} \\
& \,\,\,\,\,\,\,\, +\| \varphi(D_{x})\left[(v_{1}-v_{2})z_1\right] \|_{H^{s}} \\
& \lesssim (\|v_{1}\|_{H^{s}}+\|v_{2}\|_{H^{s}} +1)\|v_{1}-v_{2}\|_{H^{s}},
\end{align*} 
and by reducing $T>0$ if necessary, this implies
 $$\|\Gamma v_{1}-\Gamma v_{2}\|_{L^{\infty}_{T}H^{s}}\leq \frac{1}{2}\|v_{1}-v_{2}\|_{L^{\infty}_{T}H^{s}}.$$
Combining this with \eqref{fixed1} shows $\Gamma$ is a strict contraction from $B_{R}$ into itself and hence has a unique fixed point $v^{\omega}\in L^{\infty}_{T}H^{s}_{x}(\T)$ where $T \sim L^{-1}$. The continuity in time of $v$ now follows by the continuity in time of $z_1$ and $z_2$ and that if $v\in L^{\infty}_{T}H^{s}_{x}$, then 
\begin{align*}
 \int_{0}^{t}S(t-t')\vp(D_x)(v^{2}(t'))dt'\in C_{T}H^{s}_{x}.
\end{align*} We omit details. 
\end{proof}

We now apply Proposition~\ref{prop:detlocal} to prove Theorem~\ref{Thm:ASLWP}. 

\begin{proof}[Proof of Theorem~\ref{Thm:ASLWP}]
Let $z_{1}=z^{\o}$, where $z^{\o}$ solves the linear problem 
\begin{align*}
\begin{cases}
i\dt z^{\o}=  \vp(D_x)(z^{\o}) \\
z^{\o}|_{t = 0} =u_{0}^{\o},
\end{cases}
\end{align*}
with $u_{0}^{\o}$ given by \eqref{initialdata}, and $z_{2}=\NN(z^{\o})$ defined in \eqref{nonlin}. 
From Proposition~\ref{prop:stochobjects}, we have 
\begin{align*}
z^{\o}\in  C_{t,\text{loc}}W_{x}^{\al-\frac 12- ,\infty}(\R\times \T) \qquad \text{and} \qquad \NN(z^\o)\in C_{t,\text{loc}}H^{2\al-}_{x}(\R\times \T), 
\end{align*}
almost surely, provided $\al>\tfrac 14$.
So now we fix $\al \in (\tfrac 14, \tfrac 12]$ and let $s\in (\tfrac 12-\al,2\al)$. For $0<T\leq 1$, we define $\O_{T}\subset \O$  by
\begin{align*}
\O_{T}:=\{ \o\in \O \,:\, \|z^{\o}\|_{C([0,1];W^{\al-\frac 12 -, q(s)})}+\| \NN(z^{\o})\|_{C([0,1];H^{2\al-})}\leq C^{-1}T^{-1} \},
\end{align*} 
where 
\begin{align}
q(s):= 
\begin{cases}
\frac{1}{s} \,\,& \text{if} \,\,\, s\in (0,\tfrac 12 ], \\
\frac{1}{1-s} \,\, & \text{if}\,\,\, s\in (\tfrac 12, 1).
\end{cases} \label{qs}
\end{align}
By Proposition~\ref{prop:stochobjects}, we have $\prob(\O_{T}^{c})\leq Ce^{-\frac{c}{T}}$. Then, for each $\o\in \O_{T}$, we apply Proposition~\ref{prop:detlocal} to obtain a unique solution $v^{\o}\in C([0,T];H^{s}(\T))$ to \eqref{BBMv}. It follows, that for each $\o\in \O_{T}$, there exists a solution $u^{\o}=z^{\o}+v^{\o}$ to \eqref{BBMD} on $[0,T]$.
\noi
To obtain the almost sure existence, we set $\Sigma=\cup_{n=1}^{\infty} \O_{1/n}$ and note that $\prob(\Sigma)=1$. Hence, for every $\o\in \Sigma$, there exists $T^{\o}>0$ and a unique solution $v^{\o}\in C([0,T^{\o}];H^{s}(\T))$ to \eqref{BBMv}. 
\end{proof}

\subsection{Sharpness of Theorem \ref{Thm:ASLWP}} \label{subs:sharpness}
The limiting restriction $\alpha>\frac{1}{4}$ of the above argument arises from the nonlinear term of the second order Picard expansion:
\begin{align*} 
\tilde{Z}^{\o}(t):=-\frac{i}{2}\int_{0}^{t}S(t-t')\NN(z(t'))\, dt'. 
\end{align*} 
We show in this subsection that $\NN(z)$ fails to be a distribution almost surely when $\al \leq\frac{1}{4}$.

If $Z$ were in fact a distribution almost surely at least for some $\al \geq \al_0$ where $\al_0\leq \tfrac 14$, then we may hope to lower the regularity restriction in Theorem~\ref{Thm:ASLWP} by considering the higher order perturbative expansion:
\begin{align*}
u = z^\o+\tilde{Z}^\o + w,
\end{align*}
 where we now solve the fixed point problem for the remainder $w$. The idea here is that the equation for $w$ does not contain the term $\tilde{Z}$ which was responsible for the regularity restriction in solving \eqref{BBMv} for the first order perturbative expansion $v$ (see \eqref{zz}).
 We thus expect $w$ to be almost surely smoother than $v$.  
We show that for BBM, no improvement occurs because $\NN(z)$ fails to be a distribution almost surely when $\al \leq\frac{1}{4}$. We adapt an argument in \cite{GPburger} for the stochastic Burgers equation. For fixed $\al$, let $f$ belong to the support of the Gaussian measure $\mu_{\al}$ in \eqref{mual}. 
 For simplicity, we will show the Dirichlet projected regularisation $\NN(f_{N})$, where $f_{N}=\P_{\leq N}f$, fails to define a distribution almost surely as $N\rightarrow \infty$. 
 To this end, let $\phi\in C^{\infty}(\T)$ be such that $\ft \phi (0)=0$ and $\phi \not\equiv 0$; for instance, we will assume $\ft \phi(1) \neq 0$. Let $N\geq M\gg 1$ be dyadic and define 
 \begin{align*}
X_{N}(\phi):=\langle \NN(f_{N}),\phi \rangle.
\end{align*}
We will show $X_{N}(\phi)$ fails to converge almost everywhere with respect to $\mu_{\al}$. This implies $X^{\o}_{N}(\phi)=\langle \NN(f_{N}^{\o}),\phi \rangle$, where $f^{\o}$ is given by \eqref{initialdata}, fails to converge almost surely. 
We begin by showing the sequence $\{ X_{N}(\phi)\}_{N}$ fails to converge in the Gaussian Hilbert space $L^{2}(\mu_{\al})$. 
Indeed, we have 
\begin{align*}
\| X_{N}(\phi)&-X_{M}(\phi)\|_{L^{2}(\mu_{\al})}^{2} = \E[ |X^{\o}_{N}(\phi)-X^{\o}_{M}(\phi)|^{2}] \\
 &= \sum_{n,m \neq 0} \vp(n) \ft \phi(n) \vp(m)\ft \phi(m)  \sum_{(\ast)} \frac{ \E[g_{n_1}g_{n_2}\cj{g_{m_1}g_{m_2}}]}{\jb{n_1}^{\al}\jb{n_2}^{\al}\jb{m_1}^{\al}\jb{m_2}^{\al}  },
\end{align*}
where the inner summation above is restricted to the set of $(n_1,n_2,m_1,m_2)\in \Z^4$ satisfying: 
\begin{align*}
n=n_1+n_2, \quad m=m_1+m_2, \quad M< \max(|n_1|,|n_2|)\leq N, \quad M< \max(|m_1|,|m_2|)\leq N.
\end{align*}
Hence, we have
\begin{align*}
\E[ |X^{\o}_{N}(\phi)-X^{\o}_{M}(\phi)|^{2}] &\ges  \sum_{n\neq 0} |\vp(n)|^{2}|\ft \phi(n)|^{2} \sum_{\substack{n=n_1+n_2 \\ M<|n_1|\leq N}} \frac{1}{\jb{n_1}^{2\al}\jb{n_2}^{2\al}}\\
& \ges |\vp(1)|^{2}|\ft \phi(1)|^{2} \sum_{M<|n_1|\leq N} \frac{1}{\jb{n_1}^{2\al}\jb{n_1-1}^{2\al}} \\
& \sim N^{1-4\al}.
\end{align*}
Thus, if $\al\leq \tfrac{1}{4}$, $X_{N}(\phi)$ fails to be a Cauchy sequence in $L^{2}(\mu_{\al})$ and hence fails to converge in $L^{2}(\mu_{\al})$. 
Now for every $N$, we have
\begin{align*}
X_{N}(\phi)&=\langle \NN(f_N),\phi \rangle = \bigg\langle \vp(D_x)\bigg( f_{N}^{2} - \int_{\T}f_{N}^{2}dx \bigg), \phi \bigg\rangle \\
& = \langle \vp(D_x)( f_{N}^{2} - \|f_{N}\|_{L^{2}(\mu_\al)}^{2} ), \phi \rangle +\bigg\langle \vp(D_x)\bigg(  \|f_{N}\|_{L^{2}(\mu_\al)}^{2}  - \int_{\T}f_{N}^{2}dx \bigg), \phi \bigg\rangle \\
& =\langle \vp(D_x)( f_{N}^{2} - \|f_{N}\|_{L^{2}(\mu_\al)}^{2} ), \phi \rangle \\
& =:Y_{N}(\phi).
\end{align*}
and for every $N$, $Y_{N}(\phi)\in \mathcal{H}_{2}$, the homogeneous Wiener chaos of order 2; see~\cite[Chapter II]{Janson}. For elements in a fixed homogeneous Wiener chaos, convergence in $L^{2}$ is equivalent to convergence in probability; see \cite[Theorem 3.50]{Janson}. Therefore, $Y_N(\phi)=X_{N}(\phi)$ fails to converge in probability and hence $X_{N}^{\o}(\phi)$ fails to converge almost surely when $\al \leq \tfrac 14$. Applying the above with $f^{\o}=z^{\o}(t)=S(t)u_0^{\o}$, we obtain the same conclusion, for each fixed $t$.

\section{Probabilistic global theory on $\T$}\label{section:globalt}

In this section we show that when $\al=\frac 12$, we can extend the local-in-time (random) solutions 
constructed in Theorem~\ref{Thm:ASLWP} globally in time. 
Notice that when $\al=\frac 12$, the local solution $u(t)\in H^{s}(\T)$ for $s<0$ for each fixed $t$ and thus almost surely does not belong to $L^{2}(\T)$. 
From the general local theory in Proposition~\ref{prop:detlocal}, we can extend the local-in-time solutions if we have an a priori bound on the remainder $v:=u-z$ of the form \eqref{vhsbnd} with $s=2\al-$.
Ultimately, we were able to obtain such a bound only when $\al=\tfrac 12$. We view this heuristically as overcoming the logarithmic divergence in \eqref{Ceps}.
 In the following though we will keep $\al\in (\tfrac 14, \tfrac 12 ]$ general.
 
In order to make the following computations secure, we consider the smoothed initial value problem \eqref{vepseq} for $v_{k}$:
\begin{align}
\begin{cases}
i\dt v_k = \vp(D_x) \big(v_k + \frac 12 v_{k}^2 + z_{k} v_k  \big)+\frac{1}{2} \NN(z_k) \\
v_k |_{t=0}=0,
\end{cases}
 \label{vepseq2}
\end{align}
where $z_{k}:=\rho_k \ast z$ for some smooth mollifier $\{\rho_{k} \}_{k\in \N}$ and $k\in \N$. Notice that
\begin{align}
\NN(z_{k})=-(1-\dx^{2})^{-1}\dx( \P_{\neq 0}(z_{k}^{2})). \label{rnormsmooth}
\end{align}
As $z_{k}$ is smooth, there is a unique smooth global-in-time solution $v_{k}$ to \eqref{vepseq2} for every $k\in \N$. 
The brunt of the work will be to establish the following \textit{uniform} in $k$ bound on solutions $v_{k}$ to \eqref{vepseq2}. This is the content of Subsections~\ref{sec:modified} and \ref{subsec:bdk}.

\begin{proposition}\label{prop:bdk} Let $\al =\frac 12$ and $s<1$ sufficiently close to one. Given $T, \e>0$, there exists $\tilde{\O}_{T,\e}\subset \O$ such that 
\begin{align*}
\prob( (\tilde{\O}_{T,\e})^{c}) <\e,
\end{align*}
a sufficiently large integer $k_0=k_0(T,\e)$ and a finite constant $C(T,\e)>0$ such that the following bound holds: 
\begin{align}
\sup_{k\geq k_0} \sup_{t\in [0,T]}\|v^{\o}_{k}(t)\|_{H^{s}(\T)} \leq C(T,\e), \label{bdvk}
\end{align}
for every solution $v_{k}^{\o}$ to \eqref{vepseq2} with $\o \in \tilde{\O}_{T,\ep}$.
\end{proposition}

From this result, we iterate the probabilistic local theory in Subsection~\ref{Sec:local} to conclude Theorem~\ref{Thm:ASGWP}; see Subsection~\ref{subsec:globalproof}.

To obtain the bound \eqref{bdvk}, we will apply the $I$-method in this probabilistic context, which we now describe. 
Given $\al \in (\tfrac 14, \tfrac 12 ]$, we fix $s:=2\alpha-\delta$, where $\delta>0$ is to be sufficiently small. Given $N\geq 1$, let $I_{N}=I$ be the Fourier multiplier operator defined by $\widehat{If}(n)=m_{N}(n)\widehat{f}(n),$ where $m_{N}$ is defined in \eqref{Imult}. The operator $I$ is smoothing of order $(1-s)$ and by the Littlewood-Payley square function theorem, for any $1<p<\infty$, $s_{0}\in \R$ and $0\leq a\leq 1-s$, we have  \begin{equation} 
 \|If\|_{W^{s_{0}+a,p}(\T)}\lesssim N^{a}\|f\|_{W^{s_{0},p}(\T)}. \label{smoothingI} 
 \end{equation}
We also have 
\begin{equation}
\|f\|_{H^{s}}\lesssim \|If\|_{H^{1}} \label{hstoIh1}
\end{equation}
and hence, in order to obtain \eqref{bdvk}, it suffices to obtain a uniform in $k$ (sufficiently large) bound on $Iv_{k}$ in $H^{1}$ to which we turn to in the next subsection.

The following probabilistic lemma quantifies the growth rate of the smoothed random linear solution $Iz$.

\begin{lemma}\label{lemma:izintmoment} For any $p\geq 2$ and any fixed $t\in\R$, we have 
  \begin{equation}
 \E \left[\|Iz(t)\|_{L^{p}_{x}}^{p} \right]^{\frac{1}{p}} \leq C_{s} p^{\frac{1}{2}}\phi_{2\alpha}^{\frac{1}{2}}(N), \label{expecesti}
\end{equation}
where $\phi_{2\al}$ is defined in \eqref{phi}.
Furthermore, we have \begin{equation}
\prob \left( \frac{\|Iz(t)\|_{L^{p}_{x}}}{p^{\frac{1}{2}}\phi_{2\alpha}^{\frac{1}{2}}(N) }>\lambda\right)\leq \frac{C_{s}^{p}}{\lambda^{p}}. \label{probabesti}
\end{equation}
\end{lemma}

\begin{proof} 
We split $z=\P_{\leq N}z+\P_{>N}z$ and consider each piece separately. For the low frequency one, \eqref{wienerchaos} implies
 \begin{align*}
\E \left[\|I\P_{\leq N}z\|_{L^{p}_{x}}^{p} \right] & \lesssim \int_{\T} \E \left[ \bigg\vert \sum_{|n|\leq N}\frac{e^{it\varphi(n)}g_{n}(\omega)}{\jb{n}^{\alpha}} \bigg\vert^{p}\right]dx \\
& \lesssim \int_{\T} C^{p}p^{\frac{p}{2}}\left( \sum_{|n|\leq N}\frac{1}{\jb{n}^{2\alpha}} \right)^{\frac{p}{2}} dx\les p^{\frac{p}{2}}\phi_{2\alpha}^{\frac{p}{2}}(N).
\end{align*}
For the high frequency piece, \eqref{wienerchaos} again implies
\begin{align*}
\E \left[\|I\P_{>N}z\|_{L^{p}_{x}}^{p} \right]  = \int_{\T}\E \left[ \vert  I\P_{>N}z|^{p} \right]dx &\les \int_{\T} C^{p}p^{\frac{p}{2}}\left(\sum_{|n|>N} \frac{m_{N}(n)^{2}}{\jb{n}^{2\alpha}} \right)^{\frac{p}{2}} dx\\
&\lesssim C^{p}p^{\frac{p}{2}}\left(\sum_{|n|>N} \frac{N^{2(1-s)}}{\jb{n}^{1+(1-2s+2\alpha)}} \right)^{\frac{p}{2}} \\
& \lesssim p^{\frac{p}{2}}N^{\frac{p}{2}(1-2\alpha)},
\end{align*} 
where we note $1-2s+2\alpha=1-2\alpha+2\delta>0$. Then \eqref{probabesti} follows from \eqref{expecesti} and the Chebyshev inequality.
\end{proof}

We will actually only ever use Lemma~\ref{lemma:izintmoment} when $p=2$. In this case, the set in \eqref{probabesti} no longer depends on $t\in \R$ because the operators $I$ and $S(t)$ commute and $S(t)$ is unitary on $L^{2}(\T)$.

\subsection{Modified energy estimate}\label{sec:modified}
Applying the $I$-operator to \eqref{vepseq2} and noting \eqref{rnormsmooth}, we see that $Iv_{k}$ satisfies 
\begin{align}
\begin{cases}
\partial_{t}Iv_{k} = -(1-\partial_{x}^{2})^{-1}\partial_{x}\left[Iv_{k}+\frac{1}{2} I(v_{k}^{2})+I(v_{k}z_{k})+\frac 12 I(\P_{\neq 0}(z_{k}^{2}))\right]\\
Iv_{k}|_{t = 0} = 0,
\end{cases} 
\label{Iv}
\end{align}
We define the modified energy functional $E(Iv_{k})(t):=\frac{1}{2}\|Iv_{k}(t)\|_{H^1}^{2}$. Using \eqref{Iv}, we compute 
\begin{align*}
E(Iv_{k})(t)-E(Iv_{k})(0) &= \int_{0}^{t}\int_{\T} (\partial_{t}Iv_{k})(Iv_{k})+(\partial_{t}\partial_{x}Iv_{k})(\partial_{x}Iv_k)\, dxdt'\\
& =\int_{0}^{t}\int_{\T} (Iv_{k})(1-\partial_{x}^{2})\partial_{t}(Iv_{k}) \,dxdt' \\ 
& = \frac{1}{2}\int_{0}^{t}\int_{\T} (\partial_{x}Iv_{k})[I(v_{k}^{2})-(Iv_{k})^{2}]\,dxdt' \tag{I} \\
&  \,\,\,\,\,\,\,\, +\frac{1}{2}\int_{0}^{t}\int_{\T} (\partial_{x}Iv_{k})I(\P_{\neq 0}(z_{k}^{2}))\,dxdt' \tag{II} \\
& \,\,\,\,\,\,\,\,+\int_{0}^{t}\int_{\T} (\partial_{x}Iv_{k})I(v_{k}z_{k})\,dxdt' \tag{III}.
\end{align*}
We now estimate each of (I) through (III) in the following section. Note that all implicit constants in these estimate will be independent of $k\in \N$. We also write $E_{k}(t):=E(Iv_k)(t)$.

\medskip

\noi 
$\bullet$ \textbf{Estimate for (I):}
We begin with the following lemma which arranges for a negative power of $N$ from the commutator.
\begin{lemma}\label{lemma:Iv2commutator}
 Let $s>\frac 12$ and $w\in H^{s}(\T)$. Then, we have
 \begin{align*}
\bigg \vert\int_{\T} (\partial_{x}Iw)[I(w^{2})-(Iw)^{2}]\,dx \bigg\vert\lesssim N^{-\frac{3}{2}+}\|Iw\|_{H^{1}(\T)}^{3}.
\end{align*}
\end{lemma}
\begin{proof}
The argument here is similar to that in \cite[Lemma 3.4]{WangBBM}.
By Plancherel, we have \begin{multline*}
 \int_{\T} (\partial_{x}Iw)[I(w^{2})-(Iw)^{2}]\,dx  \\ =\sum_{n_{1}+n_{2}+n_{3}=0}in_{3}m(n_{3})(m(n_{1}+n_{2})-m(n_{1})m(n_{2}))\ft w(n_{1})\ft w (n_{2})\ft w (n_{3}).
\end{multline*}
We symmetrise this to obtain $$ \sum_{n_{1}+n_{2}+n_{3}=0} M(n_{1},n_{2},n_{3}) \ft w (n_{1})\ft w (n_{2})\ft w(n_{3}),$$ where $M$ is defined to be the symmetric multiplier 
\begin{align*}
 M(n_{1},n_{2},n_{3})=\frac{i}{3} [ & n_{1}m(n_{1})(m(n_{2}+n_{3})-m(n_{2})m(n_{3})) \\ &+n_{2}m(n_{2})(m(n_{1}+n_{3})-m(n_{1})m(n_{3}))\\ &+n_{3}m(n_{3})(m(n_{1}+n_{2})-m(n_{1})m(n_{2}))].
\end{align*} 
By symmetry, we assume $|n_{3}|\leq |n_{2}|\leq |n_{1}|$. Furthermore, we assume $|n_{1}|>N$, since otherwise $m(n_{j})=1$ for all $j=1,2,3$, which implies $M(n_{1},n_{2},n_{3})=0$ on $n_{1}+n_{2}+n_{3}=0$. In addition, we also assume $|n_{2}|\gtrsim N$ since if $|n_{2}|\ll N$ we obtain a contradiction to the conditions $n_{1}+n_{2}+n_{3}=0$, $|n_{1}|>N$ and $|n_{3}|\leq |n_{2}|$. For shorthand, we define 
\begin{align*} 
\Lambda_{N}(\nbar):=\{ (n_{1},n_{2},n_{3})\in \Z^{3}  :  n_{1}+n_{2}+n_{3}=0,\, |n_{3}|\leq & |n_{2}|\leq |n_{1}|, \\
& |n_{1}|>N, \, |n_{2}|\gtrsim N \,\}.
\end{align*}
Using the condition $n_{1}+n_{2}+n_{3}=0$, it is easy to verify 
$$ M(n_{1},n_{2},n_{3})=\frac{i}{3}\left[n_{1}m^{2}(n_{1})+n_{2}m^{2}(n_{2})+n_{3}m^{2}(n_{3}) \right],$$ and hence on $\Lambda_{N}(\nbar)$, we have
\begin{equation}
|M(n_{1},n_{2},n_{3})|\lesssim |n_{3}|m^{2}(n_{3}). \label{Mbd}
\end{equation} 
Setting $ y(n)=\jb{n}m(n)\ft w(n)$ and using \eqref{Mbd}, we have thus reduced to showing 
\begin{equation} \label{v2com1}
\sum_{\Lambda_{N}(\nbar)}\frac{|y(n_{1})||y(n_{2})| |y(n_{3})|}{\jb{n_{1}}\jb{n_{2}} m(n_{1}) m(n_{2})} \lesssim N^{-\frac{3}{2}+}\|y(n)\|_{\l_{n}^{2}}^{3}.
\end{equation} 
Now we note that on $\Lambda_{N}(\nbar)$, we have for any $a\geq 1-s$, $$\jb{n_{j}}^{a}m(n_{j})\gtrsim N^{a}, \qquad j=1,2.$$ Applying this with $a=1$ and $a=\frac{1}{2}-$ (as $s>\frac{1}{2}$) and using Young's inequality, the left hand side of \eqref{v2com1} is bounded by 
\begin{align*}
CN^{-\frac{3}{2}+}\sum_{\Lambda_{N}(\nbar)}\frac{| y(n_{1})|| y(n_{2})| | y(n_{3})|}{\jb{n_{1}}^{\frac{1}{2}+} } & \lesssim  CN^{-\frac{3}{2}+} \|y(n)\|_{\l_{n}^{2}}^{2}\|\jb{n}^{-\frac{1}{2}-}y (n)\|_{\l^{1}_{n}} \\
&\lesssim N^{-\frac{3}{2}+}\|y(n)\|_{\l_{n}^{2}}^{3},
\end{align*} as required.
\end{proof}

As $\al >\frac 14$, we have $s>\frac 12$ and hence Lemma \ref{lemma:Iv2commutator}, implies
\begin{equation}
\label{esti1}
|\textup{(I)}|=\bigg\vert \frac{1}{2}\int_{0}^{t}\int_{\T} (\partial_{x}Iv_{k})[I(v_{k}^{2})-(Iv_{k})^{2}]\,dxdt'  \bigg\vert \lesssim N^{-\frac{3}{2}+}\int_{0}^{t}E_{k}^{\frac{3}{2}}(t')dt'.
\end{equation}

\medskip

\noi 
$\bullet$ \textbf{Estimate for (II):}
By Cauchy-Schwarz and \eqref{smoothingI}, we have
\begin{align}
|(\textup{II})|&=\bigg\vert \int_{0}^{t}\int_{\T} (\partial_{x}Iv_{k})I(\P_{\neq 0}(z_{k}^{2}))\,dxdt' \bigg\vert  \notag\\
& \leq \| I(\P_{\neq 0}(z_{k}^{2}))\|_{L^{\infty}([0,t]; L^{2})}\int_{0}^{t} E^{\frac 12}_{k}(t') dt' \notag \\
& \les N^{1-2\al+}\|\P_{\neq 0}(z_{k}^2)\|_{L^{\infty}([0,t]; H^{2\al-1-})}\int_{0}^{t} E^{\frac 12}_{k}(t')dt'. \label{esti2}
\end{align}
Notice that $\|\P_{\neq 0}(z_{k}^2)\|_{L^{\infty}([0,t]; H^{2\al-1-})} \sim \|\NN(z_{k})\|_{L^{\infty}([0,t]; H^{2\al-})}.$

\medskip

\noi 
$\bullet$ \textbf{Estimate for (III):}

\begin{lemma}\label{lemma:IvIzcommutator} Let $\frac{1}{4}<\alpha \leq \frac{1}{2}$, $w\in H^{s}(\T)$ and $k\in \N\cup \{\infty\}$. Then, there exists $p=p(\alpha, s)\gg 2$ sufficiently large, so that 
\begin{equation}
\|I(wz_{k})-(Iw)(Iz_{k})\|_{L^{2}}\lesssim N^{-\big(s -\frac{1}{2}-\frac{1}{p} + \big)}\|Iw\|_{H^{1}}\|z_{k}\|_{W^{\alpha-\frac{1}{2}-,p}}, \label{vzcom}
\end{equation}
where we understand $z_{\infty}:=z$.
\end{lemma}

\begin{proof} 
We may suppose $k=\infty$.
Let $\sigma:=\frac{1}{2}-\alpha+\frac{1}{10}\delta$.
Split $w=w_{N^{\frac{1}{2}}}+w^{N^{\frac{1}{2}}}$, and $z=z_{N}+z^{N}$, where $f_{N}:=\mathcal{F}^{-1}\{\mathbbm{1}_{|n|\leq  N/2}\ft f (n) \}$ and $f^{N}:=\mathcal{F}^{-1}\{\mathbbm{1}_{|n|> N/2}\ft f (n) \}$. Then 
\begin{align*}
I(wz)-(Iw)(Iz)  = &I(w_{N^{\frac{1}{2}}} \cdot z_{N})-I(w_{N^{\frac{1}{2}}})I(z_{N}) \tag{A} \\
& + I(w_{N^{\frac{1}{2}}}\cdot z^{N})-(Iw_{N^{\frac{1}{2}}})(Iz^{N}) \tag{B} \\
& - I(w^{N^{\frac{1}{2}}})(Iz) \tag{C} \\
& + I(w^{N^{\frac{1}{2}}} \cdot z). \tag{D}
\end{align*}
We estimate each piece above separately.\\
\medskip
\noi
$\bullet$  \textbf{(A):} Since $I$ is the identity on frequencies $\{ |n|\leq N\}$, we have $I(w_{N^{\frac{1}{2}}})=w_{N^{\frac{1}{2}}}$ and $I(z_{N})=z_{N}$. Next, notice that $\supp \mathcal{F}\{w_{N^{\frac{1}{2}}} \cdot z_{N}\}\subset \{|n|\leq N\}$ and hence $I(w_{N^{\frac{1}{2}}} \cdot z_{N})=w_{N^{\frac{1}{2}}} \cdot z_{N}$. Combining these two observations we see that $\textup{(A)}\equiv 0$.

\medskip
\noi
$\bullet$  \textbf{(B):}
 In this case, we argue by duality: $$ \|\textup{(B)} \|_{L^{2}(\T)}=\sup_{\|h\|_{L^{2}(\T)}=1} \bigg\vert \int_{\T} h \,\textup{(B)}dx \bigg\vert$$
Denote $f(n_{1}):=\widehat{w_{N^{\frac{1}{2}}}}(n_{1})$ and $g(n_{2}):=\widehat{z^{N}}(n_{2})$. 
By Parseval, we have 
\begin{equation}
\bigg\vert \int_{\T} h \,\textup{(B)}dx \bigg\vert = \sum_{\substack{|n_{1}|<N^{1/2}/2 \\ |n_{2}|>N/2}} |m(n_{1})m(n_{2})-m(n_{1}+n_{2})| |f(n_{1})||g(n_{2})||\ft h(n_{1}+n_{2})|. \label{commu2}
\end{equation} In this regime, $m(n_{1})\equiv 1$. The mean value theorem implies $|m(n_{2})-m(n_{1}+n_{2})| \lesssim N^{1-s}|n_{1}||n_{2}|^{-2+s}$. Thus, 
\begin{align*}
\eqref{commu2} &\lesssim N^{1-s}\sum_{\substack{|n_{1}|<N^{1/2}/2 \\ |n_{2}|>N/2}} \frac{|n_{1}|^{\frac{1}{2}+\frac{1}{10}\delta}}{|n_{2}|^{2-s-\sigma}}\frac{|n_{1}m(n_{1})f(n_{1})|}{|n_{1}|^{\frac{1}{2}+\frac{1}{10}\delta}} \frac{|g(n_{2})|}{|n_{2}|^{\sigma}}|\ft h (n_{1}+n_{2})| \\
& \lesssim N^{1-s+\frac{1}{4}+\frac{\delta}{20}-2+s+\sigma} \bigg \| \frac{|n_{1}m(n_{1})f(n_{1})|}{|n_{1}|^{\frac{1}{2}+\frac{1}{10}\delta}}\bigg\|_{\l^{1}_{n_{1}}}   \bigg\|\frac{|g(n_{2})|}{|n_{2}|^{\sigma}} \bigg\|_{\l^{2}_{n_{2}}} \|\ft h\|_{\l^{2}}\\
& \lesssim N^{-\left(\frac{1}{4}+\frac{s}{2}+ \right)}\|Iv\|_{H^{1}}\|z\|_{H^{-\sigma}}\|h\|_{L^{2}}.
\end{align*} 
We thus have $\|(B)\|_{L^{2}}\lesssim N^{-\left(\frac{1}{4}+\frac{s}{2}+ \right)}\|Iv\|_{H^{1}}\|z\|_{H^{-\sigma}}$.

\medskip
\noi
$\bullet$  \textbf{(C):}
 By the Sobolev embedding theorem and the mapping property of $I$ \eqref{smoothingI}, we have
 \begin{align*}
\| \textup{(C)}\|_{L^{2}}=\|(Iv^{N^{\frac{1}{2}}})(Iz)\|_{L^{2}}&\lesssim \|Iv^{N^{\frac{1}{2}}}\|_{L^{2}}\|Iz\|_{L^{\infty}} \\
& \lesssim N^{-\frac{s}{2}}\|v^{N^{\frac{1}{2}}}\|_{H^{s}}\|Iz\|_{W^{\frac{1}{p}+\frac{1}{10}\delta,p}} \\
& \lesssim N^{-\big( \frac{s}{2}-\frac{1}{p}-\sigma-\frac{1}{10}\delta \big)}\|Iv\|_{H^{1}}\|z\|_{W^{-\sigma,p}}\\
& \les N^{-\big( s-\frac{1}{2}-\frac{1}{p}+ \big)}\|Iv\|_{H^{1}}\|z\|_{W^{-\sigma,p}},
\end{align*}
provided that
 \begin{align*}
\frac{1}{4}+\frac{1}{2p}+\frac{1}{2}\left(\frac{\delta}{2}+\frac{2}{10}\delta\right)<\alpha<\frac{1}{2}-\frac{1}{p}+\delta-\frac{2}{10}\delta.
\end{align*}
The lower bound above appears to ensure we have a negative power of $N$ while the upper bound is due to the mapping property of $I$ \eqref{smoothingI}. We can afford these conditions on $\alpha$ if we choose $p\gg \frac{1}{\delta}$. 

\medskip
\noi
$\bullet$  \textbf{(D):}
 Once again, we argue by duality writing $$ \|\textup{(D)}\|_{L^{2}(\T)}=\sup_{\|h\|_{L^{2}(\T)}=1} \bigg\vert \int_{\T} h \textup{(D)}dx \bigg\vert.$$ Now by the fractional Leibniz rule, we have 
\begin{align*}
\bigg\vert\int_{\T} h I(v^{N^{\frac{1}{2}}} \cdot z) \, dx\bigg\vert & = \bigg\vert\int_{\T} I(h)v^{N^{\frac{1}{2}}} z \,dx\bigg\vert \\
& \lesssim \|I(h)v^{N^{\frac{1}{2}}} \|_{W^{\sigma, \frac{p}{p-1}}}\|z\|_{W^{-\sigma,p}}  \\
& \lesssim \|z\|_{W^{-\sigma,p}} \left( \|I(h)\|_{L^{\frac{2p}{p-2}}} \| v^{N^{\frac{1}{2}}}\|_{H^{\sigma}}+\|I(h)\|_{H^{\sigma}}\|v^{N^{\frac{1}{2}}}\|_{L^{\frac{2p}{p-2}}} \right).
 \end{align*} 
 For each of these four terms, we use:
\begin{itemize}
\item By the Sobolev inequality and \eqref{smoothingI}, $\|Ih\|_{L^{\frac{2p}{p-2}}} \lesssim \|Ih\|_{H^{\frac{1}{p}}}\lesssim N^{\frac{1}{p}},$ since \\ $\frac{1}{p}<1-s$ provided we choose $p\gg \frac{1}{\delta},$
\item  $\|Ih\|_{H^{\sigma}}\lesssim N^{\sigma},$ since $\sigma<1-s$ which is true as $\alpha\leq \frac{1}{2},$
\item $\|v^{N^{\frac{1}{2}}}\|_{H^{\sigma}}\lesssim N^{-\frac{1}{2}(s-\sigma)}\|Iv\|_{H^{1}},$
\item $\|v^{N^{\frac{1}{2}}}\|_{L^{\frac{2p}{p-2}}} \lesssim N^{-\frac{1}{2}\left(s-\frac{1}{p}\right)}\|Iv\|_{H^{1}}.$
\end{itemize}
With these, we obtain
\begin{align*}
\|\textup{(D)}\|_{L^{2}}& \lesssim  \left( N^{-\left(\frac{s}{2}-\frac{\sigma}{2}-\frac{1}{p} \right)}+N^{-\left( \frac{s}{2}-\s-\frac{1}{2p} \right)}  \right)\|Iv\|_{H^{1}}\|z\|_{W^{-\sigma,p}} \\
&\lesssim N^{-\left( s-\frac{1}{2}+ \right)} \|Iv\|_{H^{1}}\|z\|_{W^{-\sigma,p}}.
\end{align*} 
Finally, combining the results of (A) through (D) we obtain \eqref{vzcom} with $p=p(\alpha,s)=\frac{100}{2\alpha-s}.$
\end{proof}

\begin{lemma}\label{lemma:for4}  
Let $s>\frac{1}{2}$ and $k\in \N\cup \{\infty\}$. Then, we have 
\begin{align*}
\bigg\vert \int_{0}^{t}\int_{\T}(\partial_{x}Iv_{k})(Iv_{k})(Iz_{k})dxdt' \bigg \vert \les  \|Iu_{0}^{\omega}\|_{L_{x}^{2}} \left(\int_{0}^{t}E_{k}(Iv_{k})dt' \right).
\end{align*}
\end{lemma}
\begin{proof}
By the algebra property of $H^{s}(\T)$ and Cauchy-Schwarz, we compute
\begin{align*}
\bigg\vert \int_{0}^{t}\int_{\T}(\partial_{x}Iv_{k})(Iv_{k})(Iz_{k})dxdt' \bigg \vert& =\frac{1}{2} \bigg\vert \int_{0}^{t}\int_{\T}\dx [ (Iv_{k})^{2}]  (Iz_{k})dxdt' \bigg \vert \\
& \lesssim \int_{0}^{t} \|(Iv_{k})^{2}\|_{H^{1}_{x}}\|Iz_{k}\|_{L^{2}_{x}}dt' \\
& \lesssim \int_{0}^{t} \|Iv_{k}\|_{H^{1}}^{2}\|IS(t')u_{0,k}^{\omega}\|_{L^{2}_{x}}dt' \\
&\lesssim  \|Iu_{0}^{\omega}\|_{L_{x}^{2}}\left(\int_{0}^{t}E(Iv_{k})(t')dt' \right).
\end{align*}
\end{proof}

\noi
Writing 
\begin{align*}
\int_{\T} (\partial_{x}Iv_{k})I(v_{k}z_{k})=&\int_{\T} (\partial_{x}Iv_{k})[I(v_{k}z_{k})-(Iv_{k})(Iz_{k})]dx \\
& + \int_{\T} (\partial_{x}Iv_{k})(Iv_{k})(Iz_{k})dx, 
\end{align*}
we have from Lemmas \ref{lemma:IvIzcommutator} and \ref{lemma:for4}, 
\begin{equation} 
|\textup{(III)}| \lesssim N^{-\left( s-\frac{1}{2}+\right)}\|z_{k}\|_{L^{\infty}([0,t]; W^{\alpha-\frac{1}{2}-,p})} \int_{0}^{t}E_{k}(t')dt'  + \|Iu_{0}^{\omega}\|_{L_{x}^{2}} \int_{0}^{t}E_{k}(t')dt'. \label{esti4}
\end{equation}

\medskip 
\noi
Combining \eqref{esti1}, \eqref{esti2} and \eqref{esti4}, we have shown the following energy estimate:
\begin{align}
\begin{split}
E(Iv_{k})(t) \leq & E(Iv_k)(0)+C_{s,\al} N^{-\frac 32 +}\int_{0}^{t}E^{\frac{3}{2}}(Iv_k)(t')dt' \\
&  +C_{s,\al} N^{1-2\al+} \| \NN(z_{k})\|_{L^{\infty}([0,t];H^{2\al-})}\int_{0}^{t}E^{\frac{1}{2}}(Iv_k)(t')dt' \\
& +C_{s,\al} N^{-\left( s-\frac{1}{2}+\right)}\|z_{k}\|_{L^{\infty}([0,t]; W^{\alpha-\frac{1}{2}-,p})} \int_{0}^{t}E(Iv_k)(t')dt' \\
& +  C_{s,\al} \|Iu_{0}^{\omega}\|_{L_{x}^{2}} \int_{0}^{t}E(Iv_k)(t')dt'.
\end{split}
\label{energy}
\end{align}

\subsection{Proof of Propostion~\ref{prop:bdk}}\label{subsec:bdk}

In this subsection, we complete the proof of Proposition~\ref{prop:bdk} by turning the inequality \eqref{energy} into a bound on $Iv_{k}$ in $H^{1}$ and hence a bound on $v_{k}$ in $H^{s}$. To begin, we first deal with a technical point: we must ensure that the set $\tilde{\O}_{T,\ep}$ we construct in Proposition~\ref{prop:bdk} is independent of all $k\in \N$ large enough. The point here is that while Young's inequality ensures 
\begin{align*}
 \|z_{k}\|_{C([0,2T]; W^{\al-\frac{1}{2}-,p})} \leq \|z\|_{C([0,2T]; W^{\al-\frac{1}{2}-,p})}, 
\end{align*}
for every $k\in \N$ and $p> 1$, we cannot conclude a similar statement for the nonlinearities $\NN(z_{k})$ and $\NN(z)$. To get around this, we appeal to the almost sure convergence of their norms from Proposition~\ref{prop:stochobjects}.

With $T>0$ fixed, we define 
\begin{align*}
\Sigma_{\text{conv},T}:=\{\o\in \O\, : \, (z_{k}^{\o},\NN(z_{k}^{\o}))\rightarrow (z^{\o},\NN(z^{\o})) \,\, \text{in}\,\, C_{2T}W^{\al-\frac 12-,r}\times C_{2T}H^{s}\,\, \text{as} \,\, k\rightarrow \infty \,\},
\end{align*} where $r(s,\al):=\max\{p(s,\al),q(s)\}$ with $p(s,\al)$ given by Lemma~\ref{lemma:IvIzcommutator} and $q(s)$ given by \eqref{qs}.
Put simply, $\Sigma_{\text{conv},T}$ is the set on which we have almost sure convergence of the mollified enhanced data set (in the appropriate norms). That this set is of full probability is a direct consequence of Proposition~\ref{prop:stochobjects}.
With $K>0$ fixed, we also define 
\begin{align*}
\O_{K,T,\al}& = \{  \o \in \Sigma_{\text{conv},T}\, :\,  \|z\|_{C_{2T}W^{\al-\frac 12-,r}}+\| \NN(z)\|_{C_{2T} H^{s}} \leq K \}.
\end{align*}
By Egorov's theorem, for any $\e>0$, there exists a measurable set $\O_{\e}\subset \Sigma_{\text{conv},T}$ with $\prob( \Sigma_{\text{conv},T} \setminus \O_{\e})<\tfrac{\e}{3}$ such that $\NN(z^{\o}_{k})$ converges uniformly to $\NN(z^{\o})$ as $k\rightarrow \infty$ in $C_{2T}H^{s}_{x}$ for every $\o\in \O_{\e}$. Hence, there exists $k_0=k_0(T,\e)$ such that for every $k\geq k_0$, we have 
\begin{align}
\| \NN(z_{k}^{\o})\|_{C_{2T}H^{s}_{x}} \leq 1+K \label{Nzbound}
\end{align}
for every $\o\in \O_{K,T,\al,\e}:=\O_{\e}\cap \O_{K,T,\al}$. Now, Proposition~\ref{prop:stochobjects} implies 
\begin{align}
\prob( \O_{K,T,\al,\e}^{c}) < e^{-C\frac{K}{T^c}}+\frac{\e}{4}. \label{smallset}
\end{align}

We will need the following nonlinear Gronwall inequality which follows from~\cite[Theorem 21]{Dragomir}:

\begin{lemma}\label{gronwall}
Given $T>0$, let $f$ be a non-negative function on $[0,T]$ satisfying 
\begin{align}
f(t)\leq c+a\int_{0}^{t}f(t')\, dt' + b\int_{0}^{t}  f^{\g}(t')\, dt', \label{cgron}
\end{align}
where $a,b,c \geq 0$, $0\leq \g <1$ and for $t\in [0,T]$. Then, we have 
\begin{align*}
f^{1-\g}(t) \leq c^{1-\g}e^{(1-\g)at}+\frac{b}{a}(e^{(1-\g)at}-1)
\end{align*} 
for $t\in [0,T]$.
\end{lemma}

\begin{proof}[Proof of Proposition~\ref{prop:bdk}] Fix $T,\e>0$.
For $\Lambda, K>0$ to be determined later, we set
\begin{align}
\O_{\Lambda, N}&=\{ \o \in \O\, :\, A(N)\leq \Lambda \}, \label{Olambda}
\end{align}
where we have defined 
\begin{align*}
A(N):=\frac{  \|Iu_{0}^{\o}\|_{L^{2}}}{2^{\frac 12}\phi_{2\al}^{\frac 12}(N)},
\end{align*}
and let $\O_{K,T,\al,\e}$ be defined as above. 
We now fix $\o \in \O_{K,T,\al,\e} \cap \O_{\Lambda, N}$ and $k\geq k_0$.
From \eqref{energy} and \eqref{Nzbound}, we have 
\begin{align*}
E_{k} (t) \leq & \, N^{-\frac{3}{2}+} \int_{0}^{t} E^{\frac{3}{2}}_{k}(t') dt'  \tag{I}  \\
&+N^{1-2\al+}K\int_{0}^{t}E^{\frac 12}_{k} (t')dt'   \tag{II}\\
& +\big(N^{-(s-\frac{1}{2}+)}K+A(N)\phi_{2\al}^{\frac 12}(N)   \big)\int_{0}^{t} E_{k}(t')dt'. \tag{III}
\end{align*}
 Let 
\begin{align*}
\cj{T}_{k} =\sup \{ t>0\, :\, E(Iv_k)(t) \leq C_{K,T,s,\al}N^{2-} \},
\end{align*}
where we stress that $C_{K,T,s,\al}$ is independent of any $k\geq k_{0}$. 
From $E(Iv_k)(0)=0$ and continuity in time of $E_{k}(t)$ (since $v_k\in C_{T}H^{1}_{x}$, at least), $\cj{T}_k >0$. For $t\in [0,\cj{T}_k]$, (I) is dominated by (II) and hence 
\begin{align}
E_k (t) \leq C_{K,T,s,\al}\Lambda \phi_{2\al}^{\frac 12}(N)\int_{0}^{t}E_k (t')dt' +C_{K,T,s,\al}N^{1-2\al+}\int_{0}^{t}E_{k}^{\frac 12} (t')dt'  \label{ode}
\end{align}
By Lemma~\ref{gronwall}, this implies
\begin{align*}
E_{k}^{\frac{1}{2}}(t)\leq \frac{N^{1-2\al+}}{\phi^{\frac 12}_{2\al}(N)\Lambda}\big( e^{\frac{1}{2}C_{K,T,s,\al}\phi^{\frac 12}_{2\al}(N)\Lambda t}-1   \big).
\end{align*}
Now by continuity $E_k (\cj{T}_k)=C_{K,T,s,\al}N^{2-}$ and therefore the above inequality implies 
\begin{align}
\cj{T}_{k} \geq \frac{ 2 \log \big( 1+C_{K,T,s,\al}^{\frac 12}N^{2\al-} \phi^{\frac 12}_{2\al}(N)\Lambda  \big)}{C_{K,T,s,\al}\phi^{\frac 12}_{2\al}(N)\Lambda}. \label{T}
\end{align}
Notice that this lower bound is independent of $k\geq k_0$ and as our $k$ was arbitrary, $\cj{T}:=\inf_{k\geq k_0} \cj{T}_k$ is bounded below by the same quantity. Now given $\e>0$, Proposition~\ref{prop:stochobjects}, Lemma~\ref{lemma:izintmoment} and \eqref{smallset} allow us to choose $K=K(\e,T)$ and $\Lambda=\Lambda(\e,s)$ large enough so that 
\begin{align*}
\prob( \O_{K,T,\frac{1}{2},\e}^c ) + \prob(\O_{\Lambda, N}^{c}) <\frac{\e}{2}+\frac{\e}{2}=\e. 
\end{align*}
Thus when $\al =\tfrac 12$, \eqref{phi} and \eqref{T} imply 
\begin{align*}
\cj{T} \geq C_{\e,T,s}\log^{\frac 12}(N). 
\end{align*}
We now choose 
\begin{align}
N=N(\e,T,s)=\exp \bigg( \frac{4T^{2}}{C^{2}_{\e,T,s}} \bigg)\label{Nchoice}
\end{align} so that $\cj{T}\geq 2T$ and hence
 \begin{align*}
\sup_{0\leq t\leq T}E(Iv_k )(t)\leq  C(\ep,T), \qquad \text{on} \qquad  \tilde{\O}_{T,\e}:=\O_{K(\e,T),T,\frac 12 ,\e}\cap \O_{\Lambda(\e,s), N(\e,T)},
\end{align*}
for any $k\geq k_0(T,\e)$.
By \eqref{hstoIh1}, the definition of the modified energy $E(Iv_k)$ and \eqref{Nchoice} we obtain \eqref{bdvk}. This completes the proof of Proposition~\ref{prop:bdk}.

\end{proof}

\subsection{Proof of Theorem \ref{Thm:ASGWP}}\label{subsec:globalproof}

Our aim here is to iterate the local theory for solutions $v$ to \eqref{BBMv} by showing that the bound \eqref{bdvk} also holds for $v$ in place of $v_{k}$. 
To do this, we decompose the whole interval $[0,T]$ into $\ceil{\frac{T}{\dl}}$-many subintervals $I_{j}:=[j\dl, (j+1)\dl]\cap [0,T]$ where $\dl$ is to be determined. Let
\begin{align*}
\O_{\text{lwp}}:= \cap_{j=0}^{\ceil{\frac{T}{\dl}}} \{ \o\in \O \, : \, \|z\|_{C_{I_{j}}W^{\al-\frac 12-,r(s)}}+\|\NN(z)\|_{C_{I_j}H^{s}} \leq L \} 
\end{align*}
With $\dl=\dl(L)\les L^{-1}$ and for $\o \in \O_{T,\e}:=\O_{\text{lwp}}\cap \tilde{\O}_{T,\frac{\e}{2}}$, we have by Proposition~\ref{prop:detlocal}, that $v$ exists on $[0,\dl]$ and solves \eqref{BBMv}. By Theorem~\ref{thm:smoothapp}, we may take the limit $k\rightarrow \infty$ in \eqref{bdvk} to obtain
\begin{align*}
\sup_{t\in [0,\dl]}\|v(t)\|_{H^{s}(\T)} \leq C(T,\e)<\infty, 
\end{align*}
Then by reducing $\dl$ further so that 
\begin{align*}
\dl \sim (C(T,\e)+L)^{-1},
\end{align*}
we conclude $v$ now exists on $[0,2\dl]$ and using Theorem~\ref{thm:smoothapp} we may take the limit $k\rightarrow \infty$ in \eqref{bdvk} to obtain
\begin{align*}
\sup_{t\in [0,2\dl]}\|v(t)\|_{H^{s}(\T)} \leq C(T,\e)<\infty.
\end{align*}
Iterating in this manner finitely many times shows $v$ exists on $[0,T]$ and satisfies the bound  
\begin{align*}
\sup_{t\in [0,T]}\|v(t)\|_{H^{s}(\T)} \leq C(T,\e)<\infty. 
\end{align*}
It remains to check that $\O_{\text{lwp}}^{c}$ stays small. 
By Proposition~\ref{prop:stochobjects}, we have
\begin{align*}
\prob(\O_{\text{lwp}}^c) & \leq \sum_{j=0}^{\ceil{\frac{T}{\dl}}} \prob( \|z\|_{C_{I_{j}}W^{\al-\frac 12-,r(s)}}+\|\NN(z)\|_{C_{I_j}H^{s}} >L ) \\
& \les \frac{T}{\dl} e^{ -C\frac{L}{\dl^{c}}} \\
& \les T(L+C(T,\e))e^{ -CL(L+C(T,\e))^{c} } \\
& \les T(L+C(T,\e))e^{-CL} <\frac{\e}{2},
\end{align*}
by choosing $L=L(T,\e)$ sufficiently large. 
Thus 
\begin{align*}
\prob(\O_{T,\e}^{c})\leq \prob(\O_{\text{lwp}}^c)+\prob(\tilde{\O}_{T,\frac{\e}{2}}^c)<\frac{\e}{2}+\frac{\e}{2}=\e.
\end{align*}

To obtain the almost sure existence is a standard argument. We detail it for the convenience of the reader.  
The set $\O_{T,\ep}$ depends on $K, \Lambda$ and $N$ which in turn depend on $T$ and $\e$.
Given $\ep>0$, let $T_{j}=2^{j}$ and set $\ep_{j}=\ep/T_{j}$. From the above, we obtain sets $\Omega_{T_{j},\ep_{j}}$ by choosing $K_j=K(T_j,\ep_j)$, $\Lambda_j =\Lambda(\ep_j, s)$ large enough so that $\prob(\Omega_{T_{j},\ep_{j}}^{c})<\ep_j$ and then $N_j=N(T_j,\ep_j)$ as in \eqref{Nchoice} (with $T$ and $\ep$ replaced by $T_j$ and $\ep_j$) which implies, as above, $v$ exists on $[0,T_j]$ and satisfies 
\begin{align*}
\sup_{t\in [0,T_j]}\|v(t)\|_{H^{s}(\T)} \leq C(T_j,\ep_j)<\infty. 
\end{align*}
Then the set $\O_{\ep}=\cap_{j=1}^{\infty}\O_{T_{j},\ep_{j}}$ has measure $\prob(\O_{\ep}^{c})<\ep$ with the property that for any $\o\in \O_{\ep}$, there exists a unique solution $v^{\o}\in C([0,\infty); H^{s}_{x}(\T))$ to \eqref{BBMv}, and hence $u^{\o}=z^{\o}+v^{\o}$ solves \eqref{BBMD} on $[0,\infty)$. Then, the same property is true on $\Sigma:=\cup_{\ep>0}\O_{\ep}$ and $\prob(\Sigma)=1$.

\begin{remark}\rm \label{remark:lowhigh}
As mentioned earlier, the conservation of the energy \eqref{econs2} yields global well-posedness of the BBM equation~\eqref{BBM0} in $H^{1}(\T)$. For data in\footnote{Their argument works on both $\T$ and $\R$; see also \cite{Roumegoux}.} $L^{2}(\T)$, Bona and Tzvetkov~\cite{BonaTzvet} employed a low-high (or long wave-short wave) splitting argument to globalise solutions to \eqref{BBM0}. Their idea was to split the data $u_0\in L^{2}(\T)$ as 
\begin{align*}
u_{0}=\P_{\leq N}u_0+\P_{>N}u_0
\end{align*}
and to write the (local) solution as
\begin{align*}
u= u_{\text{low}}+u_{\text{high}},
\end{align*}
where the high part $u_{\text{high}}$ solves the original (nonlinear) BBM equation \eqref{BBMD} with $u_{\text{high}}|_{t=0}=\P_{>N}u_0$ and the low part $u_{\text{low}}$ solves the difference equation 
\begin{align} 
\dt u_{\text{low}}-\partial_{xxt} u_{\text{low}}+\dx u_{\text{low}}+u_{\text{low}}\dx u_{\text{low}}+ u_{\text{low}}\dx u_{\text{low}} +\dx(u_{\text{low}}u_{\text{high}})=0, \label{diffequlow}
\end{align}
with $u_{\text{low}}|_{t=0}=\P_{\leq N}u_{0}$. Then given any $T>0$, by choosing $N=N(T)$ sufficiently large, we ensure that $\P_{>N}u_0$ is so small in $L^{2}(\T)$ that by local theory, $u_{\text{high}}(t)\in L^{2}(\T)$ exists on $[0,T]$. Meanwhile, since $\P_{\leq N}u_0$ is smooth, we can solve \eqref{diffequlow} locally in time within $H^{1}$. To show that $u_{\text{low}}$ exists up to time $T$, a computation shows 
\begin{align*}
\begin{split}
\bigg\vert\frac{d}{dt} \bigg( \frac{1}{2} \|u_{\text{low}}(t)\|_{H^{1}}^{2} \bigg) \bigg\vert &= \bigg\vert \int u_{\text{high}}u_{\text{low}}\dx u_{\text{low}} \, dx \bigg\vert \\
&\leq  \|u_{\text{high}} (t)\|_{L^2} \|u_{\text{low}}(t)\|_{L^{\infty}} \|\dx u_{N}(t)\|_{L^2}  \\
&\leq  C\|u_{\text{low}}(t)\|_{L^{2}}\|u_{\text{low}}(t)\|_{H^1}^{2}. 
\end{split}
\end{align*}
By Gronwall's inequality, we get the a priori bound
\begin{align*}
\sup_{t\in [0,T]} \|u_{\text{low}}(t)\|_{H^1} & \leq \|\P_{\leq N}u_0\|_{H^1} \exp \bigg( C \int_{0}^{T}\|u_{\text{high}} (t')\|_{L^2} \, dt'  \bigg) \\
& \les \|\P_{\leq N}u_0\|_{H^1} \exp (CT\|\P_{>N}u_{0}\|_{L^{2}}).
\end{align*} 
Thus $u_{\text{low}}$ lives up to time $t=T$ which completes the argument. Let us emphasise here that $u_{\text{low}}$ does not solve \eqref{BBMD}, but rather the perturbed equation \eqref{diffequlow} with the additional linear term $\dx(u_{\text{low}}u_{\text{high}})$. So whilst the `energy' $E(u_{\text{low}}(t))=\tfrac{1}{2}\|u_{\text{low}}(t)\|_{H^{1}}^{2}$ is no longer conserved, its growth can still be controlled. 

A natural modification of the $I$-method based argument above would be to include the low-high splitting idea in~\cite{BonaTzvet}. With $M>0$ to be fixed later, we would set 
\begin{align*}
u_{0}^{\o}=\P_{\leq M}u_{0}+\P_{>M}u_{0}^{\o},
\end{align*} and write 
\begin{align*}
u= S(t)\P_{>M}u_{0}^{\o}+S(t)\P_{\leq M}u_{0} +v_{\text{high}}+v_{\text{low}},
\end{align*}
where $v_{\text{high}}$ solves \eqref{BBMv} with $v_{\text{high}}|_{t=0}=\P_{>M}u_{0}^{\o}$ and $v_{\text{low}}$ solves a difference equation with $v_{\text{low}}|_{t=0}=\P_{\leq M}u_{0}$. Modifying the proof of Theorem~\ref{Thm:ASLWP} shows $v_{\text{high}}$ exists almost surely up to any time $T>0$ by choosing $M=M(T)$ large enough. We then try to show $v_{\text{low}}$ is global in-time by applying the $I$-method with the modified energy 
\begin{align*}
E(Iv_{\text{low}})(t)=\frac{1}{2}\| Iv_{\text{low}}(t)\|_{H^{1}(\T)}^{2}. 
\end{align*}
To bound the growth of $E(Iv_{\text{low}})$, we must deal with the term 
\begin{align}
\int_{\T} (\dx Iv_{\text{low}})\cdot Iv_{\text{low}} \cdot I \P_{>M}S(t)u_{0}^{\o}. \label{rmk1}
\end{align}
With $2N<M$, the proof of Lemma~\ref{lemma:izintmoment} gives
\begin{align*}
\prob\big( M^{ \al+\frac{1}{2}-s}N^{-(1-s)}\|I \P_{>M}S(t)u_{0}^{\o}\|_{L^{2}_{x}}> \ld \big) \leq \frac{C}{\ld^{2}}.
 \end{align*}
Choosing $N$ such that $M\sim N^{k}$ we obtain a non-positive power of $N$ in estimating \eqref{rmk1} provided \begin{align}
k\geq 2 \bigg( 1-\frac{\dl}{1-2\al+2\dl} \bigg). \label{k1}
\end{align}
 Applying Gronwall's inequality, we get a blow-up time 
\begin{align*}
T^{\ast}(N)\sim \frac{\log\bigg( 1+\frac{B(N)^{2}}{ \sqrt{E(Iv_{\text{low}})(0)} N^{-\beta }}  \bigg)}{B(N)},
\end{align*}
where $B(N)$ is almost surely bounded when $k\geq 2$ and grows polynomially in $N$ otherwise.
We conclude provided there exists $\rho\geq 0$ such that 
\begin{align*}
\sqrt{E(Iv_{\text{low}})(0)}N^{-\beta} \les N^{1-s}M^{s+\frac{1}{2}-\al+}N^{-\beta} \|u_{0}^{\o}\|_{H^{\al-\frac{1}{2}-}(\T)}\les N^{-\rho},
\end{align*}
but as $\beta=\tfrac{3}{2}-$ (see Lemma~\ref{lemma:Iv2commutator}), we require 
\begin{align*}
k\leq 1+\frac{2\al}{2\al+1-2\dl}.
\end{align*}
This final condition fails to agree with \eqref{k1} unless $\al=\tfrac{1}{2}$. Thus, we elected to present the simpler argument in this paper.

\end{remark}

\section{Norm inflation at arbitrary data }\label{section:norminf}
We establish in this section norm inflation at arbitrary data for the BBM equation~\eqref{BBM0} in negative regularity spaces. 
\begin{theorem}\label{thm:norminfflp} Let $\M=\R$ or $\T$, $1\leq p< \infty$, $s<0$ and fix $u_{0}\in \FL^{s,p}(\M)$. Then, given any $\e >0$, there exists a smooth solution $u_{\e}$ to \eqref{BBM0} on $\M$ and $t_{\e}\in (0,\ep)$ such that \begin{align*}
\|u_{\e}(0)-u_{0}\|_{\FL^{s,p}(\M)}<\e, \qquad \text{ and } 
\qquad \| u_\ep(t_{\e})\|_{\FL^{s,p}(\M)} > \ep^{-1}. 
\end{align*} 
\end{theorem}
Then, Theorem~\ref{thm:norminf} follows from Theorem~\ref{thm:norminfflp} upon putting $p=2$. 
In this section, we present the proof of Theorem~\ref{thm:norminfflp}. 
It suffices to establish the following result. We denote by $\mathcal{C}(\T)=C^{\infty}(\T)$ and $\mathcal{C}(\R)=\S(\R)$.

\begin{proposition}\label{prop:sequenceinf}  Let $\M=\R$ or $\T$, $1\leq p< \infty$, $s<0$, and fix $u_{0}\in \mathcal{C}(\mathcal{M})$. Then given any $n\in \N$, there exists a smooth solution $u_{n}$ to the BBM equation \eqref{BBM0} and $t_{n}\in (0,\frac{1}{n})$ such that \begin{equation}
\|u_{n}(0)-u_{0}\|_{\FL^{s,p}(\M)}<\frac{1}{n}, \quad \text{and} \quad \|u_{n}(t_{n})\|_{\FL^{s,p}(\M)}>n. \label{nnorminf}
\end{equation}
\end{proposition}
To see how Theorem \ref{thm:norminfflp} follows from Proposition \ref{prop:sequenceinf}, fix $u_{0}\in 
\FL^{s,p}(\M)$ and $s<0$. By density, we can find a sequence $\{ u_{0,k}\}_{k\in \N}\in \CC(\MM)$ such that, for $k$ sufficiently large, we have
\begin{equation}
\|u_{0,k}-u_{0}\|_{\FL^{s,p}(\M)}<\frac{1}{k}. \label{PTOT1}
\end{equation}
For each fixed $k$, Proposition \ref{prop:sequenceinf} implies there exists solutions $\{u_{n,k}\}_{n\in \N}$ to \eqref{BBM0} such that 
\begin{equation}
\|u_{n,k}(0)-u_{0,k}\|_{\FL^{s,p}(\M)}<\frac{1}{n} \qquad \text{and} \qquad \|u_{n,k}(t_{n})\|_{\FL^{s,p}(\M)}>n. \label{PTOT2}
\end{equation}
Now given $\ep>0$, set $u_{\ep}=u_{n,n}$ where $n\in \N$ is fixed such that $n\geq \frac{1}{2\ep}$. Combining \eqref{PTOT1} and \eqref{PTOT2} we obtain \eqref{nnorminf}, completing the proof.

In order to prove Proposition~\ref{prop:sequenceinf}, we follow the argument in~\cite{HiroInflation} which we set up in the next subsection and complete in Subsection~\ref{sec:52}.

\subsection{Binary trees, power series expansions and multilinear estimates}\label{sec:51}
In this section, we will briefly describe the power series expansion indexed by binary trees, arising in the works \cite{HiroInflation, Christ1}. We then establish multilinear estimates controlling the terms in the power series. 
We begin by establishing the following local well-posedness result for BBM~\eqref{BBM0} in the Fourier-Lebesgue spaces. 

\begin{lemma}\label{lemma:lwpflp} Let $\M =\R $ or $\T$, $1\leq p\leq \infty$ and $s\geq \max\{0, \frac{1}{2}-\frac{1}{p} \}$, with strict inequality when $p>2$. Then for each $u_{0}\in \mathcal{F}L^{s,p}(\MM)$, there exists a time $T\sim \|u_{0}\|_{\mathcal{F}L^{s,p}(\MM)}^{-1}>0$ and a unique solution $u\in C([0,T]; \mathcal{F}L^{s,p}(\MM))$ to the BBM equation \eqref{BBM0} with $u|_{t=0}=u_0$.
\end{lemma}

The proof of this result follows by a fixed point argument using the following bilinear estimate.   

\begin{lemma}\label{lemma:flspineq} Let $\M =\R $ or $\T$, $1\leq p\leq  \infty$ and $s\geq \max\{0, \frac{1}{2}-\frac{1}{p} \}$, with strict inequality when $p>2$. Suppose that $u,v\in\mathcal{F}L^{s,p}(\M)$. Then $\varphi(D_{x})(uv)\in \mathcal{F}L^{s,p}(\M)$ and \begin{equation}
\| \varphi(D_{x})(uv) \|_{\mathcal{F}L^{s,p}(\M)}\lesssim \| u\|_{\mathcal{F}L^{s,p}(\M)}\| v\|_{\mathcal{F}L^{s,p}(\M)}. \label{FLspineq}
\end{equation}
\end{lemma}
\begin{proof} 
Consider first the case when $1\leq p\leq 2$.
By duality, it suffices to show that \begin{align}
\bigg \vert \int_{\M}\int_{\M}\frac{\jb{\xi}^{s}}{\jb{\xi-\xi_{1}}^{s}\jb{\xi_{1}}^{s}}\varphi(\xi) w(\xi)\ft u(\xi-\xi_{1})\ft v(\xi_{1})\, d\xi_{1}d\xi \bigg \vert \lesssim  \|\ft u\|_{L^{p}}\|\ft v\|_{L^{p}}, \label{dualityestimate}
\end{align}  where $w\in L^{p'} $ satisfying $\|w\|_{L^{p'}}=1$ and $1/p+1/p'=1$. 
As $s\geq 0$, the trivial inequality $\jb{\xi}^{s}\lesssim\jb{\xi-\xi_{1}}^{s}\jb{\xi_{1}}^{s}$ implies 
$$ \eqref{dualityestimate} \lesssim \bigg \vert\int_{\M}\int_{\M}\varphi(\xi) w(\xi)\ft u(\xi-\xi_{1})\ft v(\xi_{1})\, d\xi_{1}d\xi\bigg \vert.$$ 
By Young's and H\"{o}lder's inequalities, we have\begin{align*}
\bigg \vert\int_{\M}\int_{\M}\varphi(\xi) w(\xi)\ft u(\xi-\xi_{1})\ft v(\xi_{1})\, d\xi_{1}d\xi \bigg \vert & 
\les \|\vp(\xi)w(\xi)\|_{L^{\frac{p'}{2}}}\| \ft u \ast \ft v\|_{L^{\frac{p}{2-p}}} \\
& \les \|\ft u\|_{L^{p}}\|\ft v\|_{L^{p}}\| w\|_{L^{p'}}\|\varphi(\xi)\|_{L^{p'}}.
\end{align*}
 Finally, $\|\varphi(\xi)\|_{L^{p'}}<\infty$ as $p'>1$, verifying  \eqref{dualityestimate} when $1\leq p\leq 2$.
The estimate \eqref{FLspineq} can also be deduced by appealing to multilinear interpolation between the trivial inequality when $p=1$ and the result of Lemma~\ref{lemma:bonatzvet} when $p=2$.
 
  For the case when $2<p<\infty$, first notice that the above arguments lead to the following stronger estimate when $p=2$: for any $s\geq 0$,
 \begin{equation}
 \|\jb{\dx}^{-\frac{1}{2}+}(uv)\|_{\mathcal{F}L^{s,2}}\lesssim \|u\|_{\mathcal{F}L^{s,2}}\|v\|_{\mathcal{F}L^{s,2}}. \label{prod1}
 \end{equation} Fix $2<p<\infty$ and let $s=\frac{1}{2}-\frac{1}{p}+s_{0}$ for some $s_{0}>0$. Then by the embeddings $\l^{2}\subset \l^{p}$ for any $p>2$, H\"{o}lder's inequality and with some $\ep=\ep(p)$ sufficiently small, we have using \eqref{prod1}, \begin{align*}
 \| \varphi(D_{x})(uv)\|_{\mathcal{F}L^{s,p}} & \lesssim \|\jb{\dx}^{-\left(\frac{1}{2}+\frac{1}{p}-\ep \right)}(uv)\|_{\mathcal{F}L^{s_{0}-\ep,2}} \\
 & \lesssim \|u\|_{\mathcal{F}L^{s_{0}-\ep,2}}\|v\|_{\mathcal{F}L^{s_{0}-\ep,2}} \\
 & \lesssim \|u\|_{\mathcal{F}L^{s,p}}\|v\|_{\mathcal{F}L^{s,p}}.
 \end{align*} 
 The case when $p=\infty$ follows from H\"{o}lder's inequality. 
\end{proof}

\begin{remark} \rm The estimate \eqref{dualityestimate} is false if $p>2$ and $s<\frac{1}{2}-\frac{1}{p}$. To see this, it suffices to show that \eqref{dualityestimate} fails. For this, let $A\gg 1$ and take $\ft u(\xi)=\ft v(\xi)=\ind_{[-A,A]}(\xi)$ and $w(\xi)=\ind_{[0,1]}(\xi)$. Then the left hand side of \eqref{dualityestimate} behaves like $A^{1-2s}$ while the right hand side behaves like $A^{\frac{2}{p}}$. Thus we obtain a contradiction when $p$ and $s$ are as above and $A$ becomes large. 
\end{remark}

Given $u_{0}\in \mathcal{F}L^{p}(\MM)$, Lemma \ref{lemma:lwpflp} gives the existence of a unique solution $u$ to BBM~\eqref{BBM0} in the sense that there exists $T\sim \|u_{0}\|_{\mathcal{F}L^{p}(\MM)}^{-1}$ such that for each $t\in [0,T]$, $u$ satisfies $$u(t)=S(t)u_{0}+\I^{2}[u](t),$$ where $\I^{2}[u]:=\I[u,u]$ and $\I$ is the bilinear Duhamel integral operator 
\begin{align}
\I[u_{1},u_{2}](t):=-\frac{i}{2}\int_{0}^{t} S(t-t')\varphi(D_{x})(u_{1}(t')u_{2}(t'))dt'. \label{duhamelop}
\end{align}

 In order to set-up the necessary notation for the power series expansion of $u$ indexed by trees, we first restate the terminology used in \cite{HiroInflation} 
 for the binary trees we work with. 
 
 \begin{definition} \label{DEF:tree} \rm
(i) Given a partially ordered set $\TT$ with partial order $\leq$, 
we say that $b \in \TT$ 
with $b \leq a$ and $b \ne a$
is a child of $a \in \TT$,
if  $b\leq c \leq a$ implies
either $c = a$ or $c = b$.
If the latter condition holds, we also say that $a$ is the parent of $b$.

\smallskip 

\noi
(ii) 
A tree $\TT$ is a finite partially ordered set,  satisfying
the following properties:
\begin{itemize}
\item Let $a_1, a_2, a_3, a_4 \in \TT$.
If $a_4 \leq a_2 \leq a_1$ and  
$a_4 \leq a_3 \leq a_1$, then we have $a_2\leq a_3$ or $a_3 \leq a_2$,

\item
A node $a\in \TT$ is called terminal, if it has no child.
A non-terminal node $a\in \TT$ is a node 
with exactly two children,

\item There exists a maximal element $r \in \TT$ (called the root node) such that $a \leq r$ for all $a \in \TT$,

\item $\TT$ consists of the disjoint union of $\TT^0$ and $\TT^\infty$,
where $\TT^0$ and $\TT^\infty$
denote the collections of non-terminal and terminal nodes, respectively.
\end{itemize}

\end{definition}

We recall some basic combinatorial properties of binary trees. 

\begin{lemma}\label{lemma:counting} Let $\TT$ be a binary tree. 
The number of non-terminal $|\TT^{0}|$ and terminal $|\TT^{\infty}|$ nodes in $\TT$ are $j$ and $j+1$ respectively, where $j\in \N \cup \{0\}$. Consequently, $|\TT|=2j+1$. Let $\boldsymbol{T}(j)$ denote the set of all trees with $j$ parent nodes. Then there exists a constant $C_{0}>0$ such that 
\begin{equation} \label{tjcnt}
|\pmb{T}(j)| \leq C_{0}^{j}.
\end{equation}
\end{lemma}
 
 For a proof of \eqref{tjcnt}, we refer to the argument in \cite[Lemma 2.3]{HiroInflation} which can be adapted easily for binary trees.

 We have an injective map 
 $$\Psi_{\phi}:\bigcup_{j=1}^{\infty}\pmb{T}(j) \mapsto \mathcal{D}'(\MM \times [0,T]),$$ 
 which encodes the nodes of a given tree $\TT \in \pmb{T}(j) $ as $j$-times iterated Duhamel operators acting on inputs $S(t)\phi$. More precisely, given a binary tree, we replace the non-terminal nodes by the bilinear Duhamel operator \eqref{duhamelop} with its children as arguments $u_1$ and $u_2$. Then, each terminal node is replaced by the linear solution $S(t)\phi$. 
 Set $$\Xi_{j}(\phi)(t):= \sum_{\TT \in \pmb{T}(j)} \Psi_{\phi}(\TT).$$
 Then we have the following multilinear estimates. 
 
 \begin{lemma}\label{lemma:multest1} There exists $C>0$ such that for all $j\in \N$, we have the following:
Given any $\phi \in \FL^{1}(\M)$ and $\psi \in L^{2}(\M)$, we have   
 \begin{align}
 \|\Xi_{j}(\phi)(t)\|_{\mathcal{F}L^{1}} & \leq C^{j}t^{j}\|\phi\|_{\mathcal{F}L^{1}}^{j+1} \label{xi1} \\
 \|\Xi_{j}(\psi)(t)\|_{\mathcal{F}L^{\infty}} &\leq C^{j}t^{j}\|\psi\|_{L^{2}}^{j+1}. \label{xi2}
 \end{align} 
 Furthermore, for any $1\leq q\leq \infty$ and $u_{0}\in \mathcal{F}L^{q}(\MM)\cap \FL^{1}(\MM)$, 
 \begin{equation}
  \|\Xi_{j}(u_{0}+\phi)(t)-\Xi_{j}(\phi)(t)\|_{\mathcal{F}L^{q}} \leq C^{j}t^{j}\|u_{0}\|_{\mathcal{F}L^{q}}(\|u_{0}\|_{\mathcal{F}L^{1}}^{j}+\|\phi\|_{\mathcal{F}L^{1}}^{j}). \label{xi3}
 \end{equation}
 \end{lemma}
 
 \begin{proof}
 Estimates \eqref{xi1} and \eqref{xi3} are proved exactly as in \cite[Lemma 2.5 and Lemma 2.6]{HiroInflation} by using the unitarity of $S(t)$ in $\FL^1$ and Young's inequality. For \eqref{xi2}, we notice that H\"older's inequality implies
 \begin{align*}
\|\mathcal{I}[ u_1, u_2 ]\|_{\FL^{\infty}} \les |t| \| u_1\|_{L^{2}} \|u_2 \|_{L^2}.
\end{align*}
now for a given $\mathcal{T} \in T(\TT)$, $\Psi_{\psi}(\TT)$ is essentially $j=|\TT^0|$-many iterative compositions of the Duhamel integral operator $\I^2[\psi]$. Thus we first apply the above estimate followed by successive applications of Lemma~\ref{productestimateBT}, namely 
\begin{align*}
\| \I[ u_1,u_2]\|_{L^2} \les |t| \| u_1\|_{L^{2}} \|u_2 \|_{L^2}.
\end{align*}
 \end{proof}
 
 \begin{remark} \rm
 The estimate \eqref{xi2} differs from that in \cite[(2.14) in Lemma 2.5]{HiroInflation} because we have made use of the explicit smoothing of the Duhamel operator for BBM.
 \end{remark}
 
 The estimate \eqref{xi1} and Lemma~\ref{lemma:counting} imply the power series expansion
 \begin{align*}
\sum_{j=0}^{\infty} \Xi_{j}(u_0) = \sum_{j=0}^{\infty} \sum_{\TT \in \pmb{T}(j)} \Psi_{u_0}(\TT),
\end{align*}
is absolutely convergent in $C([0,T]; \FL^{1}(\M))$ provided $T\les \| u_0\|_{\FL^1}^{-1}$ and that the solution $u\in C([0,T];\FL^{1}(\M))$ with $u|_{t=0}=u_0$ can be represented as
\begin{align*}
u= \sum_{j=0}^{\infty} \Xi_{j}(u_0).
\end{align*}
This is the power series representation of $u$ indexed by trees. 
 
We now begin to proceed towards the construction of the smooth solutions $u_{n}$ stated in Proposition~\ref{prop:sequenceinf}.
Define $\phi_{n}$ through its Fourier transform by 
\begin{equation}
\ft \phi_{n}(\xi):= R\big\{\ind_{-N+ Q_A} (\xi)+ \ind_{N+ Q_A}(\xi)\big\} \label{phin}
\end{equation}
\noi
where $Q_A = [-2A, 2A]$,
$R = R(N) \geq 1 $,  and $A = A(N)\geq 1$ satisfying
\begin{align}
 \|u_0\|_{\F L^1}\ll RA, 
\qquad \text{and}
\qquad A\ll N, 
\label{phi1a}
\end{align} 
For fixed $u_{0}\in \mathcal{C}(\mathcal{M})$, set
 \begin{equation} \label{u0n}
 u_{0,n}=u_{0}+\phi_{n}.
\end{equation} For each $n$, $\ft \phi_{n}$ is even and real-valued and hence $u_{0,n}$ is also real-valued. Let $u_{n}$ be the solution of \eqref{BBM0} with initial data $u_{n}\vert_{t=0}=u_{0,n}$. We have the following series expansion
 \begin{align}
 u_{n}=\sum_{j=0}^{\infty}\Xi_{j}(u_{0,n}), \label{powerseriesun}
 \end{align}
 on $[-T,T]$ as long as 
\begin{align}
T\lesssim (\|u_{0}\|_{\mathcal{F}L^{1}}+RA)^{-1}\sim (RA)^{-1}. \label{timelocal}
\end{align}

We now state some further multilinear estimates that exploit the explicit expression of $\phi_n$. 

\begin{proposition}\label{prop:multlin2}
 Let $1\leq p <\infty$, $s<0$, $u_{0}\in \mathcal{C}(\MM)$ satisfying \eqref{phi1a} and $\phi_{n}$ and $u_{0,n}$ as in \eqref{phin} and \eqref{u0n} respectively. For any $j\in \N$, the following estimates hold:
\begin{align}
\|u_{0,n}-u_{0}\|_{\mathcal{F}L^{s,p}} &\lesssim RA^{\frac{1}{p}}N^{s}, \label{flsp1} \\
\|\Xi_{0}(u_{0,n})(t)\|_{\mathcal{F}L^{s,p}} &\lesssim 1+RA^{\frac{1}{p}}N^{s}  \label{flsp2} \\
\|\Xi_{1}(u_{0,n})(t)-\Xi_{1}(\phi_{n})(t)\|_{\mathcal{F}L^{s,p}} &\lesssim t\|u_{0}\|_{ \mathcal{F}L^{p}}RA^{\frac{1}{p}}.   \label{flsp3}  \\
\|\Xi_{j}(u_{0,n})(t)\|_{\mathcal{F}L^{s,p}} &\lesssim C^j t^{j}(RA)^j ( \|u_0\|_{\FL^p}+Rf_{p}(A)), \label{flsp4}
\end{align} 
where 
\begin{equation}
f_{p}(A):=
\begin{cases}
 1 & \textup{if} \,\, \,s<-\frac{1}{p}, \\
 \left( \log A \right)^{\frac{1}{p}}& \textup{if} \,\,\, s=-\frac{1}{p}, \\
 A^{\frac{1}{p}+s} & \textup{if} \,\, \, s>-\frac{1}{p}.
\end{cases} \label{fpA}
\end{equation} 
\end{proposition}

\begin{proof}
The estimates \eqref{flsp1}, \eqref{flsp2} and \eqref{flsp3} are easy consequences of Lemma~\ref{lemma:multest1}, \eqref{phin} and \eqref{phi1a}. 
For \eqref{flsp4}, we write 
\begin{align}
\|\Xi_{j} (u_{0,n})(t)\|_{\FL^{s,p}} \leq \| \Xi_{j} (u_{0,n})(t) - \Xi_{j}(\phi_n)(t)\|_{\FL^{p}} 
+ \| \Xi_{j}(\phi_n)(t)\|_{\FL^{s,p}}.  \label{xijdiff}
\end{align}
Now \eqref{xi3}, \eqref{phin} and \eqref{phi1a} imply 
\begin{align*}
\|\Xi_{j}(u_{0,n})(t)-\Xi_{j}(\phi_n )(t)\|_{\mathcal{F}L^{p}}& \leq C^{j}t^{j}\|u_{0}\|_{\mathcal{F}L^{p}}(\|u_{0}\|_{\mathcal{F}L^{1}}^{j}+\|\phi_n \|_{\mathcal{F}L^{1}}^{j})  \\
& \les C^{j} t^j \|u_0\|_{\FL^p} (RA)^j. 
\end{align*}
Meanwhile, as the support of $\ft \phi_n$ is two disjoint intervals of width approximately $ A$, we see that for fixed $\TT \in T(j)$, the iterated convolution structure of $\mathcal{F}\{\Psi_{\phi_n}(\TT)\}$ implies 
$\supp \mathcal{F}\{\Xi_{j}(\phi_n)\} $ is contained within at most $2^{j+1}$ intervals of width approximately $ A$. As $s<0$, $\jb{\xi }^{s}$ is decreasing in $|\xi|$ and hence by \eqref{xi2}, we have
\begin{align*}
\| \Xi_{j}(\phi_n)(t)\|_{\FL^{s,p}} & \leq \| \jb{\xi}^{s}\|_{L^{p}_{\xi}(\supp \mathcal{F}\{\Xi_{j}(\phi_n)\})} \|\Xi_{j}(\phi_{n})\|_{\FL^{\infty}} \\
& \les  \| \jb{\xi}^{s}\|_{L^{p}_{\xi}(C^{j}Q_A)} t^{j} (RA^{\frac{1}{2}})^{j+1} \\
& \les C^j t^{j}f_{p}(A) (RA^{\frac{1}{2}})^{j+1},
\end{align*}
Returning to \eqref{xijdiff} we have shown
\begin{align*}
\|\Xi_{j} (u_{0,n})(t)\|_{\FL^{s,p}}  & \les C^j t^{j}(RA)^j ( \|u_0\|_{\FL^p}+Rf_{p}(A)),
\end{align*}
which is \eqref{flsp4}.
\end{proof}

The following estimate shows the term $\Xi_{1}[\phi_n]$ is culpable for the norm-inflation phenomenon.
The argument given below is essentially the same as similar arguments in \cite{BonaTzvet, Panthee}. For completeness we include it here but adapted to the data \eqref{phin}.

\begin{proposition}\label{prop:norminf} Let $\phi_{n}$ be as in \eqref{phin} and $s<0$. Then, for $0<t\ll A$, we have 
\begin{align*}
\|\Xi_{1}(\phi_{n})(t)\|_{\mathcal{F}L^{s,p}}\gtrsim tR^{2}A. 
\end{align*}
\end{proposition}
\begin{proof}
We have $$ \Xi_{1}(\phi_{n})(t)=-\frac{i}{2}\int_{0}^{t}S(t-t')\varphi(D_{x})(S(t')\phi_{n}S(t')\phi_{n})\, dt'.$$ Taking the Fourier transform, we obtain 
\begin{align*}
\mathcal{F}_{x\rightarrow \xi}&[\Xi_{1}(\phi_{n})(t)](\xi)= -\frac{i}{2}\int_{0}^{t} \mathcal{F}_{x\rightarrow \xi}\left[S(t-t')\varphi(D_{x})(S(t')\phi_{n}S(t')\phi_{n}) \right](\xi)\, dt' \\
& = -\frac{i}{2}\int_{0}^{t} e^{-i(t-t')\varphi(\xi)}\varphi(\xi)\mathcal{F}_{x\rightarrow \xi}\left[(S(t')\phi_{n}S(t')\phi_{n}) \right]\, dt' \\
& = -\frac{i}{2}e^{-it\varphi(\xi)}\varphi(\xi)\int_{0}^{t}e^{it'\varphi(\xi)}\int_{\R} \ft \phi_{n}(\xi_{1})\ft \phi_{n}(\xi-\xi_{1})e^{-it'\varphi(\xi_{1})}e^{-it'\varphi(\xi_{1})}d\xi_{1}\, dt' \\
& = -\frac{i}{2}e^{-it\varphi(\xi)}\varphi(\xi) \int_{\R}\ft{\phi}_{n}(\xi_{1})\ft \phi_{n}(\xi-\xi_{1}) \int_{0}^{t}e^{-it'\theta(\xi,\xi_{1})} dt'\, d\xi_{1},
\end{align*}
where \begin{equation}
\theta(\xi,\xi_{1}):= \varphi(\xi_{1})+\varphi(\xi-\xi_{1})-\varphi(\xi)=\frac{\xi \xi_{1}(\xi-\xi_{1})(\xi^{2}-\xi\xi_{1}+\xi_{1}^{2}+3)}{(1+\xi_{1}^{2})[1+(\xi-\xi_{1})^{2}](1+\xi^{2})}. \label{thetaphase}
\end{equation} Integrating over $t'$ yields \begin{equation*}
\mathcal{F}_{x\rightarrow \xi}[\Xi_{1}(\phi_{n})(t)](\xi) = \frac{1}{2}e^{-it\varphi(\xi)}\varphi(\xi)\int_{\R}\ft{\phi}_{n}(\xi_{1})\ft \phi_{n}(\xi-\xi_{1}) \frac{e^{-it' \theta(\xi,\xi_{1})}-1}{\theta(\xi,\xi_{1})} d\xi_{1}
\end{equation*} Writing $I_{1}:=-N+Q_{A}$, $I_{2}:=N+Q_{A}$ and in view of \eqref{phin}, 
\begin{align*}
\mathcal{F}_{x\rightarrow \xi}[\Xi_{1}(\phi_{n})(t)](\xi) &= \frac{R^{2}}{2}e^{-it\varphi(\xi)}\varphi(\xi)\int_{\substack{\xi_{1}\in I_{1}\cup I_{2} \\ \xi-\xi_{1}\in I_{1}\cup I_{2}}}
\frac{e^{-it \theta(\xi,\xi_{1})}-1}{\theta(\xi,\xi_{1})} d\xi_{1} \\
& =  \frac{R^{2}}{2}e^{-it\varphi(\xi)}\varphi(\xi)\left\lbrace \int_{A_{1}(\xi)}\cdot +\int_{A_{2}(\xi)}\cdot \right\rbrace \\
&=: g_{1}(t,\xi)+g_{2}(t,\xi),
\end{align*}
where \begin{align*}
A_{1}(\xi)&:= \{\xi_{1} \, : \, \xi_{1}\in I_{1}, \, \xi-\xi_{1}\in I_{1} \,\, \text{or} \,\, \xi_{1}\in I_{2},\,\xi-\xi_{1}\in I_{2}  \},\\
A_{2}(\xi)&:= \{\xi_{1} \, : \, \xi_{1}\in I_{1}, \, \xi-\xi_{1}\in I_{2} \,\, \text{or} \, \xi_{1}\in I_{2},\,\,\xi-\xi_{1}\in I_{1}  \}.
\end{align*}
With $f_{j}(x,t):=\mathcal{F}^{-1}_{\xi \rightarrow x}(g_{j}), \,\, j=1,2$, we have $$
\Xi_{1}(\phi_{n})(t)=f_{1}+f_{2}.$$ 
If $\xi_{1}\in A_{1}(\xi)$, then $\xi\in 2I_{1}$ or $\xi\in 2I_{2}$, while if $\xi_{1}\in A_{2}(\xi)$, then $\xi \in Q_{2A}$. As $A\ll N$, the supports of the $g_{j}$ are disjoint which implies $$\|\Xi_{1}(\phi_{n})(t)\|_{\mathcal{F}L^{s,p}}\sim \|f_{1}(t,\cdot)\|_{\mathcal{F}L^{s,p}}+\|f_{2}(t,\cdot)\|_{\mathcal{F}L^{s,p}}.$$ The dominant contribution to the $H^{s}$ norm of $\Xi_{1}(\phi_{n})$ arises from that of $f_{2}$. 
Indeed, we have 
$$\|f_{1}\|_{\mathcal{F}L^{s,p}}\lesssim \frac{tR^{2}A^{1+\frac 1p}}{N^{1-s}}\ll tR^{2}A,$$
where second inequality follows from $A\ll N$ and $s<0$.
In the region $A_{2}(\xi)$, $|\xi_{1}| \sim |\xi-\xi_{1}|\sim N$ and $\xi \in Q_{2A}$. From \eqref{thetaphase}, we find $|\theta(\xi,\xi_{1})|\lesssim A^{-1}$, and hence for $0<t\ll A$, we have
$$ \Im \frac{e^{it \theta(\xi,\xi_{1})}-1}{\theta(\xi,\xi_{1})} \geq \frac{1}{2}t.$$
Furthermore, for $\xi \in Q_{A}\setminus Q_{1/4}$ we have $|\xi||\text{meas}(A_{2}(\xi))|\gtrsim A$ and hence 
\begin{align*}
\|f_{2}(t,\cdot)\|_{\mathcal{F}L^{s,p}}&\sim R^{2}\left(\int_{\R}\jb{\xi}^{ps}|\varphi(\xi)|^{p}\bigg\vert \int_{A_{2}(\xi)}\frac{e^{-it \theta(\xi,\xi_{1})}-1}{\theta(\xi,\xi_{1})}d\xi_{1} \bigg \vert^{p} d\xi\right)^{1/p} \\
& \gtrsim |t|R^{2}\left(\int_{Q_{2A}}\jb{\xi}^{-p(2-s)}|\xi|^{p} |\text{meas}(A_{2}(\xi))|^{p} d\xi\right)^{1/p} \\
& \gtrsim |t|R^{2}\left(\int_{Q_{A}\setminus Q_{1/4}}\jb{\xi}^{-p(2-s)}|\xi|^{p} |\text{meas}(A_{2}(\xi))|^{p} d\xi\right)^{1/p} \\
& \gtrsim |t|R^{2}A\left(\int_{Q_{1}\setminus Q_{1/4}}\jb{\xi}^{-p(2-s)} d\xi\right)^{1/p} \\
& \gtrsim |t|R^{2}A,
\end{align*} as $A\geq 1$.
\end{proof}
\begin{remark} \rm Although the proof of Proposition \ref{prop:norminf} was stated for $\R$, it also holds on $\T$ using the same ideas and the obvious modifications.
\end{remark}

\subsection{Proof of Proposition \ref{prop:norminf}}\label{sec:52}
Fix $p\in [1,\infty)$ and $s<0$.
 In order to prove Proposition~ \ref{prop:sequenceinf}, it suffices to show, given $n\in \N$, the following properties hold:
\begin{align*}
 \textup{(i)} & \quad R A^\frac{1}{p} N^s \ll \tfrac{1}{n}, \\ 
 \textup{(ii)} & \quad T R A \ll 1,  \\ 
 \textup{(iii)} & \quad TR^{2}A\gg n,\\ 
 \textup{(iv)} & \quad TR^2 A \gg T^2 R^3 A^2 f_{p}(A), \\
 \textup{(v)} & \quad \|u_{0}\|_{\mathcal{F}L^{p}}\ll R f_{p}(A), \\
 \textup{(vi)} & \quad T\ll A, \quad \text{and}\quad  A \ll N.
 \end{align*}
for some particular choices of $A,R,T$ and $N$ all depending on $n$. 

To see why this is true, we have that condition (i) ensures by \eqref{flsp1} that the approximating data $u_{0,n}$ is close to $u_0$ which is the first part of \eqref{nnorminf}. Condition (ii) combined with \eqref{timelocal} implies the power series expansion \eqref{powerseriesun} converges in $C([0,T]; \FL^{1}(\T))$. Proposition~\ref{prop:norminf} and conditions (iii) and (vi) are responsible for the required growth to conclude norm inflation, while (iv) and (v) ensure that the first term of the expansion $\Xi_{1}(u_{0,n})$ dominates all other terms. We detail these last two deductions now. 
Namely, assuming (ii) and (v) hold, \eqref{flsp4} implies
\begin{align*}
\bigg\| \sum_{j=2}^{\infty}\Xi_{j}(u_{0,n})(T)\bigg\|_{\mathcal{F}L^{s,p}} &\lesssim \sum_{j=2}^{\infty}\bigg\| \Xi_{j}(u_{0,n})(T)\bigg\|_{\mathcal{F}L^{s,p}} \\
& \lesssim \sum_{j=2}^{\infty} (CTRA)^{j} (\|u_0 \|_{\FL^{p}}+Rf_{p}(A))  \\
& \lesssim  T^{2}R^{2}A^{2}Rf_{p}(A)  \sim T^{2}R^{3}A^{2}f_{p}(A).
\end{align*}
Then, assuming (i), (ii), (iii), (iv) and (v) hold and using Propositions~\ref{prop:multlin2} and \ref{prop:norminf}, we have
\begin{align*}
\|u_{n}(T)\|_{\mathcal{F}L^{s,p}}& \geq \|\Xi_{1}(\phi_{n})(T)\|_{\mathcal{F}L^{s,p}}-\|\Xi_{0}(u_{0,n})\|_{\mathcal{F}L^{s,p}}\\
& \hphantom{X}-\|\Xi_{1}(u_{0,n})(T)-\Xi_{1}(\phi_{n})(T)\|_{\mathcal{F}L^{s,p}}-\bigg\| \sum_{j=2}^{\infty}\Xi_{j}(u_{0,n})(T)\bigg\|_{\mathcal{F}L^{s,p}}\\
&\gtrsim TR^{2}A- (1+RA^{\frac{1}{p}}N^{s}) -T\|u_{0}\|_{\mathcal{F}L^{p}}RA^{\frac{1}{p}}-T^{2}R^{3}A^{2}f_{p}(A)\\
& \sim TR^{2}A \gg n.
\end{align*}
This verifies the second estimate of \eqref{nnorminf} at time $t_{n}:=T$. Finally choosing $N=N(n)$ sufficiently large, we conclude the proof of Proposition~\ref{prop:sequenceinf}.

It remains to verify (i)-(vi) hold. Notice that (iv) follows if we obtain $TRAf_{p}(A) \ll 1$. This is stronger than (ii), so we focus on obtaining (i) and (iii) through (vi). Recalling the definition of $f_{p}(A)$ from \eqref{fpA}, it is natural to consider the following three cases:

\medskip

\noi
$\bullet$
{\bf Case 1:} $s<-\frac 1p$

For $\dl>0$ small enough so that 
\begin{align*}
\frac{1}{p}+\bigg(2-\frac 1p \bigg)\dl <-s,
\end{align*} we choose
\begin{align*}
A = N^{1-\dl}, 
\quad 
R = N^{2\dl}, 
\quad \text{and} 
\quad
T= N^{-1-2\dl}.
\end{align*}
Then we check
\begin{align*}
&RA^{\frac 1p}N^{s}=N^{\frac 1p+\big(2-\frac 1p \big)\dl +s} \ll \frac 1n, \\
&TR^{2}A =N^{\dl} \gg n,\\
& TRAf_{p}(A)\sim TRA=N^{-\dl} \ll 1.
\end{align*}
Clearly (v) and (vi) are also satisfied.
\noi

\medskip

\noi
$\bullet$
{\bf Case 2:} $s=-\frac 1p$

We choose
\begin{align*}
A = \bigg( \frac{N}{\log N} \bigg)^{\frac 12}, 
\quad 
R = \bigg( \frac{N}{\log N} \bigg)^{\frac{1}{2p}}, 
\quad \text{and} 
\quad
T=\frac{1}{N^{  \frac{1+p}{2p}}(\log N)^{\frac{3-p}{2p}}}.
\end{align*}
Thus, by choosing $N=N(n)$ sufficiently large, we ensure
\begin{align*}
&RA^{\frac 1p}N^{s}= (\log N)^{-\frac{1}{p}} \ll \frac 1n, \\
&TR^{2}A =\frac{N^{\frac{1}{2p}}}{ (\log N)^{\frac{5}{2p}}} \gg n,\\
& TRAf_{p}(A)\sim TRA(\log A)^{\frac{1}{p}}\sim (\log N)^{-\frac{2}{p}}(\log N - \log \log N)^{\frac 1p}\sim (\log N)^{-\frac 1p} \ll 1.
\end{align*}
Furthermore, both (v) and (vi) are also satisfied.

\medskip

\noi
$\bullet$
{\bf Case 3:} $-\frac 1p <s <0$

We choose
\begin{align*}
A = N^{\dl p}, 
\quad 
R = N^{-s-\dl-\ta}, 
\quad \text{and} 
\quad
T=N^{2s+3\ta +(2-p)\dl},
\end{align*}
where $0<\ta \ll \dl \leq \tfrac{1}{3p}$ are sufficiently small so that 
\begin{align}
-s>\max \bigg( \frac{2(\ta+\dl)}{1+\dl p}, \frac{3\ta}{2} -(p-1)\dl \bigg). \label{case3scond}
\end{align}
Then 
\begin{align*}
&RA^{\frac 1p}N^{s}= N^{-\ta} \ll \frac 1n, \\
&TR^{2}A =N^{\ta} \gg n,\\
& TRAf_{p}(A)\sim TRA^{1+\frac 1p +s} =N^{(1+\dl p)s+2(\ta +\dl)} \ll 1.
\end{align*}
The second condition in \eqref{case3scond} ensures (vi) holds. Meanwhile, we also satisfy (v) because when 
 $\dl \leq \frac{1}{3p}$, we have  
$\tfrac{1}{1-\dl p}\leq \frac{2}{1+\dl p}$.

\begin{appendix}

\section{On well-posedness of BBM below $L^{2}(\T)$ with non-Gaussian randomised initial data}\label{app:Gauss}

In this appendix, we discuss how we may extend the local and global well-posedness results of Theorems~\ref{Thm:ASLWP} and \ref{Thm:ASGWP} for BBM~\eqref{BBMD} to more general random initial data of the form:
\begin{align}
u^{\o}_{0}(x)=\sum_{n\in \Z}\frac{g_{n}(\o)}{\jb{n}^{\al}} e^{inx}, \label{newdata}
\end{align}
where the family of (not necessarily Gaussian) complex-valued $\{g_{n}\}_{n\in \Z}$ random variables satisfy the following assumptions:
\begin{enumerate}[\normalfont (i)]
\setlength\itemsep{0.3em}
\item \label{item1} the random variables $\{g_{n}\}_{n\in \mathbb{N}\cup\{0\}} $ are independent, 
\item \label{item2}$g_{-n}:=\cj{g_{n}}$ for  $n\in \mathbb{N}\cup \{0\}$ and $g_{0}$ is real,
\item \label{item3} $\E[g_{n}]=0$ and $\E[ |g_n|^2]=1$,
\item \label{item4} there exist $C_0, C_1>0$ such that for all $\g\in \R$ and for all $n\in \Z$, we have 
\begin{align*}
\E[e^{\g |g_n|}] \leq C_0 e^{C_1 \g^{2}},
\end{align*}
\item \label{item5} there exists an angle $\ta$ satisfying\footnote{Note $840=\text{lcm}(2,3,5,7,8)$.} $840\, \ta \neq 0$ mod $2\pi$, such that the law of $e^{i\ta}g_{n}$ is the same as the law of $g_{n}$.
\end{enumerate}

A computation shows the random distribution $u_0^{\o}$ given in \eqref{newdata} belongs to $H^{\al-\frac{1}{2}-}(\T)$ almost surely. If we additionally impose the following non-degeneracy condition:
\begin{align*}
\textup{(vi)} \,\,\, \text{there exists}\, c>0 \, \,\text{such that} \, \limsup_{n\rightarrow \infty} \prob( |g_n|\leq c) <1,
\end{align*}
then the argument in \cite[Lemma B.1]{BTlocal} shows $u_0^{\o}$ does not belong to $H^{\al-\frac{1}{2}}(\T)$ almost surely. In view of the global well-posedness of BBM~\eqref{BBMD} in $L^{2}(\T)$, we focus on $\al\leq \tfrac{1}{2}$.

For simplicity, in the following we make the additional assumption
\begin{align*}
\textup{(vii)} \,\,\, g_0\equiv 0.
\end{align*}
We stress that this assumption only reduces the number of cases we must consider in the proof of Lemma~\ref{LEM:NZ} below. We handle the remaining cases coming from removing assumption (vii) in that proof using similar analysis and they yield the same (overall) restrictions on $\al$ and $s$ as stated in Lemma~\ref{LEM:NZ}.

We have three main points to discuss: 
\begin{enumerate}[\normalfont (I)]
\item the regularity and integrability properties of the stochastic objects: $z^{\o}(t)=S(t)u_0^{\o}$, the random linear solution to BBM~\eqref{BBM0} with initial data~\eqref{newdata}, and $\NN(z^{\o})$, where $\NN$ is defined in \eqref{nonlin},
\item almost sure local well-posedness for BBM~\eqref{BBMD} with initial data \eqref{newdata},
\item  almost sure global well-posedness for BBM~\eqref{BBMD} with initial data \eqref{newdata}.
\end{enumerate}

\medskip

\noi
\underline{$\bullet$ \textbf{(I):}}
We begin with the random linear solution $z^{\o}$. By assumption~\eqref{item4}, the argument in \cite[Lemma 3.1]{BTlocal} implies, for any $(a_n)\in \l^{2}$, we have 
\begin{align}
\bigg\| \sum_{n\in \Z} a_{n}g_{n} \bigg\|_{L^{p}(\O)} \les p^{\frac{1}{2}}\|a_{n}\|_{\l^{2}_{n}} \sim p^{\frac{1}{2}}\bigg\| \sum_{n\in \Z} a_{n}g_{n} \bigg\|_{L^{2}(\O)} \label{almostgauss}
\end{align}
for any $p\geq 2$. Now we note that, in Proposition~\ref{prop:reg}, we may weaken the assumption that the stochastic process $X_{k}(t)\in \mathcal{H}_{\leq \l}$ for each $t\in \R_{+}$ to the following moment control: there exists $C, k>0$ such that
\begin{align}
\|X_{k}(t)\|_{L^{p}(\O)}\leq Cp^{\frac{\l}{2}}\|X_{k}(t)\|_{L^{2}(\O)}, \label{hyperweak}
\end{align}
for any $p\geq 2$, for every $k\in \N$ and for each $t\in \R_{+}$. In particular, stochastic processes belonging to $\mathcal{H}_{\leq \l}$ for some $\l\in \N$ satisfy \eqref{hyperweak} because of the Wiener chaos estimate (Lemma~\ref{Lemma:wienerchaos}). Thus by \eqref{almostgauss}, $z^{\o}$ satisfies \eqref{hyperweak} and then applying the same arguments as in the proof of Proposition~\ref{prop:stochobjects} for the Gaussian random linear solution, for any $T>0$ we obtain
\begin{align*}
z\in C([0,T];W^{\al-\frac 12-,\infty}(\T)) 
\end{align*}
almost surely. Moreover, if $\{ \rho_{k}\}_{k\in \N}$ is a family of mollifiers on $\T$, then we have
\begin{align}
z_{k}=S(t)(u_0^{\o}\ast \rho_{k})\rightarrow z \label{znongas}
\end{align}
 as $\e\rightarrow 0$ in $L^{q}(\O; C([0,T];W^{\al-\frac 12-,\infty}(\T))$, for any $1\leq q<\infty$ and almost surely in $C([0,T];W^{\al-\frac 12-,\infty}(\T)$. Furthermore, the limit is independent of the mollification kernel $\rho$.

For the stochastic object $\NN(z)$, it is not at all obvious if \eqref{hyperweak} is satisfied (we may not even have a version of Wick's theorem). However, it will suffice for our purposes to show the fourth moment of $\|\NN(z)\|_{L^{4}_{T}H^{2\al-}_{x}}$ is finite; see Lemma~\ref{LEM:NZ} below. In this case, we argue directly using the following lemma.
This lemma appears in \cite[Lemma 4.3]{desuzzoni3} up to considering fourth-order moments. In this appendix, we require knowledge up to the eighth-order moments. Given $n\in \N$, we denote by $\{ \cj{n}\}$ the set $\{  k\in \N\, : k\leq n\}$.

\begin{lemma} \label{LEM:8}
Let $\{g_{n}\}_{n\in \Z}$ be a family of complex-valued random variables satisfying assumptions \eqref{item1} through \eqref{item5} and \textup{(vii)} above. Given $n\in \N$, we denote by $S^{n}$ the set of permutations of~  $\{\cj{n}\}$. Then, we have $\E[g_{n_1}g_{n_2}]=\dl_{n_1,-n_2}$ and
\begin{align*}
 \E\bigg[ \prod_{j=1}^{k}g_{n_k} \bigg]=0 \,\,\,\, \textup{for every odd } k\in \{ \cj{8}\}.
\end{align*}
Furthermore, we have
\begin{align*}
\E & \bigg[ \prod_{j=1}^{4} g_{n_j}  \bigg]  \\
& = 
\begin{cases}
\E[|g_n|^{4}] & \textup{if}\,\, \exists  \, \s\in S^4\,\, \textup{such that}\,\, n_{\s(1)}=n_{\s(3)}=-n_{\s(2)}=-n_{\s(4)} \\
\E[|g_n|^2]^{2}  &    \textup{if}\,\, \exists  \, \s\in S^4\,\, \textup{such that}\,\, n_{\s(1)}=-n_{\s(3)},\,
          \,\, n_{\s(2)}=-n_{\s(4)}  \\
          & \textup{and} \,\, |n_{\s(1)}|\neq |n_{\s(2)}|, \\
 0 &  \textup{if}\,\, \exists  \, \s\in S^4\,\, \textup{such that}\,\, |n_{\s(j)}|\neq |n_{\s(j')}| \, \textup{for every} \, j\neq j',\\
 & \textup{where}\,\, j,j'\in \{\cj{4}\}
\end{cases}
\end{align*}
and
\begin{align*}
\E & \bigg[ \prod_{j=1}^{8} g_{n_j}  \bigg]  \\
& = 
\begin{cases}
\E[|g_n|^{8}] & \textup{if}\,\, \exists  \, \s\in S^8\,\, \textup{such that}\,\, n_{\s(1)}=n_{\s(3)}=n_{\s(5)}=n_{\s(7)}\\
& \hphantom{XXXXXXX}     =-n_{\s(2)}=-n_{\s(4)}=-n_{\s(6)}=-n_{\s(8)},   \\
& \\
\E[|g_n|^6]\, \E[|g_n|^2]  &    \textup{if}\,\, \exists  \, \s\in S^8\,\, \textup{such that}\,\, n_{\s(1)}=n_{\s(3)}=n_{\s(5)}=-n_{\s(2)}\\
& =-n_{\s(4)}=-n_{\s(6)},  
           \, n_{\s(7)}=-n_{\s(8)} \,\,\textup{and} \,\, |n_{\s(1)}|\neq |n_{\s(7)}|, \\
        & \\
\E[|g_n|^4]^2   &    \textup{if}\,\, \exists  \, \s\in S^8\,\, \textup{such that}\,\, n_{\s(1)}=n_{\s(3)}=-n_{\s(2)}=-n_{\s(4)}, \\
&  \, n_{\s(5)}=n_{\s(7)}=-n_{\s(6)}=-n_{\s(8)} \,\, \textup{and} \,\, |n_{\s(1)}|\neq |n_{\s(5)}|,    \\
& \\
\E[|g_n|^2]^4   &  \textup{if}\,\, \exists  \, \s\in S^8\,\, \textup{such that}\,\,\textup{for each odd}\, j\in \{ \cj{8}\}, \,  \textup{we have}      \\
       &\, n_{\s(j)}=-n_{\s(j+1)} \, \textup{and} \,\, |n_{\s(j)}|=|n_{\s(j')}| \, \textup{for} \, j\neq j' \,  \\
       & \textup{and} \, j,j'\in \{ \cj{8}\} \, \textup{odd}\\
       & \\
0 &  \textup{if}\,\, \exists  \, \s\in S^8\,\, \textup{such that}\,\, |n_{\s(j)}|\neq |n_{\s(j')}| \, \textup{for every} \, j\neq j' ,\\
& \textup{where}\,\, j,j'\in \{\cj{8}\}.
\end{cases}
\end{align*}

\end{lemma}

\begin{proof}
The proof follows by a long case-by-case analysis using the  assumptions \eqref{item1}, \eqref{item2}, \eqref{item3} and the following consequence of assumption~\eqref{item5}:
for any non-negative integers $k$ and $\l$ satisfying $k+\l\leq 8$, we have
\begin{align}
\E[ g_{n}^{k}\cj{g_{n}}^{\l}]=\E[ |g_{n}|^{2k}]\dl_{k,\l}. \label{gkl}
\end{align}
To observe \eqref{gkl}, we may assume $k<\l$. Then by assumption~\eqref{item5}, we have 
\begin{align*}
\E[ g_n^k \cj{g_n}^{\l}]=\E[ |g_n|^{2k} \cj{g_n}^{\l-k} ] = e^{i\ta(k-\l)}\E[ |g_n|^{2k} \cj{g_n}^{\l-k}],
\end{align*}
but now the second equality and assumption~\eqref{item5} imply $\E[ |g_n|^{2k} \cj{g_n}^{\l-k}]=0$.
\end{proof}

\begin{lemma}\label{LEM:NZ}
Given $\al\in \big(\tfrac 14, \tfrac 12]$, let $s<2\al$ and fix $T>0$. Then, there exists $C_{s,\al}>0$ such that
\begin{align}
\E [ \| \NN(z) \|_{L^{4}_{T}H^{s}_{x}}^{4} ] \leq C_{s,\al}T<\infty. \label{quad}
\end{align}
Moreover, if $\{ \rho_{k}\}_{k\in \N}$ is a family of mollifiers on $\T$, then
$\NN(z_{k})$ converges to $\NN(z)$ as $k\rightarrow \infty$  in $L^{4}(\O; L^{4}_{T}H^{s}_{x})$. In particular, $\NN(z_{M})$ converges to $\NN(z)$ as $M\rightarrow \infty$, for $M$ dyadic, almost surely in  $L^{4}_{T}H^{s}_{x}$.
\end{lemma}

\begin{proof}
Using \eqref{nonlin} and \eqref{newdata}, we write 
\begin{align*}
\jb{\dx}^{s}\NN(z)=\sum_{n_1,n_2\in \Z} \mathcal{B}(n_1,n_2)g_{n_1}g_{n_2},
\end{align*}
where 
\begin{align*}
\mathcal{B}(n_1,n_2)&:=\ind_{\{ n_1+n_2\neq 0\}} \jb{n_1+n_2}^{s}\vp(n_1+n_2)e^{i(n_1+n_2)x} a(n_1)a(n_2),\\
a(n)&:= \frac{e^{-it\vp(n)}}{\jb{n}^{\al}}.
\end{align*}
Notice $\mathcal{B}(n_1,n_2)=\mathcal{B}(n_2,n_1)$.
With this notation, we have 
\begin{align}
\E[ \|\NN(z)\|_{L^{4}_{T}H^{s}_{x}}^{4}] & \leq  \bigg\|  \Big\| \sum_{n_1,n_2\in \Z} \mathcal{B}(n_1,n_2)g_{n_1}g_{n_2} \Big\|_{L^{4}(\O)} \bigg\|_{L^{4}_{T}L^{2}_{x}}^{4}. \label{b1}
\end{align}
 Clearly, if we show
\begin{align*}
 \bigg\| \sum_{n_1,n_2\in \Z} \mathcal{B}(n_1,n_2)g_{n_1}g_{n_2} \bigg\|_{L^{4}(\O)} \leq C<\infty, 
\end{align*}
where $C$ above is independent of $(x,t)\in \T \times \R_{+}$, \eqref{b1} will imply \eqref{quad}. 
By expanding, we have 
\begin{align}
\begin{split}
\bigg\| \sum_{n_1,n_2\in \Z} &\mathcal{B}(n_1,n_2)g_{n_1}g_{n_2} \bigg\|_{L^{4}(\O)}^{4} \\
& = \sum_{  \substack{n_1,n_2, k_1, k_2\in \Z \\ m_1, m_2, \l_1, \l_2\in \Z} }  \mathcal{B}(n_1,n_2)\cj{\mathcal{B}(m_1,m_2)}\mathcal{B}(k_1,k_2)\cj{\mathcal{B}(\l_1,\l_2)} \\
& \hphantom{XXXXXXXX}\times \E[ g_{n_1}g_{n_2}\cj{g_{m_1}}\cj{g_{m_2}}g_{k_1}g_{k_2}\cj{g_{\l_1}}\cj{g_{\l_2}}] 
\end{split}
 \label{big8}
\end{align}
We now use Lemma~\ref{LEM:8} to handle the expectation above. This naturally requires a case-by-case analysis. We first fix some terminology. 
We say we have a pair if there exist $j\in \{n_1,n_2,k_1,k_2\}$ and $j'\in \{m_1,m_2,\l_1,\l_2\}$ such that $j=j'$.
Let $j_1, j_2\in \{n_1,n_2,k_1,k_2\}$ be distinct and $j_3,j_4\in \{m_1,m_2,\l_1,\l_2\}$ be distinct.
We say we have a $2$-pair if, in fact, $j_1=j_2=j_3=j_4$. Similarly, we also define $3$-pairs and $4$-pairs in the obvious way. Hence, Lemma~\ref{LEM:8} implies the right hand side of \eqref{big8} is non-zero if we have:
\medskip 

\noi 
$\bullet$ \textbf{Case 1:} a $4$-pair 

In this case, we have $n_1=n_2=k_1=k_2=m_1=m_2=\l_1=\l_2$.
Hence,
\begin{align*}
\text{RHS of} \, \,\eqref{big8} & \sim \sum_{n} |\mathcal{B}(n,n)|^{4} \sim \sum_{n} \frac{1}{ \jb{n}^{8\al-4s+4}},
\end{align*}
which is summable provided $s<\tfrac 34 +2\al$.

\medskip

\noi
$\bullet$ \textbf{Case 2:} a $3$-pair and a pair

By the symmetry in $\mathcal{B}$, we may assume
\begin{align*}
k_1=k_2=n_1=m_1=m_2=\l_1 \,\,\,\text{and} \,\, n_{2}=\l_{2},
\end{align*}
but $n_1\neq n_2$.
Since $s<1$ and using Lemma~\ref{lemma:sumestimate}, we have 
\begin{align*}
\text{RHS of} \,\, \eqref{big8} & \sim \sum_{n_1} |\B(n_1,n_1)|^{2} \sum_{n_2} |\B(n_1,n_2)|^{2} \\
& \les \sum_{n_1}|\B(n_1,n_1)|^{2} \frac{1}{\jb{n_1}^{2\al}} \sum_{n_2} \frac{1}{\jb{n_1+n_2}^{2(1-s)}}\frac{1}{\jb{n_2}^{2\al}} \\
& \les \sum_{n_1}|\B(n_1,n_1)|^{2}\frac{1}{\jb{n_1}^{2\al}}\frac{1}{\jb{n_1}^{2(1-s)+2\al-1}}  \\
& \les \sum_{n_1} \frac{1}{\jb{n_1}^{3-4s+8\al}}<\infty
\end{align*}
provided $s<\tfrac 12 +\al$. 

\medskip

\noi
$\bullet$ \textbf{Case 3:} a $2$-pair and a $2$-pair

Again by the symmetry in $\B$, we have two further subcases. 

\medskip

\noi
$\bullet$ \textbf{Subcase 3.1:}  $n_1=m_1=n_2=m_2$ and $k_1=k_2=\l_1=\l_2$

We have
\begin{align*}
\text{RHS of} \,\, \eqref{big8} &\sim \bigg(\sum_{n_1} |\B(n_1,n_1)|^{2}\bigg)^{2} \sim \bigg(\sum_{n_1} \frac{1}{\jb{n_1}^{4\al}}\frac{1}{\jb{n_1}^{2(1-s)}}\bigg)^{2} <\infty,
\end{align*}
provided $s<\tfrac 12 +2\al$.

\medskip

\noi
$\bullet$ \textbf{Subcase 3.2:} $n_1=m_1=k_1=\l_1$ and $n_2=m_2=k_2=\l_2$

Using Lemma~\ref{lemma:sumestimate}, we have
\begin{align*}
\text{RHS of} \,\, \eqref{big8} & \sim \sum_{n_1,n_2}|\B(n_1,n_2)|^{4} \\
& \sim \sum_{n} \frac{1}{\jb{n}^{4-4s}} \sum_{n=n_1+n_2} \frac{1}{\jb{n_1}^{4\al}\jb{n_2}^{4\al}} \\
& \les \sum_{n} \frac{1}{\jb{n}^{4-4s+4\al}}<\infty
\end{align*}
provided $s<\tfrac 34 +\al$.

\medskip

\noi
$\bullet$ \textbf{Case 4:} four pairs

We reduce to three further subcases.

\medskip

\noi
$\bullet$ \textbf{Subcase 4.1:} $n_1=n_2$, $k_1=k_2$, $m_1=m_2$, $\l_1=\l_2$

Using Lemma~\ref{lemma:sumestimate}, we get 
\begin{align*}
\text{RHS of} \,\, \eqref{big8} &\sim \bigg( \sum_{n_1}|\B(n_1,n_1)| \bigg)^{4}  \sim \bigg( \sum_{n_1}\frac{1}{\jb{n_1}^{1-s}}\frac{1}{\jb{n_1}^{2\al}}  \bigg)^{4}<\infty,
\end{align*}
provided $s<2\al$.

\medskip

\noi
$\bullet$ \textbf{Subcase 4.2:} $n_1=m_1$, $n_2=m_2$, $k_1=\l_1$, $k_2=\l_2$

Using Lemma~\ref{lemma:sumestimate}, we have 
\begin{align*}
\text{RHS of} \,\, \eqref{big8} &\sim \bigg( \sum_{n_1,n_2}|\B(n_1,n_2)|^{2} \bigg)^{2} \\
& \les \bigg( \sum_{n} \frac{1}{\jb{n}^{2(1-s)}} \sum_{n=n_1+n_2} \frac{1}{\jb{n_1}^{2\al}}\frac{1}{\jb{n_2}^{2\al}} \bigg)^{2} \\
& \les \bigg( \sum_{n}\frac{1}{\jb{n}^{2(1-s)}}\frac{1}{\jb{n}^{4\al-1}}\bigg)^{2} <\infty,
\end{align*}
provided $\al>\tfrac 14$ and $s<2\al$.

\medskip

\noi
$\bullet$ \textbf{Subcase 4.3:} $n_1=m_1$, $n_2=\l_2$, $k_1=m_2$, $k_2=\l_1$  

In this case, we have 
\begin{align*}
\text{RHS of} \,\, \eqref{big8} \sim \sum_{n_1,n_2}&\frac{1}{\jb{n_1}^{2\al}\jb{n_2}^{2\al}\jb{n_1+n_2}^{1-s}} \sum_{k_2} \frac{1}{\jb{k_2}^{2\al}}\frac{1}{\jb{k_2+n_2}^{1-s}} \\
&\times \sum_{k_1} \frac{1}{\jb{k_1}^{2\al}\jb{k_1+k_2}^{1-s}\jb{k_1+n_1}^{1-s}}.
\end{align*}
For the innermost summation, we apply Cauchy-Schwarz and Lemma~\ref{lemma:sumestimate} twice to get 
\begin{align*}
 \sum_{k_1} \frac{1}{\jb{k_1}^{2\al}\jb{k_1+k_2}^{1-s}\jb{k_1+n_1}^{1-s}} \les \frac{1}{\jb{k_2}^{\frac 12 -s+\al}}\frac{1}{\jb{n_1}^{\frac 12-s+\al}},
\end{align*}
provided $s<\tfrac 12+\al$. Inserting this bound back into the above and using Lemma~\ref{lemma:sumestimate} to sum in $k_2$ provided $s<\tfrac 14 +\tfrac 32 \al$, we have 
\begin{align*}
\sum_{n_1,n_2}&\frac{1}{\jb{n_1}^{\frac{1}{2}-s+3\al}\jb{n_2}^{2\al}\jb{n_1+n_2}^{1-s}} \sum_{k_2} \frac{1}{\jb{k_2}^{\frac 12 -s+3\al}}\frac{1}{\jb{k_2+n_2}^{1-s}} \\
& \les \sum_{n_1} \frac{1}{\jb{n_1}^{\frac{1}{2}-s+3\al} }\sum_{n_2} \frac{1}{\jb{n_2}^{\frac 12-2s+5\al}}\frac{1}{\jb{n_1+n_2}^{1-s}}.
\end{align*}
Using Lemma~\ref{lemma:sumestimate} to sum in $n_2$ provided $s<\tfrac 16 +\tfrac 53 \al$, we bound the above by 
\begin{align*}
\sum_{n_1} \frac{1}{\jb{n_1}^{\frac{1}{2}-s+3\al} }\frac{1}{\jb{n_1}^{\frac 12 -3s+5\al}} =\sum_{n_1} \frac{1}{\jb{n_1}^{8\al+1-4s}}<\infty
\end{align*}
as long as $s<2\al$. 

Collating all the cases, we see that for $\tfrac 14 <\al\leq \tfrac 12$, the worst regularity restriction is indeed $s<2\al$. 

By slightly modifying the above arguments and using \eqref{rhomvt} and the uniform (in $k$ and $n$) of $\ft{\rho_{k\hphantom{'}}}(n)$, we also obtain
\begin{align}
\E[ \| \NN(z_{k})-\NN(z_{k'}) \|_{L^{4}_{T}H^{s}_{x}}^{4} ] \les C_{T,s,\al}k^{-4\ta} \label{Bcauchy}
\end{align}
for some small $\ta>0$. 
From \eqref{Bcauchy}, we have $\NN(z_{k})$ converges to $\NN(z)$ as $k\rightarrow \infty$  in $L^{4}(\O; L^{4}_{T}H^{s}_{x})$. In particular, taking $k'\rightarrow 0$ in \eqref{Bcauchy}, we get
\begin{align}
\E[ \|\NN(z_{M})-\NN(z)\|_{L^{4}_{T}H^{s}_{x}}^{4}] \les M^{-4\ta}, \label{BM}
\end{align}
where $M\geq 1$ is dyadic. Thus by a Borel-Cantelli argument with \eqref{BM}, we have $\NN(z_{M})$ converges to $\NN(z)$ as $M\rightarrow \infty$, for $M$ dyadic, almost surely in  $L^{4}_{T}H^{s}_{x}$.
\end{proof}

At this stage, we do not know how to show the almost sure convergence of $\NN(z_{k})$ along $k\in \N$. If we did have convergence of the full sequence, we would then obtain analogues of the results in Theorem~\ref{thm:asnorminf} and Theorem~\ref{thm:smoothapp} for random initial data of the form \eqref{newdata}. We note that this does not cause an issue for the global well-posedness argument in (iii) below as we only require the almost sure convergence of a subsequence (which we take to be dyadic).

\medskip

\noi 
\underline{$\bullet$ \textbf{(II):}}
 The key difference here is we weaken the assumption \eqref{assumptions} to $z_{2}\in L^{2}_{t}H^{s}_{x}([0,T]\times \T)$. Then, for the fixed point argument in the proof of Proposition~\ref{prop:detlocal}, we apply Cauchy-Schwarz to bound the $z_2$ term as follows:
\begin{align*}
\bigg\| \int_{0}^{t} S(t-t')z_{2}(t') dt'\bigg\|_{L^{\infty}_{T}H^{s}_{x}} & \les T^{\frac 12} \|z_{2}\|_{L^{2}_{T}H^{s}_{x}}.
\end{align*}
Hence, for the analogue of the almost sure local well-posedness result of Theorem~\ref{Thm:ASLWP} for random initial data of the form~\eqref{newdata}, we put $z_{2}=\NN(z)$ and use Lemma~\ref{LEM:NZ}. This implies almost sure existence of solutions below $L^{2}(\T)$ for BBM~\eqref{BBMD} with random initial data of the form~\eqref{newdata}, provided $\al>\tfrac 14$. 

\medskip

\noi 
\underline{$\bullet$ \textbf{(III):}}
We now describe the analogue of the almost sure global well-posedness result of Theorem~\ref{Thm:ASGWP}. 
All that is necessary is to obtain the following analogue of Proposition~\ref{prop:bdk}, except here we now replace the smoothed initial value problem~\eqref{vepseq2} by the smoothed initial value problem with dyadic $M\geq 1$:
\begin{align}
\begin{cases}
i\dt v_M = \vp(D_x) \big(v_M + \frac 12 v_{M}^2 + z_{M} v_M  \big)+\frac{1}{2} \NN(z_M) \\
v_M |_{t=0}=0,
\end{cases}
\label{vepseq3}
\end{align}

\begin{proposition}\label{prop:bd} Let $\al =\frac 12$ and $s<1$ sufficiently close to one. Given $T, \e>0$, there exist $\tilde{\O}_{T,\e}\subset \O$ such that 
\begin{align*}
\prob( (\tilde{\O}_{T,\e})^{c}) <\e,
\end{align*}
a dyadic integer $M_{0}=M_{0}(T,\e)$ and a finite constant $C(T,\e)>0$ such that the following bound holds: 
\begin{align*}
\sup_{ \substack{M \geq M_0 \\ M \textup{dyadic}}} \sup_{t\in [0,T]}\|v_{M}(t)\|_{H^{s}(\T)} \leq C(T,\e),
\end{align*}
for every solution $v_{M}^{\o}$ to \eqref{vepseq3} with $\o \in \tilde{\O}_{T,\ep}$.
\end{proposition}

\begin{proof}

With $T>0$ fixed, we define 
\begin{align*}
\Sigma_{\text{conv},T}=\{ \o\in \O\,:\,(z_{M}^{\o},\NN(z_{M}^{\o}))\rightarrow (z^{\o},\NN(z^{\o})) \,\,\text{in} \,\,  C_{2T}&W^{\al-\frac 12-, r}\times L^{4}_{2T}H^{s} \\
& \text{as}\, M\rightarrow \infty, \,M \,\text{dyadic} \},
\end{align*}
where $r=r(s,\al)$ is as in Subsection~\ref{subsec:bdk}. Note that \eqref{znongas} and Lemma~\ref{LEM:NZ} imply $\prob(\Sigma_{\text{conv},T})=1$. With $K>0$ fixed, we define 
\begin{align*}
\O_{K,T,\al}=\{ \o\in \Sigma_{\text{conv},T}\, :\, \|z\|_{C_{2T}W^{\al-\frac 12-,p}_{x}}+\|\NN(z)\|_{L^{4}_{2T}H^{s}_{x}}\leq K\}.
\end{align*}
By Egoroff's theorem, for any $\e>0$, there exists $\O_{\e}\subset \Sigma_{\text{conv},T}$ with $\prob( \Sigma_{\text{conv},T} \setminus \O_{\e})<\tfrac{\e}{3}$, such that $\NN(z_{M}^{\o})$ converges uniformly to $\NN(z^{\o})$ as $M\rightarrow \infty$ in $L^{4}_{2T}H^{s}_{x}$ for every $\o\in \O_{\e}$.
Hence, there exists $M_0= M_0(T,\e)$ such that for every $M\geq M_0$, we have
\begin{align*}
\| \NN(z^{\o}_{M})\|_{L^{4}_{2T}H^{s}_{x}}\leq 1+K 
\end{align*}
for every $\o\in \O_{K,T,\al, \e}:=\O_{\e}\cap \O_{K,T,\al}$. Lemma~\ref{LEM:NZ} implies
\begin{align*}
\prob( \O_{K,T,\al, \e}^{c}) \leq \frac{C_{T,s,\al}}{K^{4}}+\frac{\e}{3},
\end{align*}
and hence with $K=K(\e,T,s,\al)$ large enough, we have $\prob( \O_{K,T,\al,\e}^{c})<\tfrac{2\e}{3}$.

We now obtain an analogue of \eqref{energy} for the growth of the modified energy $E(Iv_{M})(t)=:E_{M}(t)$. 
We first note that the result of Lemma~\ref{lemma:izintmoment} holds for the initial data~\eqref{newdata} in view of \eqref{almostgauss}.
The only modification we make is in estimating the term (II):
\begin{align*}
\int_{0}^{t}\int_{\T}(\dx Iv_M) I(\P_{\neq 0}(z_{M}^{2})) dx dt'.
\end{align*}
Namely, by Cauchy-Schwarz and H\"{o}lder's inequalities, we have 
\begin{align*}
\bigg\vert \int_{0}^{t}\int_{\T}(\dx Iv_M)& I(\P_{\neq 0}(z_{M}^{2})) dx dt' \bigg\vert  \\
& \leq \int_{0}^{t} \|I(\P_{\neq 0}(z_{M}^{2}))\|_{L^{2}_{x}}E^{\frac{1}{2}}(t')\,dt' \\
& \leq \|I(\P_{\neq 0}(z_{M}^{2}))\|_{L^{4}_{T}L^{2}_{x}}\bigg( \int_{0}^{t} E^{\frac{2}{3}}(t')\,dt'\bigg)^{\frac{3}{4}}\\
& \les N^{1-2\al+}\|\P_{\neq 0}(z_{M}^{2})\|_{L^{4}_{T}H^{s-1-}_{x}}\bigg[ 1+ \int_{0}^{t} E^{\frac{2}{3}}(t')\,dt'\bigg].
\end{align*}
Then, applying the same ideas as in the proof of Proposition~\ref{prop:bdk} we obtain an inequality of the form \eqref{cgron} (the analogue of \eqref{ode}) with $\g=\tfrac 23$ and $c\sim N^{1-2\al+}K$. We then apply Lemma~\ref{gronwall} and we complete the argument as in the proof of Proposition~\ref{prop:bdk} provided $\al=\tfrac{1}{2}$, with $\tilde{\O}_{T,\e}:=\tilde{\O}_{T,\e}:=\O_{K(\e,T),T,\frac 12 ,\e}\cap \O_{\Lambda,N}$ and $\O_{\Lambda,N}$ defined in \eqref{Olambda}.

\end{proof}

\section{Tail estimates on random variables}\label{app:Garsia}

In this appendix, we state some standard results which allows us to prove tail estimates of random variables using estimates on the moments of differences as in, for example, \eqref{253}. Our reference for this appendix is~\cite[Appendix A.2 and A.3]{FV}.

\begin{theorem}[Garsia-Rudemich-Rumsey inequality,~{\cite[Theorem A.1]{FV}}] 
Let $(E,d)$ be a metric space and $f\in C([0,T];E)$. Let $\Psi$ and $P$ be continuous strictly increasing functions on $[0,\infty)$ with $P(0)=\Psi(0)=0$ and $\Psi(x)\rightarrow \infty$ as $x\rightarrow \infty$. Suppose 
\begin{align*}
\int_{0}^{T} \int_0^T \Psi\bigg(\frac{d(f(t'),f(t))}{P(|t-s|)}  \bigg)dt'dt \leq F. 
\end{align*}
Then for any $0\leq s<t\leq T$, we have 
\begin{align*}
d(f(t'),f(t))\leq 8 \int_{0}^{t-t'}\Psi^{-1}\bigg( \frac{4F}{x^{2}}\bigg) dP(x).
\end{align*}
\end{theorem}

Putting $\Psi(x)=x^{q}$ and $P(x)=x^{\beta+\frac 1q}$ for $\beta>\frac 1q$, we obtain the following useful corollary.

\begin{corollary}
Suppose $q>1$ and $\beta >\frac 1q$.
Then for $0\leq t'<t\leq T$, we have
\begin{align*}
d(f(t'),f(t))^{q}\leq C(\beta,q)^{q}|t-t'|^{q\beta-1}\iint_{[t',t]}\frac{d(f(u),f(v))^{q}}{|u-v|^{q\beta+1}}dudv,
\end{align*}
where 
\begin{align*}
C(\beta,q):=32^{q}\bigg(\frac{q\beta+1 }{q\beta-1}\bigg)^{q}.
\end{align*}
In particular, with $\g =\beta-\frac 1q$, we have
\begin{align}
\begin{split}
\|f\|_{\dot{C}^{\g}([0,T];E)}&:=\sup_{0\leq t'<t\leq T}\frac{ d(f(t),f(t'))}{|t-t'|^{\g}} \\
& \leq C(\beta,q)^{\frac 1q} \bigg( \int_{0}^{T}\int_{0}^{T} \frac{d(f(u),f(v))^{q}}{|u-v|^{q\beta+1}}dudv   \bigg)^{\frac 1q}.
\end{split} \label{holder}
\end{align}
\end{corollary}

We use \eqref{holder} to help estimate the moments of $C^{\g}([0,T];E)$-norms of random variables. In applying Kolmogorov's continuity criterion, one shows a difference estimate like 
\begin{align}
\E [ d(f(t),f(t'))^{q}] \leq K(\eta,q)|t-t'|^{1+\eta}, \label{kolm1}
\end{align}
where $\eta>0$. 

We then conclude from \eqref{kolm1} that the process $f$ belongs to $C^{\frac{\eta}{q}-\g}([0,T];E)$ for any $\g <\frac{\eta}{q}$ almost surely. Using \eqref{kolm1} in \eqref{holder}, we have 
\begin{align}
\E[ \|f\|_{\dot{C}^{\g}([0,T];E)}^{q}] &\leq K(\eta,q)C(\beta,q)\int_{0}^{T}\int_{0}^{T} \frac{|u-v|^{1+\eta}}{|u-v|^{q\beta+1}}dudv \notag \\
& = K(\eta,q)C(\beta,q)\frac{2T^{2+\eta-q\beta}}{(1+\eta-q\beta)(2+\eta-q\beta)}, \label{kolm2}
\end{align}
provided $q\beta -\eta<1$.
Thus we require:
\begin{align*}
\beta \in \bigg( \frac{1}{q}, \frac{\eta}{q}+\frac{1}{q}\bigg) \qquad \text{and} \qquad \g<\frac{\eta}{q}.
\end{align*}
We then use \eqref{kolm2} to obtain appropriate tail estimates. 

As an example, in \eqref{zdiff} we obtained
\begin{align*}
\E[ \|z(t)-z(t')\|_{W^{s_1,p}(\T)}^{q}] \les q^{\frac q2}|t-t'|^{q}
\end{align*}
 for $q\geq 2$.
Thus from \eqref{kolm2} we get 
\begin{align*}
\E[ \|z\|_{\dot{C}^{\g}([0,T];W^{s_1,p}(\T))}^{q}]\leq \frac{2C(\beta,q)q^{\frac q2}}{q(1-\beta)[1+q(1-\beta)]}T^{q(1-\beta)+1},
\end{align*}
provided $\g<1-\frac 1q$ and $\beta \in \big(\tfrac{1}{q}, 1\big)$. To convert this into a tail estimate on $ \|z\|_{\dot{C}^{\g}([0,T]; W^{s_1,p}(\T))}$, Chebyshev's inequality implies 
\begin{align}
\prob( \|z\|_{\dot{C}^{\g}([0,T]; W^{s_1,p}(\T))}>\ld) \leq \frac{D(\beta,q)}{\ld^{q}}T^{q(1-\beta)+1}, \label{kolm3}
\end{align}
where  
\begin{align*}
D(\beta,q)&=\frac{32^{q}\big(\frac{q\beta+1}{q\beta-1}\big)^{q}q^{\frac q2}}{q(1-\beta)[1+q(1-\beta)]}  \leq \frac{32^{q}\big(\frac{q\beta+1}{q\beta-1}\big)^{q}q^{\frac q2}}{q(1-\beta)}
\end{align*}
We notice that since $q\geq 2$, we can choose $\beta=\tfrac{1}{4}+\tfrac{1}{q}$, which allows for the loose bound  
\begin{align*}
D(\beta,q)\leq (Cq)^{\frac{q}{2}}.
\end{align*}
Now, we optimise in $q$ the bound on the right hand side \eqref{kolm3} and we obtain the exponential tail estimate 
\begin{align}
\prob( \|z\|_{\dot{C}^{\g}([0,T]; W^{s_1,p}(\T))}>\ld) \leq \exp(-c\ld^{2}T^{-\frac{3}{4}}). \label{exptail}
\end{align}

\end{appendix}

\begin{ackno}\rm
J.\,F. is grateful to his PhD advisors Tadahiro Oh and Oana Pocovnicu for suggesting this topic and for their guidance and support. The author would like to thank Nikolay Tzvetkov and Leonardo Tolomeo for helpful discussions and encouragement. The author would also like to thank Martin Hairer for asking a question which initiated Remark~\ref{remark: gaussianity}.
J.\,F.~was supported by The Maxwell Institute Graduate School in Analysis and its
Applications, a Centre for Doctoral Training funded by the UK Engineering and Physical
Sciences Research Council (grant EP/L016508/01), the Scottish Funding Council, Heriot-Watt
University and the University of Edinburgh.
J.\,F. also acknowledges support from Tadahiro Oh's ERC starting grant no. 637995 “ProbDynDispEq”.
\end{ackno}


\begin{thebibliography}{99}




\bibitem{Model1}
A.\,A.~Alazman, J.\,P.~Albert, J.\,L.~Bona, M.~Chen, J.~Wu, 
{\it Comparisons between the BBM equation and a Boussinesq system,}
 Adv. Differential Equations 11 (2006) 121--166.


\bibitem{BBM}
T.\,B.~Benjamin, J.\,L.~Bona, J.\,J.~Mahony,
 {\it Model equations for long waves in nonlinear dispersive systems,} 
Phil. Trans. R. Soc. Lond. A (1972) 272 47--78.


\bibitem{BOP3}
 \'{A}.~B\'{e}nyi, T.~Oh, O.~Pocovnicu,
  {\it Higher order expansions of the probabilistic 
Cauchy theory of the cubic nonlinear Sch\"{o}dinger equation on $\R^{3}$,} 	
 Trans. Amer. Math. Soc. Ser. B 6 (2019), 114--160.


\bibitem{BOP4}
 \'{A}.~B\'{e}nyi, T.~Oh, O.~Pocovnicu,
  {\it On the probabilistic Cauchy theory for nonlinear dispersive PDEs,} 	
Landscapes of Time-Frequency Analysis. 1--32, Appl. Numer. Harmon. Anal., Birkhäuser/Springer, Cham, 2019.


\bibitem{BCS1}
J.\,L.~Bona, M.~Chen, J.-C.~Saut, 
{\it Boussinesq equations and other systems for small- amplitude long waves in nonlinear dispersive media. I. Derivation and linear theory}, J. Non- linear Sci. 12 (2002), no. 4, 283--318.

\bibitem{BCS2}
J.\,L.~Bona, M.~Chen, J.-C.~Saut, 
{\it Boussinesq equations and other systems for small-amplitude long waves in nonlinear dispersive media. II. The nonlinear theory}, Nonlinearity 17 (2004), no. 3, 925--952.

\bibitem{Model3}
J.\,L.~Bona, T.~Colin, D.~Lannes, 
{\it Long wave approximations for water waves,}
 Arch. Ration. Mech. Anal. 178 (2005), 373--410.

\bibitem{BonaNormInf}
 J.\,L.~Bona, M.~Dai,
 {\it Norm Inflation for the BBM equation,}
  J. Math. Anal. Appl. 446 (2016), 879--885.
  
\bibitem{Model2}
J.\,L.~Bona, W.\,G.~Pritchard, L.\,R.~Scott, 
{\it An evaluation of a model equation for water waves},
 Philos. Trans. R. Soc. Lond.
Ser. A Math. Phys. Eng. Sci. 302 (1981), 457--510.

\bibitem{BonaTzvet}
J. L.~Bona, N.~Tzvetkov, 
 {\it Sharp well-posedness results for the BBM equation,}
  Discrete Contin. Dyn. Syst.  (2007), no. 23, 1241--1252.

\bibitem{Bourgain2}
 J.~Bourgain, 
 {\it Periodic nonlinear Schr\"{o}dinger equation and invariant measures,}
  Comm. Math. Phys. 166 (1994), no. 1, 1--26.

\bibitem{Bourgain1}
 J.~Bourgain, 
 {\it Invariant measures for the 2D-defocusing nonlinear Schr\"{o}dinger equation,}
  Comm. Math. Phys. 176 (1996), no. 2, 421--445.


 \bibitem{BTlocal}
  N.~Burq, N.~Tzvetkov, 
 {\it Random data Cauchy theory for supercritical wave equations, I: Local theory,} Invent. Math. 173 (2008), no. 3, 449–475.
 
 \bibitem{BTglobal}
 N.~Burq, N.~Tzvetkov, 
 {\it Random data Cauchy theory for supercritical wave equations. II. A global existence
result,} Invent. Math. 173 (2008), no. 3, 477--496.
 

\bibitem{CP}
A.~Choffrut, O.~Pocovnicu,
{\it Ill-posedness of the cubic nonlinear half-wave equation and other fractional NLS on the real line,} 
Int. Math. Res. Not. IMRN 2018, no. 3, 699--738.

 \bibitem{Christ1}
 M.~Christ, 
 {\it Power series solution of a nonlinear Schr\"odinger equation,} Mathematical aspects of nonlinear
 dispersive equations, 131--155, Ann. of Math. Stud., 163, Princeton Univ. Press, Princeton, NJ,
 2007.  
 
 
 \bibitem{Iteam1}
 J.~Colliander, M.~Keel, G.~Staffilani, H.~Takaoka, T.~Tao,
 {\it Global well-posedness for
Schr\"{o}dinger equations with derivative}, SIAM J. Math. Anal., 33 (2001), pp. 649--669.
 
 \bibitem{Iteam2}
  J.~Colliander, M.~Keel, G.~Staffilani, H.~Takaoka, T.~Tao, 
  {\it Sharp global well-posedness for
KdV and Modified KdV on $\R$ and $\T$ }, J. Amer. Math. Soc. 16 (2003), no. 3, 705--749.
 
\bibitem{CollOh}
 J.~Colliander, T.~Oh,
 {\it Almost sure well-posedness of the cubic nonlinear Schr\"{o}dinger equation below $L^{2}(\T)$,}
  Duke Math. J. 161 (2012), no. 3 , 367--414.

\bibitem{dpd}
 G.~Da Prato, A.~Debussche, 
 {\it Two-dimensional Navier-Stokes equations driven by a space-time white
noise}, J. Funct. Anal. 196 (2002), no. 1, 180--210.

\bibitem{desuzzoni1}
A.-S.~de Suzzoni,
{\it Continuity of the flow of the Benjamin-Bona-Mahony equation on probability measures,}
Discrete Contin. Dyn. Syst. 35 (2015), no. 7, 2905--2920.

\bibitem{desuzzoni2}
A.-S.~de Suzzoni,
{\it Wave turbulence for the BBM equation: stability of a Gaussian statistics under the flow of BBM,}
Comm. Math. Phys. 326 (2014), no. 3, 773--813.

\bibitem{desuzzoni3}
A.-S.~de Suzzoni, N.~Tzvetkov,
{\it On the propagation of weakly nonlinear random dispersive waves, }
Arch. Ration. Mech. Anal. 212 (2014), no. 3, 849--874.

\bibitem{Dragomir}
S.\,S.~Dragomir, {\it Some Gronwall type inequalities and applications,} Nova
Science Publishers, Inc., Hauppauge, NY, 2003.

\bibitem{FH}
P.\,K.~Friz, M.~Hairer, 
{\it A course on rough paths. 
With an introduction to regularity structures}. Universitext. Springer, Cham, 2014. xiv+251 pp.

\bibitem{FV}
P.\,K.~Friz, N.~Victoir, 
{\it Multidimensional stochastic processes as rough paths. 
Theory and applications. }
Cambridge Studies in Advanced Mathematics, 120. Cambridge University Press, Cambridge, 2010. xiv+656 pp.

\bibitem{GTV}
J.~Ginibre, Y.~Tsutsumi, G.~Velo,
{\it On the Cauchy problem for the Zahkharov system,}
J. Funct. Anal. 151 (1997), no. 2, 384--436.

\bibitem{GKO}
 M.~Gubinelli, H.~Koch, T.~Oh, {\it Renormalization of the two-dimensional stochastic nonlinear wave equations,}
 Trans. Amer. Math. Soc. 370 (2018), no 10, 7335--7359.

\bibitem{GKOT}
 M.~Gubinelli, H.~Koch, T.~Oh, L.~Tolomeo,
  {\it Global dynamics for the two-dimensional stochastic nonlinear wave equations,}
 preprint.
 
 \bibitem{GPburger}
 M.~Gubinelli, N.~Perkowski, {\it Probabilistic Approach to the Stochastic Burgers Equation}, In: Eberle A., Grothaus M., Hoh W., Kassmann M., Stannat W., Trutnau G. (eds) Stochastic Partial Differential Equations and Related Fields. SPDERF 2016. Springer Proceedings in Mathematics \& Statistics, vol 229. Springer, Cham.
 
 \bibitem{GPnotes}
 M.~Gubinelli, N.~Perkowski,
{\it Lectures on singular stochastic PDEs, }
Ensaios Matemáticos [Mathematical Surveys], 29. Sociedade Brasileira de Matemática, Rio de Janeiro, 2015. 89 pp.
 
\bibitem{HairerLec}
 M.~Hairer,
 {\it An Introduction to Stochastic PDEs,}
  http://www.hairer.org/notes/SPDEs.pdf, 2009.

\bibitem{IO}
T.~Iwabuchi, T.~Ogawa,
{\it Ill-posedness for the nonlinear Schr\"odinger equation with quadratic non-linearity in low dimensions,}
Trans. Amer. Math. Soc. 367 (2015), no. 4, 2613--2630.

 \bibitem{Janson}
   S.~Janson, 
{\it Gaussian Hilbert spaces}, Cambridge University Press, 1997.

\bibitem{Kishimoto}
N.~Kishimoto,
{\it A remark on norm inflation for nonlinear Schrödinger equations},
Commun. Pure Appl. Anal. 18 (2019) 1375--1402.
 
  
  \bibitem{mckean}
 H.\,P.~McKean, 
 {\it Statistical mechanics of nonlinear wave equations. IV. Cubic Schr\"odinger}, Comm. Math.
Phys. 168 (1995), no. 3, 479--491. {\it Erratum: Statistical mechanics of nonlinear wave equations. IV. Cubic
Schr\"odinger,} Comm. Math. Phys. 173 (1995), no. 3, 675.
  
\bibitem{Nelson}
   E.~Nelson,  
   {\it A quartic interaction in two dimensions}, 1966 Mathematical Theory of Elementary Particles
(Proc. Conf., Dedham, Mass., 1965) pp. 69--73 M.I.T. Press, Cambridge, Mass.
   
  \bibitem{HiroInflation} 
T.~Oh, 
{\it A remark on norm inflation with general initial data for the cubic nonlinear Schr\"odinger equations in negative Sobolev spaces},
 Funkcial. Ekvac. 60 (2017) 259--277.
  
\bibitem{HiroSZEGOKDV}
 T.~Oh,
 {\it Remarks on nonlinear smoothing under randomization for the periodic KdV and the cubic Szeg\"{o} equation,}
   Funkcial. Ekvac. 54 (2011), no. 3, 335--365.
    
 \bibitem{OOT}
 T.~Oh, M.~Okamoto, N.~Tzvetkov,   
    {\it Uniqueness and non-uniqueness of the Gaussian free field evolution 
under the two-dimensional Wick ordered cubic wave equation,} 
preprint.
  
   \bibitem{OPT}
 T.~Oh, O.~Pocovnicu, N.~Tzvetkov, 
 {\it  Probabilistic local well-posedness of the cubic nonlinear wave equation in negative Sobolev spaces,}
	arXiv:1904.06792 [math.AP].


  \bibitem{Panthee}
  M.~Panthee,
 {\it  On the ill-posedness result for the BBM equation,} Discrete Contin. Dyn. Syst. 30 (2011), 253--259. 
  
   \bibitem{Peregrine}
 D.\,H.~Peregrine,
 {\it  Calculations of the development of an undular bore,} J. Fluid Mech. 25 (1966), 321--330. 


\bibitem{Roumegoux}
D.~Roum\'egoux,
{\it A symplectic non-squeezing theorem for BBM equation}, 
Dyn. Partial Differ.
Equ., 7 (4) (2010), 289-–305.

	
 \bibitem{ThomTzvet}
 L.~Thomann, N.~Tzvetkov,  
 {\it Gibbs measure for the periodic derivative nonlinear Schr\"{o}dinger
equation,} 
Nonlinearity 23, no. 11 (2010), 2771--2791.

\bibitem{Leonardo}
 L.~Tolomeo,
  {\it Global well-posedness of the two-dimensional
stochastic nonlinear wave equation on an unbounded domain,}
 preprint.
 
   \bibitem{Tzv1}
N.~Tzvetkov, 
 {\it Random data wave equations}, 
  	arXiv:1704.01191 [math.AP].

	
\bibitem{Tzv3}
N.~Tzvetkov,
{\it Quasi-invariant Gaussian measures for one dimensional Hamiltonian PDE’s,}
Forum
Math. Sigma 3 (2015), e28, 35 pp.

 \bibitem{WangBBM}
 M.~Wang,
 {\it Sharp global well-posedness of the BBM equation in $L^{p}$ type Sobolev spaces,}
  Discrete Contin. Dyn. Sys. 36 (2016), no. 10, 5763--5788.

  \bibitem{Xia}
B.~Xia, 
 {\it Generic ill-posedness for wave equation of power type on 3D torus}, 
  	arXiv:1507.07179 [math.AP].








\end{thebibliography}
\end{document}